\documentclass[a4paper,11pt]{amsart}

\pdfoutput=1

\usepackage[text={440pt,660pt},centering]{geometry}
\usepackage{amsthm, amssymb, amsmath, amsfonts, mathrsfs}
\usepackage{mathtools}
\usepackage{scalerel} 
\usepackage{stackrel}
\usepackage[colorlinks=true, pdfstartview=FitV, linkcolor=blue, citecolor=blue, urlcolor=blue,pagebackref=false]{hyperref}
\numberwithin{equation}{section}
\usepackage[export]{adjustbox} 
\usepackage{caption}
\usepackage{subcaption}
\usepackage{enumitem}
\usepackage{multicol} 

\usepackage{tikz} 
\usetikzlibrary{shapes,arrows,positioning,calc}

\tikzstyle{result} = [rectangle, 
minimum width=4cm, 
minimum height=1cm,  
text width=4cm, 
align=flush center,
font = \footnotesize, 
draw=black]

\tikzstyle{input} = [rectangle, 
minimum width=3cm, 
minimum height=1cm, 
text width=3cm, 
align=flush center,
font = \footnotesize,
draw=black]

\tikzstyle{arrow} = [thick, ->,>=stealth]

\newtheorem{thm}{Theorem}[section]
\newtheorem{remark}[thm]{Remark}
\newtheorem{lemma}[thm]{Lemma}
\newtheorem{discussion}[thm]{Discussion}
\newtheorem{prop}[thm]{Proposition}
\newtheorem{definition}[thm]{Definition}
\newtheorem{corollary}[thm]{Corollary}

\definecolor{darkgreen}{rgb}{0,0.5,0}
\definecolor{darkblue}{rgb}{0,0,0.7}
\definecolor{darkred}{rgb}{0.9,0.1,0.1}

\newcommand{\gucomment}[1]{\marginpar{\raggedright\scriptsize{\textcolor{darkred}{#1}}}}

\newcommand{\zhaocomment}[1]{\marginpar{\raggedright\scriptsize{\textcolor{darkblue}{#1}}}}

\newcommand{\cCf}{\cC_{\infty}} 
\newcommand{\cCm}{\cC_{*}} 
\newcommand{\cCb}{\cC_{\operatorname{bc}}} 
\newcommand{\cCp}{\cC_{**}} 
\newcommand{\cCa}{\cC_{\infty}(\a)}
\newcommand{\cCam}{\cC^*_{\a}(\cu_m)}
\newcommand{\cCaGm}{\cC^*_{\a^{G}}(\cu_m)}
\newcommand{\OO}{\mathscr{O}} 
\newcommand{\EE}{\mathcal{E}} 

\renewcommand{\leq}{\leqslant}
\renewcommand{\geq}{\geqslant}

\renewcommand{\subset}{\subseteq}
\renewcommand{\supset}{\supseteq}
\newcommand{\mcl}{\mathcal}

\newcommand{\aA}{{\mathcal{A}}}
\newcommand{\G}{\mathcal{G}}
\newcommand{\E}{\mathbb{E}}
\newcommand{\eE}{\mathcal{E}}
\newcommand{\fE}{\mathfrak{E}}

\renewcommand{\a}{\mathbf{a}}

\newcommand{\atilde}{\tilde{\mathbf{a}}}

\newcommand{\ab}{{\overbracket[1pt][-1pt]{\a}}}

\newcommand{\adm}{\a^{\delta}_{m}}
\newcommand{\ai}{\tilde{\a}^\delta}
\newcommand{\aim}{\tilde{\a}^\delta_{m}}
\newcommand{\aii}[1]{\tilde{\a}^{\delta,(#1)}_{m}}
\newcommand{\aio}[1]{{\a}^{#1}}

\newcommand{\na}{\nabla}
\newcommand{\Ll}{\left}
\newcommand{\Rr}{\right}
\newcommand{\lhs}{left-hand side}
\newcommand{\rhs}{right-hand side}

\newcommand{\lL}{\mathcal{L}}
\newcommand{\R}{\mathbb{R}}

\newcommand{\cC}{\mathscr{C}}
\newcommand{\N}{\mathbb{N}}

\newcommand{\ZZ}{\mathbb{Z}}
\newcommand{\Zd}{{\mathbb{Z}^d}}
\newcommand{\Ed}{E_d}
\newcommand{\Edo}{\overrightarrow{E_d}}
\newcommand{\Eda}{E_d^{\a}}
\newcommand{\itr}{\mathrm{int}}
\renewcommand{\P}{\mathbb{P}}
\newcommand{\Pp}{\mathbb{P}_\p}
\newcommand{\pP}{\mathcal{P}}
\newcommand{\pPN}{\mathcal{P}^{(N,k)}}

\newcommand{\kK}{\mathcal{K}}

\newcommand{\rR}{\mathcal{R}}
\newcommand{\mM}{\mathcal{M}}
\newcommand{\tT}{\mathcal{T}}

\newcommand{\sSl}{\mathcal{S}^{loc}}
\newcommand{\gG}{\mathcal{G}}
\newcommand{\gGN}{\mathcal{G}^{(N)}}
\newcommand{\wW}{\mathcal{W}}
\newcommand{\fF}{\mathcal{F}}

\renewcommand{\bar}{\overline}

\renewcommand{\tilde}{\widetilde}
\newcommand{\eps}{\varepsilon}

\renewcommand{\epsilon}{\varepsilon}

\newcommand{\cu}{{\scaleobj{1}{\square}}}

\newcommand{\p}{\mathfrak{p}}

\newcommand{\oO}{\mathcal{O}}
\newcommand{\name}{{\textcolor{blue}{weighted-invasion-percolation}}}
\newcommand{\Ind}[1]{\mathbf{1}_{\left\{#1\right\}}}
\newcommand{\id}{\mathsf{Id}}
\newcommand{\norm}[1]{\left\Vert{#1}\right\Vert}
\newcommand{\bracket}[1]{\left\langle{#1}\right\rangle}
\newcommand{\de}{\hat{\delta}} 
\newcommand{\rr}{\mathsf{r}} 
\renewcommand{\aa}{\mathsf{c}} 
\newcommand{\stable}{\mathsf{stable}} 
\newcommand{\W}{\mathbf{W}}
\newcommand{\aL}{\underline{L}}
\newcommand{\g}{\mathbf{g}}
\newcommand{\w}{\mathbf{w}}

\DeclareMathOperator{\dist}{dist}
\DeclareMathOperator{\size}{size}
\DeclareMathOperator{\supp}{supp}
\DeclareMathOperator{\diam}{diam}
\DeclareMathOperator*{\osc}{osc}

\title[Differentiability of diffusivity on percolation]{The diffusivity of supercritical Bernoulli percolation is infinitely differentiable}
\author{Chenlin Gu, Wenhao Zhao}

\address[Chenlin Gu]{Yau Mathematical Sciences Center, Tsinghua University, Beijing, China}
\email{gclmath@tsinghua.edu.cn}

\address[Wenhao Zhao]{
	EPFL, Lausanne, Switzerland \& School of Mathematical Sciences, Peking University, Beijing, China}
\email{wenhao.zhao@epfl.ch}

\begin{document}
	\begin{abstract}
		We prove that, the diffusivity and conductivity on $\Zd$-Bernoulli percolation ($d \geq 2$) are infinitely differentiable in supercritical regime. This extends a result by Kozlov [\emph{Uspekhi Mat. Nauk} 44 (1989), no. 2(266), pp 79–120]. The key to the proof is a uniform estimate for the finite-volume approximation of derivatives, which relies on the perturbed corrector equations  in homogenization theory. The renormalization of geometry is then implemented in a sequence of scales to gain sufficient degrees of regularity. To handle the higher-order perturbation on percolation, new techniques, including \emph{cluster-growth decomposition} and \emph{hole separation}, are developed.
		
		\bigskip
		
		
		\medskip
		

	\end{abstract}
	\maketitle
	
	\setcounter{tocdepth}{1}
	\tableofcontents

	\begin{figure}[h!]
		\centering
		\includegraphics[width=0.3\textwidth]{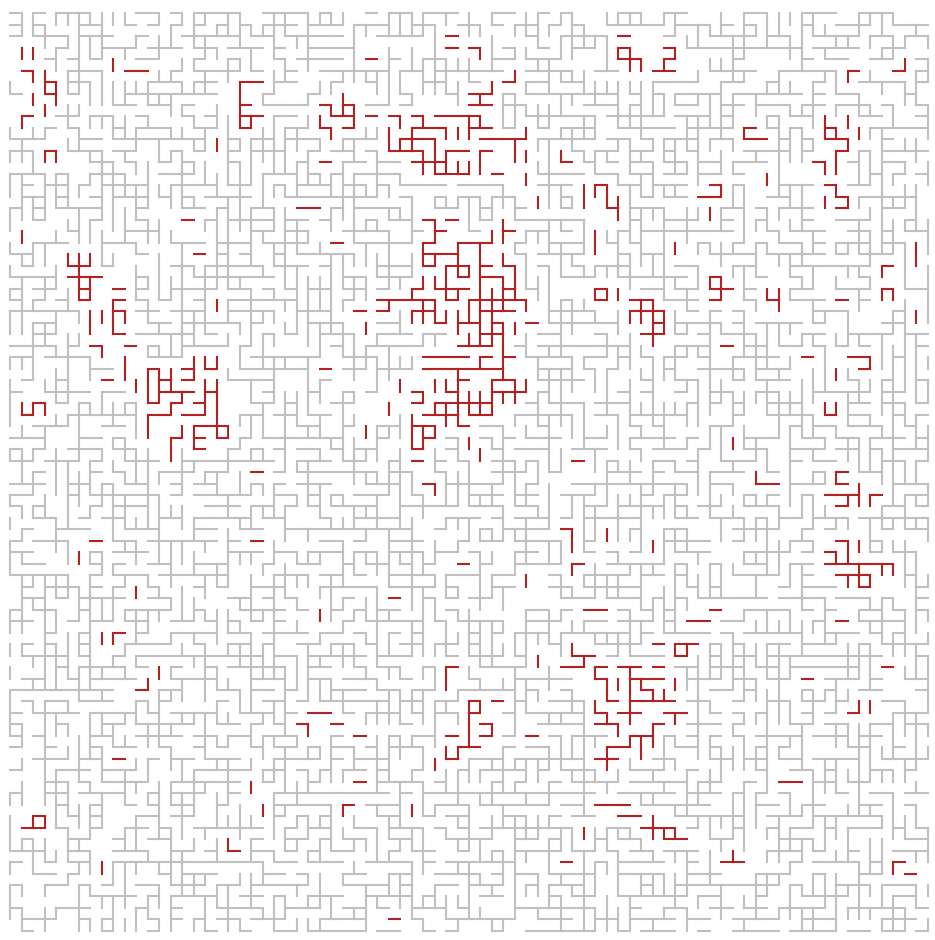}
		\includegraphics[width=0.3\textwidth]{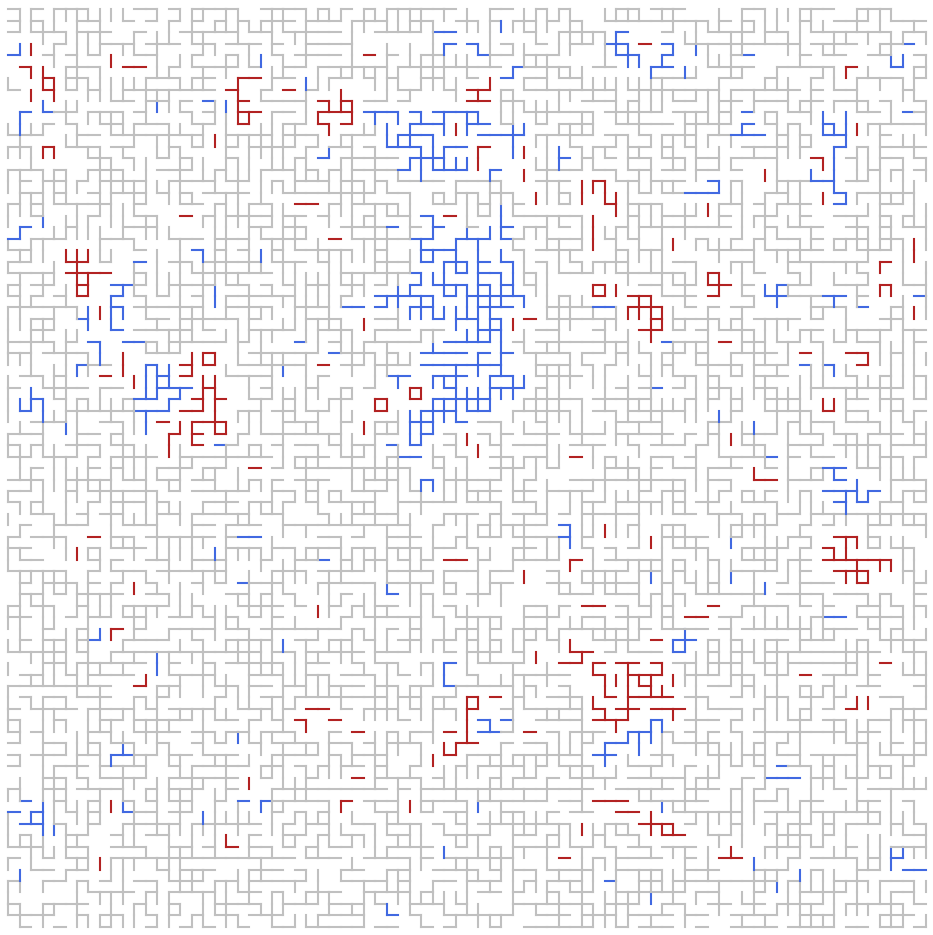}
		\includegraphics[width=0.3\textwidth]{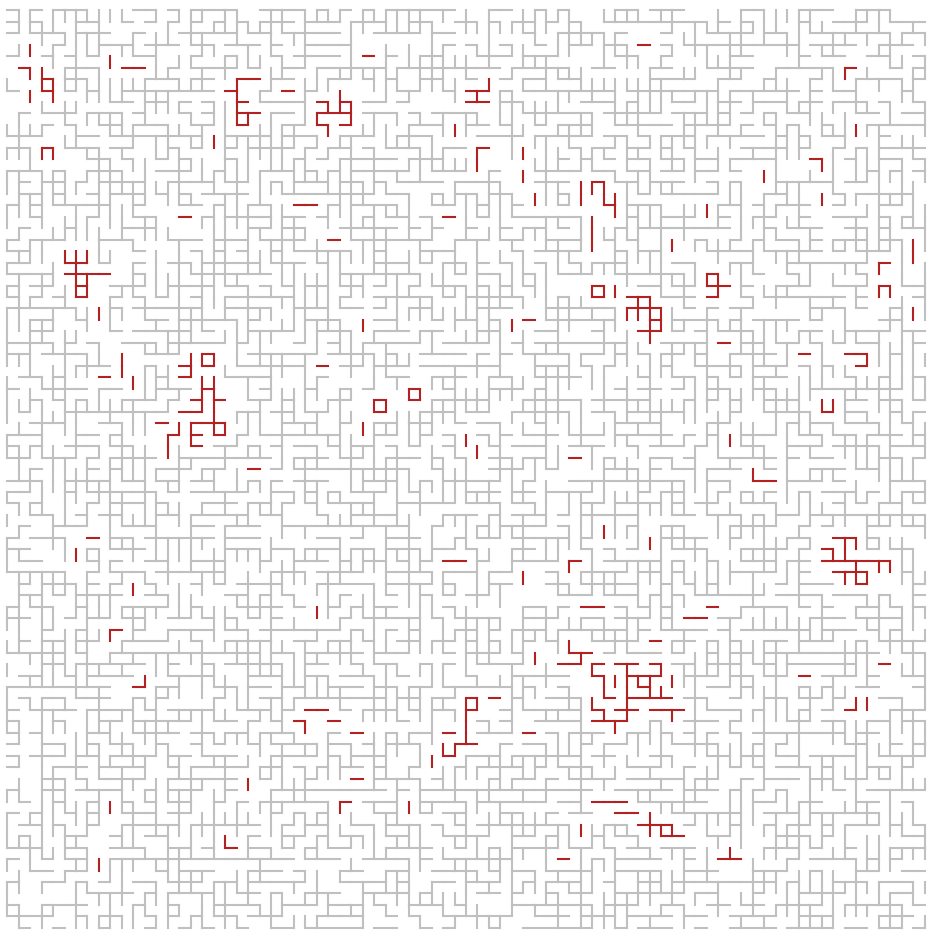}
		\caption{\scriptsize{Figures on the left and right are coupled supercritical percolation of parameter $\p$ and $\p+\delta$, with isolated clusters in red and boundary-connecting clusters in gray. In the figure in the middle, the clusters in blue are the growth of the boundary-connecting cluster in this passage.}}
		\label{fig.cover}
	\end{figure}

	\newpage

	\section{Introduction}

	\subsection{Main result} 
	The Bernoulli percolation was introduced by Broadbent and Hammersley \cite{broadbent1957percolation} in order to study the flow in porous material. Let $\Zd$ be the $d$-dimensional lattice and $\Ed$ be the associated edges/bonds which connect the nearest vertices.
	We sample every bond independently with probability $\p \in [0,1]$ to be \emph{open}, and with probability $(1-\p)$ to be \emph{closed}. That is a family of i.i.d. Bernoulli random variables $\{\a(e)\}_{e \in \Ed}$ satisfying 
	\begin{align}\label{eq.def_Percolation}
		\Pp[\a(e) = 1] = \p, \qquad \Pp[\a(e) = 0] = 1 - \p.
	\end{align}
	The open edges form connected components called \emph{clusters}. The $\Zd$-Bernoulli percolation is the random graph defined by ${(\Zd, \{e\in\Ed:\a(e)=1\})}$. Despite its simple description, this model captures the phase transition in statistical physics:  for $d \geq 2$, there exists a critical point $\p_c(d) \in (0,1)$, such that there exists almost surely a unique infinite cluster, denoted by $\cCf$, in \emph{the supercritical phase} $\p > \p_c$ (see \cite{aizenman1987uniqueness, burton1989density}), and no infinite cluster in \emph{the subcritical phase} $\p < \p_c$. The connected probability is defined as
	\begin{align}\label{def.theta}
		{\theta(\p):=\Pp[0\in \cCf]}.
	\end{align}
	This quantity is zero in subcritical phase and positive in supercritical phase. It was proved by $\theta(\p_c) = 0$ for $d = 2$ by Kesten in \cite{kesten1980critical}, and for $d \geq 19$ by Hara and Slade in \cite{hara1990mean}, and for $d \geq 11$ by Fitzner and van der Hofstad in \cite{fitzner2017mean}, but it still remains as a major open problem for $3 \leq d \leq 10$. We refer to the monograph by Grimmett \cite{grimmett1999} and the proceeding by Duminil-Copin \cite{duminil2017sixty} for a detailed introduction of this topic.

	$\Zd$-Bernoulli percolation in supercritical regime is usually believed better understood. As indicated in the work \cite{pisztora1996surface} by Pisztora and in \cite{kesten1990} by Kesten and Zhang, the infinite cluster has a geometry close to $\Zd$ when zooming out, and the other finite clusters are generally of small size. For these reasons, the random walk on supercritical percolation should be close to Brownian motion in large scale. This observation, known as \emph{``the ant in the labyrinth"}, was mentioned by de Gennes in \cite{de1976percolation}. For the rigorous results, the Gaussian bound was proved by Barlow in \cite{barlow2004random}; the quenched invariance principle was obtained at first by Sidoravicius and Sznitman in \cite{sidoravicius2004quenched} for $d \geq 4$, then generalized by Berger and Biskup in \cite{berger2007quenched}, and by Mathieu and Piatnitski in \cite{mathieu2007quenched} for any dimension $d \geq 2$. These results have also been extended to other settings (including the stationary ergodic environment, the unbounded conductance, the correlated percolation, stable process, etc) in a series of work \cite{ABDH,andres2015invariance,  BD, hambly2009parabolic, procaccia2016quenched, sapozhnikov2017random, BCKW2021, CCKW2025, CKW2021}; see also the survey by Biskup \cite{biskup2011recent} and a very recent survey by Andres  \cite{andres2025homogenization}. 
	
	We state rigorously the quenched invariance principle for \emph{the variable-speed random walk} $\left( X_t \right)_{t \geq 0}$. It is a  continuous time Markov-jump process associated to the generator 
	$\nabla \cdot \a \nabla$
	\begin{equation} \label{eq.VSRW}
		\forall x \in \cCf, \qquad \nabla \cdot \a \nabla u(x) := \sum_{z: z \sim x} \a(\{x,z\}) \left( u(z) - u(x) \right).
	\end{equation}
	Here we use the notation $z \sim x$ if $\{x,z\} \in \Ed$. 	In another word, every open edge is associated to an independent exponential clock of parameter $1$, and the random walk will cross the neighbor open edge whose clock rings first. The quenched invariance principle states that, there exists a deterministic diffusive constant $\sigma(\p) > 0$ such that, for $\P_\p$-almost every realization $\{\a(e)\}_{e \in \Ed}$, the following convergence in law holds under the Skorokhod topology
	\begin{equation} \label{eq.IP}
		\Ll(\frac{1}{\sqrt{N}} X_{N t}\Rr)_{t \geq 0} \stackrel{N \to \infty
		}{\Longrightarrow} ( \sigma B_{t})_{t \geq 0},
	\end{equation}
	where $(B_{t})_{t \geq 0}$ is the $d$-dimensional Brownian motion. 
	
	An important method to study the invariance principle on random conductance is the homogenization. This theory, appearing at first in PDE, studies the large-scale behavior of operators when assuming periodic or ergodic structure. Percolation is also an important example, and its study dates back to Zhikov in \cite{zhikov1989}. In the last decades, there is a lot of progress in the field of quantitative homogenization (see \cite{armstrong2016quantitative, armstrong2016lipschitz, armstrong2016mesoscopic, armstrong2017additive, NS, gloria2011optimal, gloria2012optimal, gloria2015quantification, GO3, armstrong2024renormalization} and the monographs \cite{AKMbook,armstrong2022elliptic}), and they have significantly contributed to the percolation theory. Armstrong and Dario established in \cite{armstrong2018elliptic} the complete Liouville theorem on supercritical percolation cluster, which extends the result in previous work \cite{benjamini2015disorder} by Benjamini, Duminil-Copin, Kozma, and 	Yadin. Later, the estimates of the first-order correctors and Green's function were obtained respectively in \cite{dario2021corrector, dario2021quantitative}. Using these results, the first author, together with Su and Xu, recently construct a good coupling for the invariance principle \eqref{eq.IP} in \cite{gu2024coupling}.

	The present work focuses on the regularity of $\sigma(\p)$ in function of $\p$. This question is motivated by the phase transition in statistical physics, since the fundamental quantities are usually believed to be smooth away from the threshold, and exhibit a power law decay at the critical point. Concerning the Bernoulli percolation, the smoothness of $\theta(\p)$ was studied by Russo in  \cite{russo1978note}, and by Chayes,  Chayes, Newman in \cite{chayes1987bernoulli}; see also  \cite[Chapter 8.7]{grimmett1999}. Its analyticity was confirmed by Georgakopoulos and Panagiotis  in \cite{georgakopoulos2018analyticity}. The critical exponent  of $\theta(\p)$, together with other quantities including the susceptibility, were discussed by Smirnov and Werner in \cite{smirnov2001} for planar percolation, and by Hara and Slade in \cite{hara1990mean} for high-dimensional percolation.  The regularity was also established for the free energy in Ising model, see \cite{lee1952statistical, messager1974analyticity, messager1976analyticity, ott2020weak}. Therefore, $\sigma(\p)$ is expected to illustrate similar properties, but very few rigorous results can be found in the literature. In \cite[Open Problem 14.1]{high2017},  Heydenreich and van der Hofstad mentioned the open problem about the monotonicity of $\sigma(\p)$ and  the critical exponent $\sigma^2(\p) \approx (\p - \p_c)^t$. The discussion about critical exponent can also be  found in \cite[Problem 4.21]{biskup2011recent} by Biskup, and in \cite[equation (45)]{Hughes2021} by Hughes.

	Our main theorem states the regularity of this diffusive constant in the supercritical phase. We believe it is the first step to understand the questions above. 
	
	\begin{thm}\label{thm.main}
		For $d \geq 2$, the mapping $\p \mapsto \sigma(\p)$ is infinitely differentiable in  $(\p_c, 1]$. 
	\end{thm}
	Here the derivative at $1$ is understood as the left derivative.  To the best of our knowledge, the only previous result in the same direction is \cite{kozlov1989} by Kozlov. On site-percolation setting, he justified the left derivative at $1$  in \cite[Theorem~9]{kozlov1989}. At the end of the introduction there, he also asked the question about the differentiability on the whole interval  (\emph{``$\cdots$the differential properties of $a(p, 0)$ for $p > 0$ remain unclear$\cdots$"}). See also \cite[(9.25)]{jikov1994}. Theorem~\ref{thm.main} answers his question in supercritical regime.

	We have two remarks about our main theorem. Firstly, we focus on VSRW model, while \emph{the constant speed random walk } (CSRW) is also a natural random walk on the percolation. It is well-known that the diffusivity in two random walks only differ by a factor $2d \E_\p[\a(\{0,1\}) \vert 0 \in \cCf]$; see \cite[Theorem~1.1, Section~6.2]{ABDH}. The smoothness of this factor in function of $\p$ can be justified following the same argument in \cite[Chapter 8.7]{grimmett1999}. Therefore, Theorem~\ref{thm.main} is also valid for CSRW.
	
	Secondly, as indicated in \cite[Theorem~1.1]{mathieu2007quenched} and \cite[Theorem~1.1]{berger2007quenched}, the limit Brownian motion in \eqref{eq.IP} is isotropic. This property is not utilized in the paper. Our proof is robust, and can also be adapted to the inhomogeneous Bernoulli percolation model \cite[Chapter~11.9]{grimmett1999}.
	
	
	\subsection{A sketch of proof}
	In this part, we highlight the ingredients of the proof, together with some auxiliary results and useful techniques developed in the paper.
	
	
	\subsubsection{Outline}
	
	We start by recalling Einstein's conductivity-diffusivity formula (see \cite[eq.(181)]{dario2021quantitative})
	\begin{align}\label{eq.Einstein}
		\frac{\sigma^2(\p) \id}{2} = \frac{\ab(\p)}{\theta(\p)}.
	\end{align}
	Here $\ab(\p)$ is the effective conductivity, which is a symmetric $\R^{d \times d}$ matrix. Thanks to the previous work, in all dimensions $d \geq 2$, the connected probability $\theta(\p)$   was proved infinitely differentiable in supercritical regime \cite{russo1978note, georgakopoulos2018analyticity}. Thus, it suffices to study the regularity of $\ab(\p)$, and its proof can be explained in three steps.
	\begin{enumerate}
		\item Find a finite-volume approximation of the quantity. 
		\item Calculate the derivative of the finite-volume approximation.
		\item Give a uniform control of the derivative with respect to the volume.
	\end{enumerate}
	Afterwards, a classical theorem of uniform convergence (see \cite[Theorem~7.17]{rudin1976}) ensures the exchange between the limit and derivative.

	Step~1 for $\ab(\p)$ relies on the homogenization theory. Especially, Armstrong and Dario proved a finite-volume approximation in \cite[Section~5]{armstrong2018elliptic}
	\begin{align*}
		\ab(\p)  =  \lim_{m \to \infty} \ab_m(\p), \\
	\end{align*}  
	where $\ab_m(\p)$ is a symmetric $\R^{d \times d}$  matrix defined by
	\begin{align}\label{eq.abm_intro}
		\xi \cdot \ab_m(\p) \xi :=  \E_{\p} \Ll[\inf_{v \in \xi \cdot x + C_0(\cu_m)}\frac{1}{\vert \cu_m \vert} \sum_{e \in \Ed(\cu_m)}  \a(e) \vert \nabla v(e)\vert^2\Rr].
	\end{align}
	Here $\cu_m = \Zd \cap \Ll(-\frac{3^m}{2}, \frac{3^m}{2}\Rr)^d$ is the finite lattice cube, the operator $\nabla$  denotes the discrete difference along the edge $e$, and $C_0(\cu_m)$ is the set of functions with value zero on the boundary.
	Since $\ab_m(\p)$ only depends on the percolation in $\cu_m$, the mapping $\p \mapsto \ab_m(\p)$ is a polynomial and Step~2 also works. The $k$-th derivative of $\ab_m(\p)$ provides a natural candidate to approximate that of $\ab(\p)$. We denote them respectively by $\ab_m^{(k)}(\p)$ and $\ab^{(k)}(\p)$. Actually,  we can further prove the following quantitative version of convergence. 
	
	\begin{thm}\label{thm.ab}
		Given $d \geq 2$ and $\p \in (\p_c, 1]$, the following statement is valid for the effective conductivity $\ab(\p)$: there exists an exponent $\alpha_*(d, \p) > 0$, such that for every $k \in \N$  and $\alpha \in (0, \alpha_*)$, there exists a finite positive constant $C(d,k, \p, \alpha)$ satisfying
		\begin{align}\label{eq.ck}
			\Ll\vert \ab_m^{(k)}(\p) - \ab^{(k)}(\p)\Rr\vert \leq C(d,k, \p, \alpha) 3^{-\alpha m}.
		\end{align}
	\end{thm}
	
	Theorem~\ref{thm.ab} and Step~3 mentioned in the beginning are related to the conductivity under perturbation. This direction has a long history, as Clausius (1879) and Mossotti (1850) have independently proposed an expansion of effective conductivity of two-phase media. A lot of rigorous work emerges recently. Clausius--Mossotti formula was proved by Almog in \cite{almog1,almog2,almog3}  under certain condition. Mourrat studied a first-order expansion under Bernoulli perturbation in \cite{dl-diff}, where the reference media is also random. Anantharaman and Le Bris proposed a second order expansion in \cite{AL1,AL2}. In the work \cite{dg1}, Duerinckx and Gloria  not only relaxed the condition for Clausius--Mossotti formula, but also generalized the expansion to arbitrarily high order. They further simplified the proof of Clausius--Mossotti formula in \cite{dg23}. The argument of higher-order expansion is robust, and was applied to infinite interacting particle systems in \cite{GGMN} and to derive Einstein’s effective viscosity formula in \cite{dg_Einstein}.

	\subsubsection{Improved $\ell^1$-$L^2$ energy estimate}
	Our proof of Theorem~\ref{thm.ab} is also inspired by the work \cite{dg1, GGMN, dg_Einstein}, and the heart is \emph{the improved $\ell^1$-$L^2$ energy estimate}: the derivatives of $\xi \cdot \ab_m(\p) \xi$ is bounded by a family of improved energy $\{I_{m, \xi}(i,j)\}_{i,j \in \N}$, so it suffices to prove their uniform bound  with respect to the domain $\cu_m$. The improved energy $I_{m, \xi}(i,j)$ is defined as  
	\begin{align}\label{eq.defV_intro_1}
		I_{m, \xi}(i,j) :=  \frac{1}{\vert \cu_m \vert} \sum_{\substack{\vert F \vert = i \\ F \subset \Ed(\cu_m )}} \sum_{e \in \Ed(\cu_m )} \vert \nabla V_{m,\xi}(F,j)(e)\vert^2, 
	\end{align}
	and the function $V_{m, \xi}(F, j) $ in the integral is defined as 
	\begin{align}\label{eq.defV_intro_2}
		V_{m, \xi}(F, j) := \sum_{\substack{|G|=j\\G\subset \Ed(\cu_m ) \setminus F}} D_{F\cup G} \Ll(v_{m,\xi}(\a)\Rr). 
	\end{align}
	Here $v_{m,\xi}$ is the minimiser of Dirichlet energy in \eqref{eq.abm_intro}, and $D_G$ is \emph{the Glauber derivative}. 	Roughly, $\xi \cdot \ab_m(\p) \xi$ can be expressed in function of $v_{m,\xi}$, and the perturbation of $\p$ is then transformed to the Glauber derivative over the environment. Precisely, $\a^G$ means setting all the bonds $e$ of $G$ open, and $D_G$ is defined as 
	\begin{align*}
		(D_G f)(\a) := \sum_{G' \subset G} (-1)^{\vert G \setminus  G'\vert} f(\a^{G'}).
	\end{align*}

	At the first glance, the integral $I_{m, \xi}(i,j)$ counts a huge sum of ${ \vert \cu_m\vert \choose i+j}$ terms, because $F$ and $G$ go over all possible subsets in \eqref{eq.defV_intro_1} and \eqref{eq.defV_intro_2}. The heuristic of uniform bound roots from the compensation from the $(i+j)$-th Glauber derivative contained in $V_{m, \xi}(F, j)$.
	
	A rigorous approach to study $I_{m, \xi}(i,j)$ is \emph{the perturbed corrector equation} below
	\begin{equation}\label{eq.recurrence_intro}
		-\na \cdot (\aio{F} \na V_{m, \xi}(F,j)) = \na \cdot \W_{m, \xi}(F,j) \qquad \text{ in } \cu_m,
	\end{equation}
	whose counterpart can be found in \cite[(4.1)]{dg1}, \cite[(3.15)]{dg_Einstein}, \cite[(5.27)]{GGMN}. Since $v_{m,\xi}$ is the minimiser of Dirichlet energy in \eqref{eq.abm_intro}, under the percolation environment $\a^F$, we have 
	\begin{align*}
		-\na \cdot (\aio{F} \na v_{m, \xi}(\a^F)) = 0 \qquad \text{ in } \cu_m.
	\end{align*} 
	Then an inclusion-exclusion type calculation will yield \eqref{eq.recurrence_intro}, where the {\lhs} is the term of leading order and $\W_{m, \xi}$ only contains terms of the lower order. Therefore, if we can dominate the Dirichlet energy of $V_{m, \xi}$ by that of $\W_{m, \xi}$, an induction will give a uniform bound for $I_{m, \xi}(i,j)$.

	Nevertheless, unlike the previous work \cite{dg1, GGMN}, equation \eqref{eq.recurrence_intro} is defined on percolation setting, so the uniform ellipticity is missing. All the information accessible on its {\lhs} is on $\a^F$, and the domination we obtain is actually
	\begin{align}\label{eq.VWbound_intro}
		I^{*}_{m}(i_0,j_0) \lesssim	\sum_{(i,j)\prec(i_0,j_0)}I_{m}(i,j).
	\end{align}
	Here $(i,j)\prec(i_0,j_0)$ represents a partial order relation. The notation ``$\lesssim$" stands for a domination up to a constant, which is independent of $m$. As $\xi \mapsto V_{m, \xi}(F,j)$ is linear, we set $\xi = 1$ and omit its dependence in notations. The quantity $I^{*}_{m}(i,j)$ is analogue to $I_{m}(i,j)$, but supported on the percolation cluster  
	\begin{align*}
		I^{*}_{m}(i,j) 
		:= \frac{1}{\vert \cu_m \vert}\sum_{\substack{|F|=i\\F\subset \Ed(\cu_m)}}\Ll( \sum_{e\in \EE(\cCf(\a^F) \cap \cu_m)} |\na V_{m}(F, j)(e)|^2 \Rr).
	\end{align*}
	Clearly, we cannot close the induction only using \eqref{eq.VWbound_intro}. The previous work \cite{dg_Einstein} covers some cases in continuum percolation model, but it requires a uniform separation assumption \cite[Page~21]{dg_Einstein}, which does no hold in the entire supercritical regime of the $\Zd$-Bernoulli percolation. Therefore, we discuss the new inputs in the following paragraphs.

	\subsubsection{New inputs}
	In order to iterate \eqref{eq.VWbound_intro}, a natural attempt is to justify $I_{m}(i,j) \approx	I^{*}_{m}(i,j)$.  This is the idea of coarse-graining, and the work \cite{armstrong2018elliptic} illustrates a basic principle.
	\begin{enumerate}
		\item \emph{Extension from the grain}: set the infinite cluster as the grains, and extend the function value from the grains to their neighborhood.
		\item \emph{Renormalization of geometry}: ensure the cluster is locally well-connected, so the error in coarse-graining is small.
	\end{enumerate}
	The idea applies to the case $(i,j) = (0,0)$, where it suffices to study $v_{m}$. Then we will get $I_{m}(0,0) \approx	I^{*}_{m}(0,0)$ 
	and it confirms the basis of our induction. However, the cases $(i,j) \neq (0,0)$ contain more challenges, because the Glauber derivative even perturbs the geometry of the percolation. As $V_m(F,j)$ involves at least $2^{\vert F \vert} \times { \vert \cu_m \vert \choose j}$ different percolation configurations, its behaviors outside $\a^F$ is quite complicated, and ``the extension from the grain" above is broken; see its definition in \eqref{eq.defV_intro_2}. Therefore, we develop several techniques to enhance the coarse-graining argument.
	
	\medskip
	Concerning the case $j=0$, we only need to treat $V_m(F,0) = D_F v_{m}$. Our first observation is \emph{the cluster-growth decomposition}, which says $D_F$ can be divided into two independent parts.  Notice that, when the percolation grows from $\a$ to $\a^F$, the contribution of $F$ can be decomposed into two disjoint subset $F = F_* \sqcup F_\circ$:
	\begin{itemize}[label=---]
		\item $F_*$ enlarges the clusters connecting to the boundary $\partial \cu_m$.
		\item $F_\circ$ merges the isolate clusters disconnecting to the boundary $\partial \cu_m$.
	\end{itemize} 
	The function $v_m$ is the solution of a Dirichlet problem, so its value is determined by the clusters connecting to the boundary, and is only perturbed by $F_*$. Meanwhile, for the isolated clusters, if we choose the nearest grain from $\cC_{\infty}(\a)$, then the growth from $F_*$ does not provide more choice, and only $F_\circ$ can change the grain.  Therefore, we then discover the following decomposition for $D_F v_{m}$
	\begin{equation}\label{eq.decom_DF_intro}
		\forall x \in \cu_m, \qquad (D_F v_{m})(\a, x) = (\nabla_{F_\circ} D_{F_*} v_{m})(\a, x).
	\end{equation}
	Here $\nabla_{F_\circ}$ is a higher-order spatial finite difference operator (rigorously defined in \eqref{eq.def_diff_shift}). This decomposition further yields a coarse-graining involving the terms of lower order
	\begin{align}\label{eq.Im_Is_0_intro}
		I_{m}(i_0,0) \lesssim \sum_{i \leq i_0}	I^{*}_{m}(i_0,0).
	\end{align}
	Together with \eqref{eq.VWbound_intro}, we can close the induction for $\{I_m(i,0)\}_{i \geq 0}$.
	
	\medskip
	
	The general case $j \geq 1$ is more delicate, and we need a technique called \emph{the hole separation} as the second input. The index $j$ has a different structure from the index $i$, as \eqref{eq.defV_intro_1} and \eqref{eq.defV_intro_2} indicate. If we apply Cauchy--Schwarz inequality naively to $V_m(F,j)$, we can transfer $I_m(i,j)$ to ${\vert \cu_m \vert \choose j}I_m(i+j,0)$, where the additional factor explodes when $m \to +\infty$. We thus propose a more careful manipulation using a variant of $V_m(F, j)$
	\begin{equation}\label{eq.V_F_j_Generalized_intro}
		V_m(F, j \vert \cu) := \sum_{\substack{|G|=j\\G\subset (\Ed(\cu_m) \setminus F) \setminus \Ed(\cu) }} D_{F\cup G} v_m  .
	\end{equation}
	Roughly speaking, we separate the perturbation of $G$ from $\cu \subset \cu_m$, so its information there is more accessible. It also has its associated version of improved energy 
	\begin{align}\label{eq.defIalpha_hole_intro}
		\tilde{I}_{m}(i,j) := \frac{1}{\vert \cu_m \vert}\sum_{\substack{|F|=i\\F\subset \Ed(\cu_m)}}\Ll(\sum_{\substack{\cu \in \pP(\cu_m)}} \sum_{e\in \Ed( \cu )} \Ll|\na V_{m}(F, j \vert \cu)(e) \Rr|^2 \Rr).
	\end{align}
	Here $\pP(\cu_m)$ stands for a partition of cube, and the size of $\cu$ there is small. 
	
	The quantity $\tilde{I}_{m}$  serves as a link between $I_m$ and  $I^*_m$, because it captures the tiny fluctuation in the growth of clusters. By counting the number of perturbation on $\cu$, we have the following expansion
	\begin{align*}
		V_m(F, j) = \sum_{\ell=1}^j \sum_{\substack{G \subset \Ed(\cu) \\ \vert G \vert = \ell}} V_m(F \cup G,j-\ell \vert \cu).
	\end{align*}
	It then yields a domination of $I_{m}$ by $\tilde{I}_{m}$
	\begin{align}\label{eq.Im_It_intro}
		I_{m}(i_0,j_0) \lesssim	\tilde{I}_{m}(i_0,j_0) + \sum_{(i,j)\prec(i_0,j_0)}\tilde{I}_{m}(i,j).
	\end{align}
	Moreover, by the definition of $V_m(F, j \vert \cu)$ in \eqref{eq.V_F_j_Generalized_intro}, we have excluded the influence from the second index $j$, so all the perturbation  on $\cu$ comes from $F$. Thus, the cluster-growth decomposition \eqref{eq.decom_DF_intro} locally holds on $\cu$, which leads to a coarse-graining with lower-order terms similar as \eqref{eq.Im_Is_0_intro}
	\begin{align}\label{eq.It_Is_intro}
		\tilde{I}_{m}(i_0,j_0) \lesssim	I^{*}_{m}(i_0,j_0) + \sum_{(i,j)\prec(i_0,j_0)}\tilde{I}_{m}(i,j).
	\end{align}
	The equations  \eqref{eq.Im_It_intro}, \eqref{eq.It_Is_intro}, \eqref{eq.VWbound_intro} allow us to close an induction for $\{I_m(i,j) \}_{i,j \in \N}$. 
	
	\medskip
	
	The third ingredient is a \emph{pyramid partitions of good cubes}. We also need technical estimates in the induction and coarse-graining above.  Especially, in \eqref{eq.VWbound_intro},  $L^2$-estimate is not sufficient and we need $L^p$-Meyers' estimate. To treat a specific case with two sides of different supports in \eqref{eq.recurrence_intro}, we further develop a \emph{coarse-grained Meyers' estimate}. All these results require  the renormalization of geometry. Therefore, we develop the idea of partition in \cite[Section~2]{armstrong2018elliptic}, and build a sequence of good cubes  $\{\gG_{k}\}_{k \in \N}$ inductively
	\begin{align*}
		\gG_0 \supset \gG_1 \supset \gG_2 \supset \cdots \supset \gG_{k-1} \supset \gG_k \supset \cdots.
	\end{align*}
	The cubes in $\gG_0$ already have nice properties including Meyer's estimate, and the cubes in  $\gG_k$ are even better since it can be further partitioned into finer good cubes in $\gG_{k-1}$.  Therefore, $\gG_k $ allows us to implement \eqref{eq.VWbound_intro} for $k$ rounds, and control the $k$-th derivative.
	
	\smallskip
	
	We remark that, in order to present the main idea, the notations including $I_{m}(i,j),\gG_k $ and $V_m(F, j \vert \cu)$ are simplified in the discussion above. The complete version of the improved $\ell^1$-$L^2$ energy is $I^{\; (N,k)}_{\; m,\alpha, h}(i,j)$ defined in \eqref{eq.defIalpha}, which is explained quickly Fig~\ref{fig.energy}. In the procedure of induction, the indices $(i,j)$ decrease according to the partial order ``$\prec$", while the moment $\alpha$ and the scale of partition $h$ will increase. We need the percolation configuration in $\cu_m$ to be connected and robust enough, as the parameters $k, N$ respectively indicate, to finish the induction of uniform bound. This is the case in large scale when sending $m \to +\infty$ and fixing other parameters.   
	\begin{figure}[b]
		\includegraphics[width=0.5\textwidth]{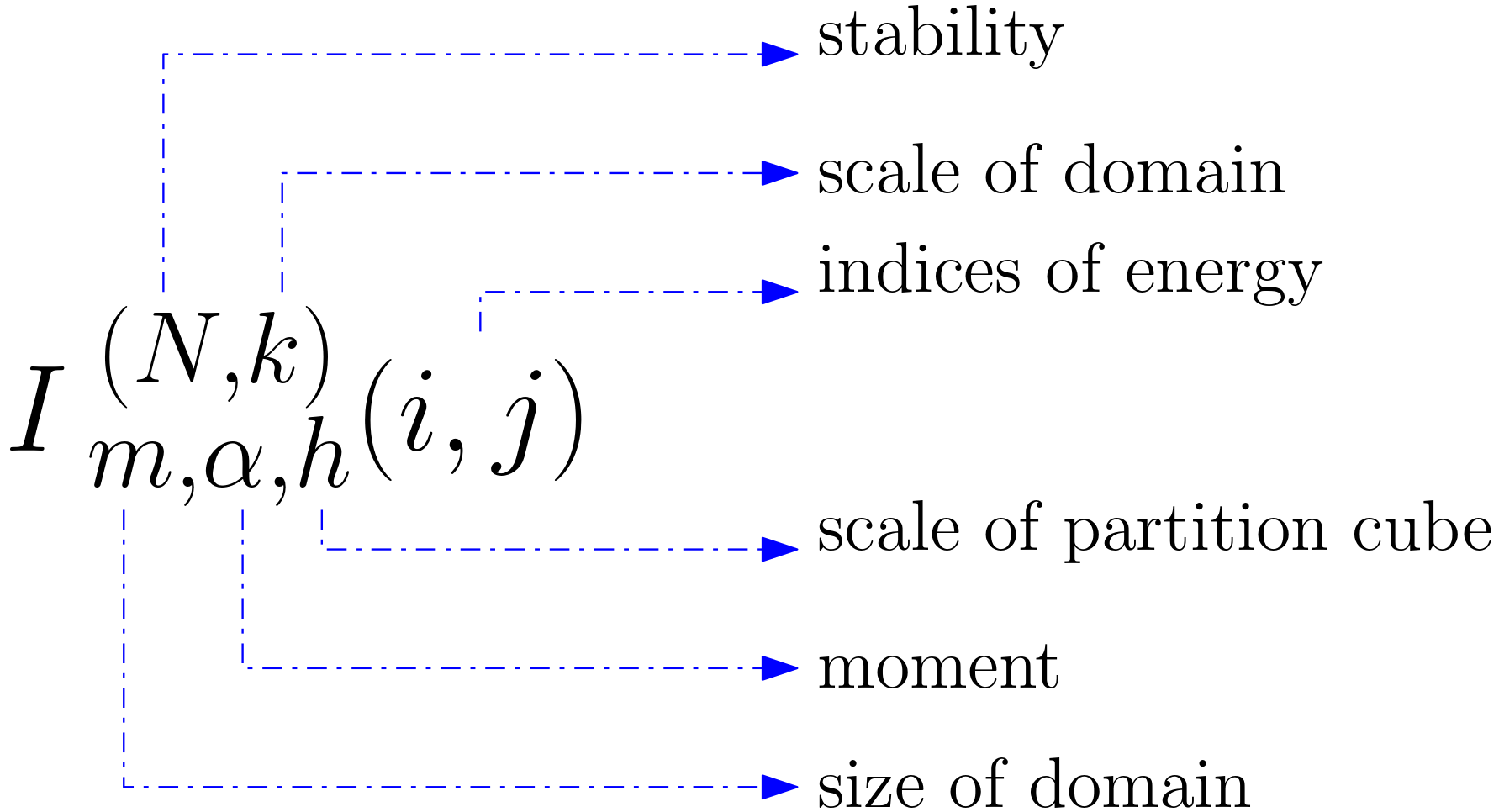}
		\caption{Quick explanation of parameters in the $\ell^1$-$L^2$ improved energy}\label{fig.energy}
	\end{figure}

	\medskip
	
	Finally, a natural question is the analyticity of $\p \mapsto \ab(\p)$. We prove the uniform bound of $\ab_m^{(k)}(\p)$ with respect to $m$, but does not track its dependence on $k$. Some difficulty comes from  Meyers' estimate in our proof, where the exponent depends on $k$ implicitly. As the work \cite{dg_Einstein} obtained the analyticity under the uniform separation condition,  we hope to explore a combination between the methods there and the present paper in the future. 

	\subsection{Organization of paper} The remaining of paper is organized as follows. In Section~\ref{sec.Pre}, we recall the results from previous work. The elementary objects, including the expression of $\ab_m^{(k)}(\p)$, the  $\ell^1$-$L^2$ energy, and the perturbed corrector equation are deduced in Section~\ref{sec.correctorEquation}. We then study the uniform bound of the  $\ell^1$-$L^2$ energy from the aspects of geometry, analysis, and combinatorics. In Section~\ref{sec.cube}, the pyramid partition of good cubes are discussed intensively. Section~\ref{sec.Poisson} studies the weighted estimate for the Poisson equation on percolation, where a coarse-grained Meyers' estimate is developed. Afterwards, we explore the cluster-growth decomposition in Section~\ref{sec.HarmonicExtension}, which allows us to conclude the induction $\{I_m(i,0)\}_{i \in \N}$ in Section~\ref{sec.KeyEstimate1}. The technique of hole separation is introduced in Section~\ref{subsec.Hole}, and we complete the induction  $\{I_m(i,j)\}_{i,j \in \N}$ in the rest of Section~\ref{sec.KeyEstimate2}. Finally, we conclude our proof of main theorem in Section~\ref{sec.Pf}. See Figure~\ref{fig.outline} for the outline. 
	
	A list of notations is attached in Appendix~\ref{sec.notation_list} for convenience.

	\bigskip
	
	\begin{figure}
		\centering
		\begin{tikzpicture}[node distance=2cm]\label{fig.diagram}
			
			\node (AD18) [input] {\cite{armstrong2018elliptic}, Section~\ref{subsec.Homo} \\ $\ab_m \xrightarrow{m \to \infty} \ab$};
			\node (Corrector) [input, below of = AD18] {Section~\ref{sec.correctorEquation} \\ \textit{Perturbed corrector equations}};
			\node (I0) [input, below of = Corrector] {Section~\ref{sec.KeyEstimate1} \\ \textit{Induction of $I_m(i,0)$}};
			\node (I1) [input, below of = I0] {Section~\ref{subsec.Inductionij1}-\ref{subsec.Inductionij2}\\ \textit{Induction of $I_m(i,j), j \geq 1$}};
			\node (Main) [input, below of = I1] {Section~\ref{sec.Pf} \\ $\ab^{(k)}_m \xrightarrow{m \to \infty} \ab^{(k)}$};
			
			\node (Partition) [input, left of = Corrector, xshift = -3cm] {Section~\ref{sec.cube} \\ \textit{Pyramid partition of good cubes}};
			\node (Poisson) [input, below of = Partition] {Section~\ref{subsec.fieldEda} \\ \textit{Weighted estimate of Poisson equation on clusters}};
			\node (Meyer) [input, below of = Poisson] {Section~\ref{subsec.fieldEd} \\ \textit{Coarse-grained Meyers' estimate}};
			
			\node (Growth) [input, right of = I0, xshift = 3cm] {Section~\ref{sec.HarmonicExtension} \\ \textit{Cluster-growth decomposition}};
			\node (Hole) [input, below of = Growth] {Section~\ref{subsec.Hole} \\ \textit{Hole separation}};
			
			\draw [arrow] (AD18) -- (Corrector);
			\draw [arrow] (Corrector) -- (I0);	
			\draw [arrow] (I0) -- (I1);	
			\draw [arrow] (I1) -- (Main) ;
			\draw [arrow] (AD18.west) |- ($(AD18.west)  + (-5.6cm,0cm)$) |- ($(Main.west)  + (-5.6cm,0cm)$) |- (Main.west);	
			\draw [arrow] (AD18)   -| (Growth);	
			
			\draw [arrow] (Partition) -- (Poisson);	
			\draw [arrow] (Poisson) -- (Meyer);	
			\draw [arrow] (Poisson) -- (I0);	
			\draw [arrow] (Meyer) -- (I1);	
			\draw [arrow] (Partition.west) |- ($(Partition.west)  + (-0.3cm,0cm)$) |- ($(Meyer.west)  + (-0.3cm,0cm)$) |- (Meyer.west);	
			\draw [arrow] (Partition.north) |- ($(Partition.north)  + (0cm, 0.3cm)$) |- ($(Partition.north)  + (12cm,0.3cm)$) |- (Hole.east);

			\draw [arrow] (Growth) -- (Hole);	
			\draw [arrow] (Growth) -- (I0);	
			\draw [arrow] (Hole) -- (I1);	
		\end{tikzpicture}
		\caption{The outline of proof.}\label{fig.outline}
	\end{figure}
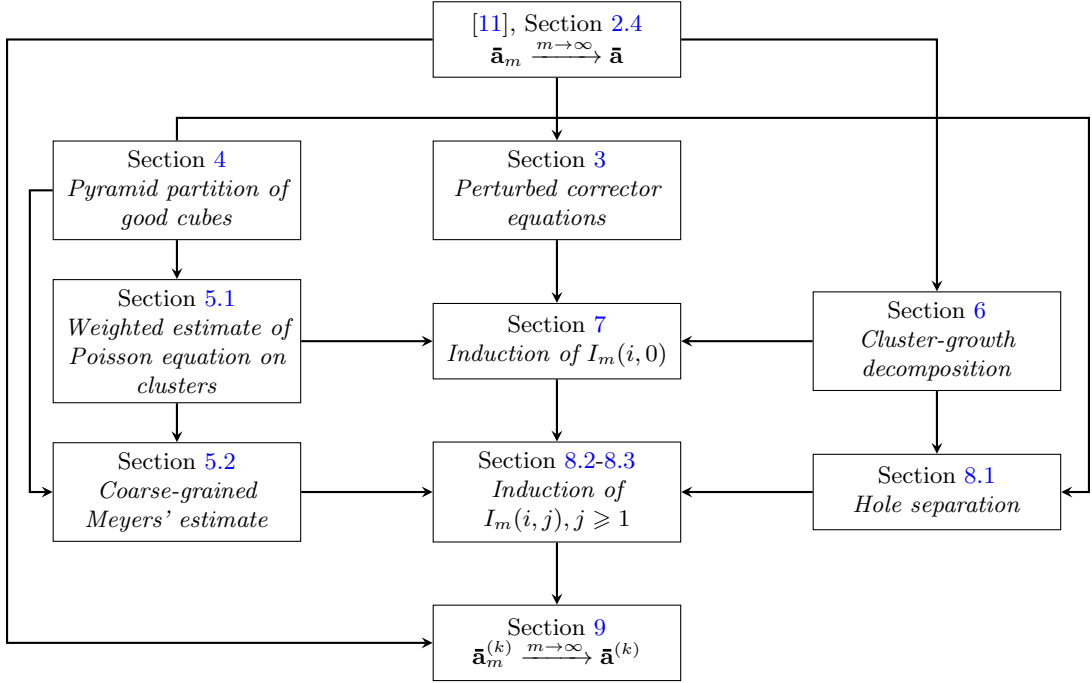
	
	\section{Preliminaries}\label{sec.Pre}
	In this part, we introduce the notations used throughout the paper and recall some preliminary results on homogenization.
	\subsection{Notations}
	We introduce at first the basic notations.
	\subsubsection{Space}\label{subsec.space}	
	We recall $(\Zd, \Ed)$ the $d$-dimensional integer lattice graph with the edge set induced by the nearest vertices. The distance function on $\Zd$ is defined with respect to the $\ell^{\infty}$ norm, i.e. $\dist(x,y) = \sup_{1\leq i \leq d}|x_i-y_i|$. For every $ V \subset \Zd$, we can construct two types of subgraphs: $(V, \Ed(V))$ with the geometry inherited from $(\Zd,\Ed)$, which is called \emph{the induced graph}; $(V, \Eda(V))$ with the geometry inherited from the percolation, and we call it \emph{the percolation graph}. They are defined respectively as 
	\begin{equation}\label{eq.defEdgePerco}
		\begin{split}
			\Ed(V) &:= \left\{\{x,y\} | x,y \in V, x \sim y \right\},  \\
			\Eda(V) &: =  \left\{\{x,y\} | x,y \in V, \a(x,y) \neq 0 \right\}.
		\end{split}
	\end{equation}
	The \emph{interior} of $V$ with respect to $(V, \Ed(V))$ and $(V, \Eda(V))$ are defined as
	\begin{align*}
		\itr(V) &:= \{x\in V | y \sim x \Longrightarrow y \in V\}, \\
		\itr_{\a}(V) &:= \{x\in V | y \sim x, \a(x,y) \neq 0 \Longrightarrow y \in V\},
	\end{align*}
	and the \emph{boundaries} are defined as $\partial(V) := V \backslash \itr(V)$ and $\partial_{\a}(V) := V \backslash \itr_{\a}(V)$. For any subset $U, V \subset \Zd$, we say $U \xleftrightarrow{\a} V$ if there exists an open path connecting $U$ and $V$. 
	

	We also define \emph{the oriented bonds} $\overrightarrow{\Ed}(V)$ for any set $V \subset \Zd$, 
	\begin{align*}
		\overrightarrow{\Ed}(V) := \{(x,y) \in V \times\ V : |x-y| = 1\},
	\end{align*}
	and we use $\overrightarrow{\Ed}$ for shorthand notation of $\overrightarrow{\Ed}(\Zd)$. An \emph{anti-symmetric vector field} $\overrightarrow{F}$ on $\overrightarrow{\Ed}$ is a function $\overrightarrow{F} : \overrightarrow{\Ed} \rightarrow \mathbb{R}$ such that ${\overrightarrow{F}(x,y) = - \overrightarrow{F}(y,x)}$. We abuse the notation by writing $\overrightarrow{F}(e)$, with $e = \{x,y\} \in \Ed$, to give its value with an arbitrary orientation of $e$, in the case it is well-defined (for example $\vert \overrightarrow{F} \vert(e)$). The \textit{divergence} of $\overrightarrow{F}$ is defined as $\nabla \cdot \overrightarrow{F} : \Zd \rightarrow \mathbb{R}$
	\begin{align*}
		\forall x \in \Zd, \quad \nabla \cdot \overrightarrow{F}(x) := \sum_{y \in \Zd : y \sim x} \overrightarrow{F}(x,y).
	\end{align*}
	For any $u : \Zd \rightarrow \mathbb{R}$, we define the discrete derivative $\nabla u : \overrightarrow{\Ed} \rightarrow \mathbb{R}$ as a vector field
	\begin{align*}
		\forall (x,y) \in \overrightarrow{\Ed}, \quad \nabla u(x,y):= u(y) - u(x),
	\end{align*} 
	and $\a \nabla u, \nabla u \Ind{\a \neq 0}$ are vector fields defined by
	\begin{align*}
		\a \nabla u (x,y) := \a(x,y)\nabla u(x,y), \qquad \nabla u \Ind{\a \neq 0}(x,y) := \nabla u(x,y)\Ind{\a(x,y) \neq 0}.
	\end{align*}
	Using these notations, the divergence form operator $-\nabla \cdot \a \nabla$ is well-defined and coincides with \eqref{eq.VSRW}. 
	
	We define the inner product for the vector field $\overrightarrow{F},\overrightarrow{G}: \overrightarrow{\Ed}(V) \rightarrow \R$ 
	\begin{equation*}
		\bracket{\overrightarrow{F},\overrightarrow{G}}_V := \sum_{e \in \Ed(V)} \overrightarrow{F}(e)\overrightarrow{G}(e) = \frac{1}{2}\sum_{x,y\in V, x\sim y} \overrightarrow{F}(x,y)\overrightarrow{G}(x,y).
	\end{equation*}

	\subsubsection{Function space}
	
	For every finite subset $U \subset \Zd$, we denote by $C_0(U)$ the functions vanishing on the boundary
	\begin{align}\label{eq.defC0}
		C_0(U) := \Ll\{f:U \to \R \vert \, f = 0 \text{ on } \partial U \Rr\}.
	\end{align}
	We denote by $L^p(U)$ the $L^p$-norm on lattice. It applies to a function $f$ or a vector field $\overrightarrow{F}$
	\begin{align}\label{eq.defLp}
		\norm{f}_{L^p(U)} := \Ll(\sum_{x \in U} \vert f(x)\vert^p\Rr)^{\frac{1}{p}},  \qquad \norm{\overrightarrow{F}}_{L^p(U)} := \Ll(\sum_{e \in \Ed(U)} \vert \overrightarrow{F}(e)\vert^p\Rr)^{\frac{1}{p}},
	\end{align}
	and its meaning is clear in context. We also define the associated normalized norm $\aL^p(U)$ as follows
	\begin{align}\label{eq.defaLp}
		\norm{f}_{\aL^p(U)} := \Ll(\frac{1}{\vert U\vert}\sum_{x \in U} \vert f(x)\vert^p\Rr)^{\frac{1}{p}},  \qquad  \norm{\overrightarrow{F}}_{\aL^p(U)} := \Ll(\frac{1}{\vert U\vert}\sum_{e \in \Ed(U)} \vert \overrightarrow{F}(e)\vert^p\Rr)^{\frac{1}{p}}.
	\end{align}

	\subsubsection{Cubes}
	A \emph{cube} is a subset of $\Zd$ of the form
	\begin{align*}
		\ZZ^d \cap \Ll(z + \Ll(-\frac{L}{2}, \frac{L}{2}\Rr)^d \Rr), \qquad z \in \ZZ^d, \qquad L \in 2 \N + 1,
	\end{align*}
	where \emph{center} and \emph{size} of the cube above is respectively $z$ and $L$. For a cube $\cu$, and $c \in \R_+$, we use $c\cu$ to denote the cube after scaling of $c$, i.e. the cube with the same center as $\cu$, and with size $2\lfloor \frac{c\size(\cu)}{2} \rfloor + 1$.
	
	We denote by $\cu_m(z)$ the triadic cube
	\begin{equation}\label{eq.defCube}
		\cu_m(z) := \ZZ^d \cap \Ll(z + \Ll(-\frac{3^m}{2}, \frac{3^m}{2}\Rr)^d \Rr), \qquad z \in 3^m \ZZ^d, \qquad m\in \N.
	\end{equation}
	We also write simply $\cu_m := \cu_m(0)$. 
	
	Let $\tT_n := \{z+ \cu_n: z\in 3^n \ZZ^d\}$ be the collection of triadic cubes of size $3^n$, and $\tT = \cup_{n=1}^{\infty} \tT_n$ be the collection of all cubes. For each $\cu \in \tT$, the predecessor of $\cu$ is the unique triadic cube $\Tilde{\cu} \in \tT$ satisfying 
	\begin{equation}\label{eq.predecessor}
		\cu \subset \Tilde{\cu} \quad \text{ and } \quad \frac{\text{size}(\Tilde{\cu})}{\text{size}(\cu)} = 3.
	\end{equation}
	If $\Tilde{\cu}$ is the predecessor of $\cu$, then we also say that $\cu$ is a successor of $\Tilde{\cu}$. Thus a cube of size larger than $1$ has $3^d$ successors.

	\subsection{Percolation}\label{subsec.perco}
	Recall that cluster refers to the connected component formed by open edges, and we usually use the vertices set $\cC$ to represent it. One should keep in mind its edge set on percolation $\Eda(\cC)$ can be different from the induced graph $\Ed(\cC)$; see Figure~\ref{fig.Induced} for an example.
	
	We denote by $\cCf$  the infinite cluster. We call a cluster \emph{hole} if it is a finite cluster, and denote it by $\OO$. Especially, for $x \notin \cCf$, let $\OO(x)$ be the hole containing $x$. Given a finite subset $U \subset \Zd$, we also define $\cCb(U)$ the boundary-connecting clusters
	\begin{align}\label{eq.def_Cbc}
		\cCb(U) := \{x \in \cu: x \stackrel{\a}{\longleftrightarrow} \partial U \text{ within } U \}.
	\end{align}
	We write $\cCf^{\a}, \OO^\a(x), \cCb^\a(U)$ when we highlight their dependence on the environment $\a$.

	\begin{figure}[b]
		\centering
		\includegraphics[height=100pt]{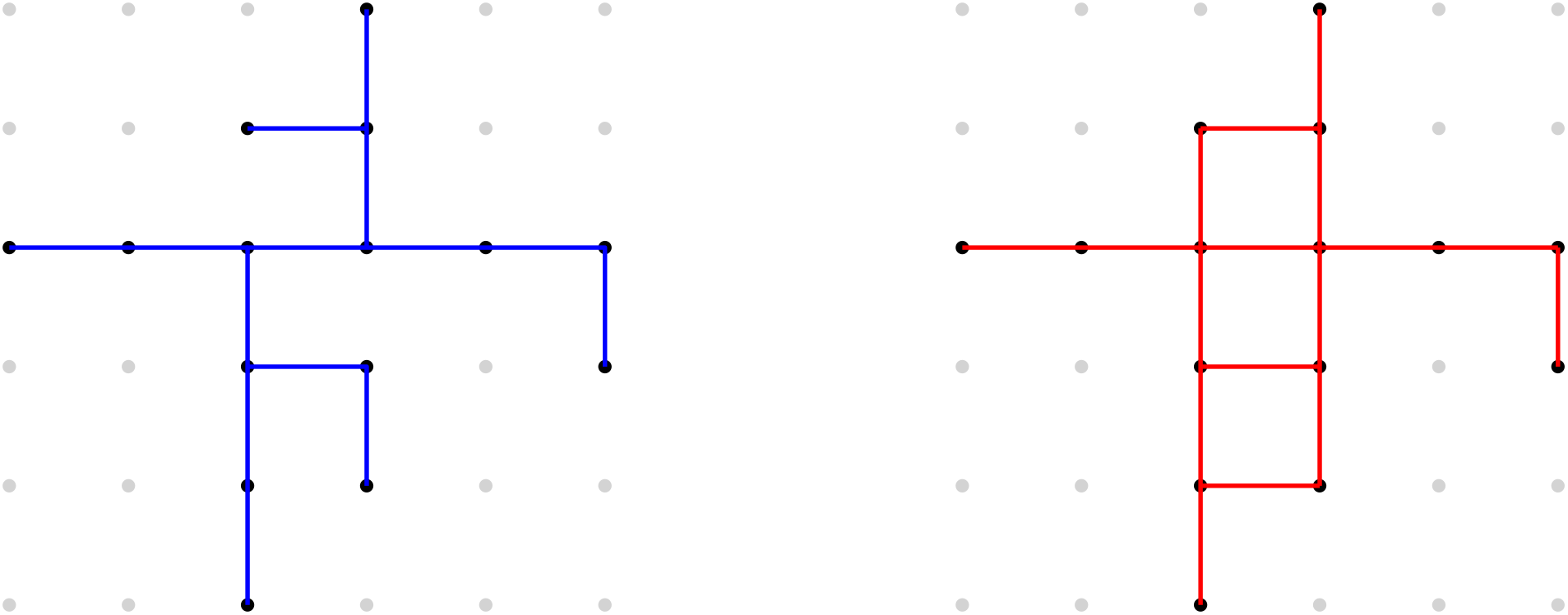}
		\caption{The figure on the left is a realization of cluster of percolation $(\cC, \Eda(\cC))$, and the figure on the right is its induced graph $(\cC, \Ed(\cC))$.}\label{fig.Induced}
	\end{figure}

	\subsection{Perturbation on environment}\label{subsec.Diff}

	Given $\a \in \{0,1\}^{\Ed}$ and $G \subset \Ed$, let $\a^G$ stand for the configuration $\a$ with all the bonds in $G$ open
	\begin{align}\label{eq.defaG}
		\a^G(e) := \Ind{e\in G} + \a(e)\Ind{e\not \in G}.    
	\end{align}
	For $e_1, e_2, \cdots, e_n \in \Ed$, we sometimes abuse the notation by identifying that 
	\begin{align}\label{eq.abuse}
		\a^{e_1, e_2, \cdots, e_n} := \a^{\{e_1, e_2, \cdots, e_n\}}.
	\end{align}
	Given a function $f$ defined on $\{0,1\}^{\Ed}$ and $G \subset \Ed$, we also keep the convention
	\begin{align}\label{eq.f_convention}
		f^G(\a) := f(\a^G),
	\end{align}
	For example, $\cCf^G$ stands for the infinite cluster under $\a^G$, and is shorthand for $\cCf^{\a^G} $. Recall the notation \eqref{eq.defEdgePerco} for the bond in percolation, and we define the following shorthand notation to indicate \emph{the intrinsic open bonds} 
	\begin{align}\label{eq.def_intrinsic}
		\EE(\cCf^G):=\Ed^{\a^G}(\cCf^{\a^G}).
	\end{align}
	This notation also applies similarly to other subset of percolation including $\cCf^{\a}, \OO^\a(x), \cCb^\a(U)$.
	
	\medskip
	
	For any $e \in \Ed$, \emph{the Glauber derivative} $D_e$ (also called \emph{the vertical derivative}) measures the variation when opening the edge $e$. Given $f: \{0,1\}^{\Ed} \to \R $, we define 
	\begin{align}\label{eq.defDe}
		(D_e f)(\a) := f(\a^{e}) - f(\a).
	\end{align} 
	Notice that for every $e, e' \in \Ed$, we have 
	\begin{equation*}  
		(D_{e'} D_e f)(\a) = (D_e f)(\a^{e'}) -  (D_e f)(\a) = f(\a^{e, e'}) - f(\a^{e'}) - f(\a^{e}) + f(\a).
	\end{equation*}
	In particular, the operators $D_e$ and $D_{e'}$ commute.  We can therefore define, for every $G = \{e_1,\ldots, e_n\} \subset \Ed$, the higher-order derivative
	\begin{equation}\label{e.def.DE}
		(D_G f)(\a) := (D_{e_1} \circ D_{e_2} \cdots \circ D_{e_n} f)(\a).
	\end{equation}
	
	We collect basic properties of the Glauber derivative.
	\begin{prop}
		For every function ${f,g: \{0,1\}^{\Ed} \to \R}$ and every finite set $G \subset \Ed$, the following identities hold.
		\begin{itemize}
			\item \textit{Inclusion-exclusion formula}
			\begin{align}\label{eq.Difference}
				(D_G f) = \sum_{G' \subset G} (-1)^{\vert G \setminus  G'\vert} f^{G'}.
			\end{align}
			\item \textit{Telescoping formula}
			\begin{align}\label{eq.Telescope}
				(f^G) = \sum_{G' \subset G} (D_{G'} f).
			\end{align}
			\item \textit{Leibniz formulas}
			\begin{align}\label{eq.Leibniz1}
				(D_G (fg)) = \sum_{G' \subset G} (D_{G'} f)(D_{G \setminus G'} g^{G'}),
			\end{align}
			and
			\begin{align}\label{eq.Leibniz2}
				(D_G (fg)) = \sum_{G_1, G_2 \subset G, G_1 \cup G_2 = G} (D_{G_1} f)(D_{G_2} g).
			\end{align}
			\item \text{Local expansion}
			\begin{align}\label{eq.aLocal}
				(\a^G - \a) = \sum_{e \in G} (\a^e - \a).
			\end{align}
		\end{itemize}
	\end{prop}
	\begin{proof}
		Equations \eqref{eq.Difference}, \eqref{eq.Telescope}, \eqref{eq.Leibniz1} and \eqref{eq.Leibniz2} follow the same proof as that in \cite[Proposition 5.1]{GGMN}. Concerning the identity \eqref{eq.aLocal}, it suffices some verification. The main observation is that 
		\begin{align*}
			\forall e,e' \in \Ed, e \neq e', \qquad \a^e(e') = \a(e').
		\end{align*}
		Thus at most one term on the {\rhs} of \eqref{eq.aLocal} contributes. For the case $e' \notin G$, two sides of  \eqref{eq.aLocal} vanish. For the case $e' \in G$, then the edge $e'$ is open in $\a^G$ and $(\a^G - \a)(e') = 1 - \a(e')$, which is equal to the contribution of $(\a^{e'} - \a)(e')$ on {\rhs} of \eqref{eq.aLocal}. This concludes the identity.
	\end{proof}

	Finally, when several functions exist in a  formula, we use the superscript $\#$ to indicate the function to which these operators are applied, keeping the others ``frozen''. That is, for $f, g: \{0,1\}^{\Ed} \to \R$, we write
	\begin{equation}\label{eq.sharp}
		\begin{split}
			D_G (f^{\#} g) &= (D_G f) g,\\
			D_G(f^{\#}  + g) &= \Ll\{
			\begin{array}{ll}
				f+g, & \text{ if } G = \emptyset, \\
				D_G f, & \text{ if } G \neq \emptyset. 
			\end{array}
			\Rr. \\
		\end{split}
	\end{equation}
	This notation can also represent a function of undetermined environment $\a^\#$.

	\subsection{Homogenization}\label{subsec.Homo}
	The effective conductivity is defined using the homogenization theory. Given $\a \in \{0,1\}^{\Ed}$, we define the normalized Dirichlet energy in a finite cube 
	\begin{equation}
		\label{eq.defmu}
		\mathscr{E}_{m,\xi}(\a) := \inf_{\ell_{\xi}+C_0(\cu_m)} \frac{1}{|\cu_m|}\Ll \langle\nabla v,\a \nabla v\Rr \rangle_{\cu_m},
	\end{equation}
	where $\ell_{\xi}(x) := \xi \cdot x$. Then this problem is associated to an elliptic equation on boundary-connecting cluster $\cCb(\cu_m)$
	\begin{equation}\label{eq.harmonicBC}
		\begin{cases}
			-\nabla \cdot (\a \nabla v_{m,\xi})=0 & \text { in } int(\cu_m) \cap \cCb(\cu_m), \\ 
			v_{m,\xi} =\ell_{\xi} & \text { on } \partial \cu_m \cap \cCb(\cu_m).
		\end{cases}
	\end{equation}
	By the Riesz representation theorem, the existence and uniqueness of $v_{m,\xi}$ is well-defined on $\cCb(\cu_m)$. Moreover, once $v_{m,\xi}$ follows \emph{the constant extension rule}, i.e. constant on every hole in $\cu_m$ disconnecting to the boundary $\partial \cu_m$, then we have 
	\begin{equation}\label{eq.harmonicCube}
		\begin{cases}
			-\nabla \cdot (\a \nabla v_{m,\xi})=0 & \text { in } int(\cu_m), \\ 
			v_{m,\xi} =\ell_{\xi} & \text { on } \partial \cu_m,
		\end{cases}
	\end{equation}
	and it is a minimiser of \eqref{eq.defmu}
	\begin{equation}\label{eq.energy_cluster}
		\begin{split}
			\mathscr{E}_{m,\xi}(\a) &= \frac{1}{|\cu_m|}\Ll \langle\nabla v_{m,\xi}(\a), \a \nabla v_{m,\xi}(\a) \Rr \rangle_{\cu_m} \\
			&= \frac{1}{|\cu_m|} \sum_{e \in \EE(\cCb(\cu_m))} \a(e) \vert \nabla  v_{m,\xi}(e) \vert^2.
		\end{split}
	\end{equation}
	We always assume the constant extension rule for $v_{m,\xi}$ throughout the paper. One should notice such extension is not unique. A specific choice of constant extension will be fixed later in Section~\ref{sec.HarmonicExtension}.

	We will use the following definition as the conductivity on percolation clusters
	\begin{equation}\label{eq.ab}
		\begin{split}
			\xi \cdot \ab_m(\p) \xi &:= \lim_{m \to \infty}\E_{\p}[\mathscr{E}_{m,\xi}(\a)], \\
			\ab(\p) &:= \lim_{m \to \infty}\ab_m(\p).
		\end{split}
	\end{equation}
	This definition is slightly different from the original $\ab(\p)$ in \cite[Lemma~4.3, Definition~5.1]{armstrong2018elliptic}, where the Dirichlet energy is only counted on the maximal cluster. However, the contribution from the boundary layer is negligible in the limit, and \cite[Theorem~C.1]{gu2022efficient} justifies that 
	\begin{align*}
		\lim_{m \to \infty} \mathscr{E}_{m,\xi}(\a),
	\end{align*}
	almost surely exists and equals to the original $\ab(\p)$ in \cite{armstrong2018elliptic}. Notice the trivial bound ${\mathscr{E}_{m,\xi}(\a) \leq d \vert \xi\vert^2}$ by testing \eqref{eq.defmu} with $\ell_\xi$, then the dominated convergence theorem applies. Therefore, the definitions of $\ab(\p)$ in \eqref{eq.ab} and \cite{armstrong2018elliptic} coincide.
	

	\section{Perturbed corrector equations}\label{sec.correctorEquation}
	This section is devoted to the basic observations about $\ab_m$ defined in \eqref{eq.ab}.

	\subsection{Chaos expansion}
	Our first observation is an explicit expression for derivatives of $\ab_m(\p)$, relying on the Dirichlet energy $\mathscr{E}_{m,\xi}$ \eqref{eq.defmu} and the Glauber derivative \eqref{eq.defDe}.
	
	\begin{prop}\label{pr:expansion_init}
		For every $m \in \N_+$, the mapping $\p \mapsto \ab_m(\p)$ is polynomial, and its $k$-th derivative $\ab^{(k)}_m$ has an explicit expression
		\begin{equation}
			\label{e:derivative_formula}
			\forall \p \in (0,1), \qquad \xi \cdot \ab^{(k)}_m(\p) \xi =  \frac{k!}{(1-\p)^k} \sum_{\substack{|F|=k\\ F\subset \Ed(\cu_m)}} \E_{\p}\Ll[D_F \mathscr{E}_{m,\xi}(\a) \Rr].
		\end{equation}
	\end{prop}
	\begin{proof}
		We divide the proof into 3 steps.
		
		\textit{Step~1: regularity of $\ab_m(p)$.} Recall that $\Eda(\cu_m)$ stands for the set of open edges in $\cu_m$, we have
		\begin{align*}
			\xi \cdot \ab_m(\p) \xi = \sum_{\a \in \{0,1\}^{\Ed(\cu_m)}} \p^{\vert \Eda(\cu_m)\vert} (1-\p)^{\vert \Ed(\cu_m) \setminus \Eda(\cu_m)\vert} \mathscr{E}_{m,\xi}(\a).
		\end{align*}
		Here $\mathscr{E}_{m,\xi}(\a)$ depends on the configuration $\a$, but does not depend on the parameter $\p$. This implies that $\p \mapsto \ab_m(\p)$ is polynomial.

		\smallskip
		
		\textit{Step~2: identification of derivatives.} Because $\p \mapsto \ab_m(\p)$ is polynomial, it is analytic and we have 
		\begin{align}\label{eq.abmTaylor}
			\ab_m(\p+\delta) - \ab_m(\p) = \sum_{k=1}^{\vert \Ed(\cu_m) \vert} \frac{\delta^k}{k!} \ab^{(k)}_m(\p).
		\end{align}
		It suffices to identify $\ab^{(k)}_m(\p)$, and we can then study it in a  coupling space: we sample $\{U(e)\}_{e \in \Zd}$ as i.i.d. random variables uniformly distributed in $[0,1]$, which are also independent of $\P_\p$. For $\delta \in [0, 1-\p]$, we define that
		\begin{equation}\label{eq.defaDelta}
			\a^{\delta}(e) := \a(e) + (1-\a(e))\boldsymbol{1}_{\{U(e) \leq \frac{\delta}{1-\p}\}},
		\end{equation}
		then clearly $\{\a^{\delta}(e)\}_{e \in \Zd}$ are i.i.d. Bernoulli random variables of parameter $(\p+\delta)$. 	
		
		In this coupling space, we denote by $\P$ the probability and $\E$ the associated expectation, then obtain that 
		\begin{align*}
			\xi \cdot (\ab_m(\p+\delta) - \ab_m(\p)) \xi = \E\Ll[\mathscr{E}_{m,\xi}(\a^\delta) - \mathscr{E}_{m,\xi}(\a) \Rr].
		\end{align*}
		We calculate at first the expectation with respect to $\{U_e\}_{e \in \Zd}$, which gives us an expansion in function of $\delta$. In order to simplify the notation, we also denote by ${\de := \frac{\delta}{1-\p}}, {N := \vert \Ed(\cu_m)\vert}$.
		\begin{equation}\label{eq.firstVariation}
			\begin{split}
				&\xi \cdot (\ab_m(\p+\delta) - \ab_m(\p)) \xi \\
				&= \E\Ll[ \E\Ll[ \mathscr{E}_{m,\xi}(\a^\delta) - \mathscr{E}_{m,\xi}(\a) \,  \big\vert  \a \Rr]\Rr]\\
				&=\sum_{G\subset \Ed(\cu_m)} \de^{|G|}  (1-\de)^{N-|G|} \E_\p \Big[\mathscr{E}_{m,\xi}(\a^G) - \mathscr{E}_{m,\xi}(\a)\Big],\\
				&=\sum_{k=0}^{N} c_{k,m}\de^k.
			\end{split}
		\end{equation}
		In the last line, we write the equation in function of $\de$. Using the Binomial expansion for $(1-\de)^{N-|G|}$ and comparing the coefficient, we get 
		\begin{align}
			\label{e:clm_1}
			c_{k,m} = \sum_{j=0}^{k} (-1)^{k-j} \binom{N-j}{k-j}\sum_{\substack{|G|=j\\ G\subset \Ed(\cu_m)}} \E_\p\Big[\mathscr{E}_{m,\xi}(\a^G) - \mathscr{E}_{m,\xi}(\a)\Big].
		\end{align}
		Compare \eqref{eq.firstVariation} and \eqref{eq.abmTaylor}, the case $k=0$ is trivial and consistent with \eqref{e:derivative_formula}. For the other cases, we deduce that
		\begin{align}\label{eq.abm_ckm}
			k \in \N_+, \qquad \xi \cdot \ab^{(k)}_m(\p) \xi = \frac{k!}{(1-\p)^k} c_{k,m},
		\end{align}		
		so it remains to simplify $c_{k,m}$.
		
		\smallskip
		
		\textit{Step~3: further simplification.}  For $k \geq 1$, we apply  Binomial expansion and obtain 
		\begin{align*}
			\sum_{j=0}^{k} (-1)^{k-j} \binom{N-j}{k-j} \E_\p\Big[ \mathscr{E}_{m,\xi}(\a)\Big] = (1-1)^k \E_\p\Big[ \mathscr{E}_{m,\xi}(\a)\Big] = 0.
		\end{align*}
		We only need to simplify 
		\begin{align*}
			\forall k \in \N_+, \qquad c_{k,m} = \sum_{j=0}^{k} (-1)^{k-j} \binom{N-j}{k-j}\sum_{\substack{|G|=j\\ G\subset \Ed(\cu_m)}} \E_\p\Big[\mathscr{E}_{m,\xi}(\a^G)\Big].
		\end{align*}
		
		This is a combinatorial identity. We observe that given any $|G|=j$, there exists exactly $\binom{N-j}{k-j}$ subsets $F$ such that $|F|=k$ and $G\subset F\subset  \Ed(\cu_m)$, thus we have the following identity
		\begin{equation*}
			\binom{N-j}{k-j} \sum_{\substack{|G|=j\\ G\subset \Ed(\cu_m)}} \mathscr{E}_{m,\xi}(\a^G) = \sum_{\substack{|F|=k\\ F\subset \Ed(\cu_m)}} \sum_{\substack{|G|=j\\ G\subset F}} \mathscr{E}_{m,\xi}(\a^G).
		\end{equation*}
		Thus $c_{k,m}$ can be further simplified as 
		\begin{align*}
			\forall k \in \N_+, \qquad c_{k,m} &= \E_\p\Ll[\sum_{j=0}^{k} (-1)^{k-j} \sum_{\substack{|F|=k\\ F\subset \Ed(\cu_m)}} \sum_{\substack{|G|=j\\ G\subset F}} \mathscr{E}_{m,\xi}(\a^G)\Rr]\\
			&=  \sum_{\substack{|F|=k\\ F\subset \Ed(\cu_m)}} \E_\p\Ll[\sum_{G\subset F} (-1)^{k-|G|}  \mathscr{E}_{m,\xi}(\a^G)\Rr] \\
			&= \sum_{\substack{|F|=k\\ F\subset \Ed(\cu_m)}} \E_\p\Ll[ D_F \mathscr{E}_{m,\xi}(\a)\Rr].
		\end{align*}
		The passage from the second line to the third line relies on \eqref{eq.Difference}. We put this identity back to \eqref{eq.abm_ckm}, and conclude the identity \eqref{eq.abmTaylor}. 
	\end{proof}
	
	We end this subsection with an explicit expression of $\E_{\p}\Ll[D_F \mathscr{E}_{m,\xi}(\a) \Rr]$.
	\begin{lemma}
		For every $F \subset \Ed(\cu_m)$ with $\vert F \vert = k \geq 1$, we have the following identity
		\begin{multline}\label{eq.DFmu_formula}
			\sum_{\substack{|F|=k\\ F\subset \Ed(\cu_m)}} D_F \mathscr{E}_{m,\xi}(\a) =  	\frac{1}{|\cu_m|}
			\Ll \langle \nabla \ell_\xi, \a \sum_{\substack{|F|=k \\ F \subset \Ed(\cu_m)}} (D_{F}\nabla v_{m,\xi})(\a)\Rr \rangle_{\cu_m}\\
			+	\frac{1}{|\cu_m|}\sum_{e\in \Ed(\cu_m)}
			\Ll \langle \nabla\ell_\xi, (\a^e-\a)\sum_{\substack{|G|=k-1 \\  G \subset \Ed(\cu_m)}} (D_{G}\nabla v_{m,\xi})(\a^e)\Rr \rangle_{\cu_m}.
		\end{multline}
	\end{lemma}
	\begin{proof} 
		Viewing that $v_{m,\xi} \in \ell_\xi + C_0(\cu_m)$ and is $\a$-harmonic, \eqref{eq.defmu} can be reformulated as
		\begin{align*}
			\mathscr{E}_{m,\xi}(\a) = \frac{1}{|\cu_m|}\Ll \langle\nabla \ell_{\xi}, \a \nabla v_{m,\xi}(\a) \Rr \rangle_{\cu_m}.
		\end{align*}
		Therefore, using \eqref{eq.sharp}, we get the equation 
		\begin{align*}
			D_F \mathscr{E}_{m,\xi}(\a) =\frac{1}{|\cu_m|}\Ll \langle\nabla \ell_{\xi}, D_F\Big(\a^\# \nabla v_{m,\xi}(\a^\#)\Big) \Rr \rangle_{\cu_m}.
		\end{align*}
		We then apply Leibniz's formula \eqref{eq.Leibniz1} and obtain
		\begin{align*}
			D_F \mathscr{E}_{m,\xi}(\a) &=  \frac{1}{| \cu_m | } \Ll\langle\nabla \ell_{\xi}, \sum_{G \subset F} (D_G \a) \Ll(D_{F \setminus G}\nabla v_{m,\xi} \Rr)(\a^{G})\Rr\rangle_{\cu_m}\\
			&=  \frac{1}{| \cu_m | } \Ll\langle\nabla \ell_{\xi}, \a  \Ll(D_{F}\nabla v_{m,\xi} \Rr)(\a)\Rr\rangle_{\cu_m} \\
			& \qquad + \frac{1}{| \cu_m | }\Ll\langle\nabla \ell_{\xi}, \sum_{e\in F} \Ll(\a^e-\a\Rr) \Ll(D_{F \setminus \{e\}}\nabla v_{m,\xi}\Rr)(\a^{e})\Rr\rangle_{\cu_m}. 
		\end{align*}
		Because $(D_G \a)(e')$ will vanish when $|G| \geq 2$, we only keep two terms in the expansion above.
		
		Finally, we make a sum over $F \subset \Ed(\cu_m)$ with $\vert F \vert = k$, which will yield the first line of \eqref{eq.DFmu_formula} immediately. The second line of \eqref{eq.DFmu_formula} comes from the fact
		\begin{align*}
			&\sum_{\substack{|F|=k\\ F\subset \Ed(\cu_m)}} \frac{1}{| \cu_m | }\Ll\langle\nabla \ell_{\xi}, \sum_{e\in F} \Ll(\a^e-\a\Rr) \Ll(D_{F \setminus \{e\}}\nabla v_{m,\xi}\Rr)(\a^{e})\Rr\rangle_{\cu_m} \\
			&= \frac{1}{|\cu_m|}\sum_{e\in \Ed(\cu_m)}
			\Ll \langle \nabla \ell_\xi , (\a^e-\a)\sum_{\substack{|G|=k-1 \\  G \subset \Ed(\cu_m) \setminus \{e\}}}\Ll(D_{G}\nabla v_{m,\xi}\Rr)(\a^{e})\Rr \rangle_{\cu_m}\\
			&= \frac{1}{|\cu_m|}\sum_{e\in \Ed(\cu_m)}
			\Ll \langle \nabla \ell_\xi , (\a^e-\a)\sum_{\substack{|G|=k-1 \\  G \subset \Ed(\cu_m)}}\Ll(D_{G}\nabla v_{m,\xi}\Rr)(\a^{e})\Rr \rangle_{\cu_m}
		\end{align*}
		The passage from the first line to the second line is a change of variable $G = F\backslash\{e\}$.  Then we  observe that $(D_G \na v_{m,\xi})(\a^e)=0$ whenever $e\in G$, so we can extend the domain of sum from the second line to the third line. This concludes \eqref{eq.DFmu_formula}.
	\end{proof}

	\subsection{Improved $\ell^1$-$L^2$ energy}
	In this part,  we aim to dominate $\ab^{(k)}_m$ via \emph{the improved  $\ell^1$-$L^2$ energy}. We set $\a$ as the default environment for the functions. Then for every $F\subset \Ed(\cu_m)$, we define that
	\begin{equation}\label{eq.V_F_j}
		V_{m,\xi}(F, j) := \sum_{\substack{|G|=j\\G\subset \Ed(\cu_m) \setminus F}} D_{F\cup G} v_{m,\xi},
	\end{equation}
	with the convention that 
	\begin{equation}\label{eq.V_F_convention}
		\begin{split}
			V_{m,\xi}(F,j)=0, \qquad  &\forall j<0. \\
		\end{split}
	\end{equation}
	In particular, $V_{m,\xi}(\emptyset, 0) = v_{m,\xi}$. The improved  $\ell^1$-$L^2$ energy is then defined as 
	\begin{align}\label{eq.improvedEnergy}
		\bar{I}_{m,\xi}(i,j) :=  \E_\p\Ll[\frac{1}{\vert \cu_m \vert} \sum_{\substack{ \vert F \vert = i \\  F \subset \Ed(\cu_m)}} \sum_{e \in \Ed(\cu_m)} \vert \nabla V_{m,\xi}(F,j)(e)\vert^2  \Rr]. 
	\end{align}
	At the first glance, this quantity is quite large due to the summations over $F$ in \eqref{eq.improvedEnergy} and $G$ in \eqref{eq.V_F_j}, but they can be bounded by the compensation from the Glauber derivatives.
	
	\medskip
	
	We express $\big(D_{G}\nabla v_{m,\xi}\big)(\aio{e})$ using the expression \eqref{eq.V_F_j} at first. 
	\begin{lemma}\label{lem.VmSimple}
		For every $e\in \Ed(\cu_m)$, the following identity holds 
		\begin{equation}\label{eq.VmSimple}
			\sum_{\substack{|G|=k-1\\G\subset \Ed(\cu_m)}}\big(D_{G}\nabla v_{m,\xi}\big)(\aio{e}) = V_{m,\xi}(\emptyset,k-1) + V_{m,\xi}(\{e\}, k-1) - V_{m,\xi}(\{e\}, k-2).
		\end{equation}
	\end{lemma}
	\begin{proof}
		Notice $v_{m,\xi}(\aio{e}) = v_{m,\xi}(\a) + (D_e v_{m,\xi})(\a)$, we split the sum into three terms
		\begin{align*}
			&\sum_{\substack{|G|=k-1\\G\subset \Ed(\cu_m)}} D_G\Ll(v_{m,\xi}(\a)+ (D_e v_{m,\xi})(\a)\Rr)\\
			=& V_{m,\xi}(\emptyset,k-1) + \sum_{\substack{|G|=k-1\\e\not \in G\subset \Ed(\cu_m)}} D_{G\cup\{e\}}v_{m,\xi}(\a) + \sum_{\substack{|G|=k-1\\e \in G\subset \Ed(\cu_m)}} D_{G \setminus \{e\}} D_eD_ev_{m,\xi}(\a).
		\end{align*}
		The second term is just the definition of $V_{m,\xi}(\{e\},k-1)$ because
		\begin{align*}
			\sum_{\substack{|G|=k-1\\e\not \in G\subset \Ed(\cu_m)}} D_{G\cup\{e\}}v_{m,\xi}(\a) &= \sum_{\substack{|G|=k-1\\  G \subset \Ed(\cu_m) \setminus \{e\}}} D_{G\cup\{e\}}v_{m,\xi}(\a)\\
			&= V_{m,\xi}(\{e\},k-1).
		\end{align*}
		For the third term, observing that $D_e D_e = -D_e$, we make a change of variable ${H = G \setminus \{e\}}$
		\begin{align*}
			\sum_{\substack{|G|=k-1\\e\in G \subset \Ed(\cu_m)}}  D_{G \setminus \{e\}} D_eD_e v_{m,\xi}(\a) 
			&= \sum_{\substack{|H|=k-2\\H\subset\Ed(\cu_m)\setminus\{e\}}}\Ll(-D_{H \cup \{e\}} v_{m,\xi}(\a)\Rr)\\ 
			&=-V_{m,\xi}(\{e\},k-2).
		\end{align*}
	\end{proof}

	We are now ready to give a bound of $\Ll\vert \ab^{(k)}_m(\p) \Rr\vert$ using the improved Dirichlet energy.
	\begin{lemma}\label{lem.abk_VFj}
		For every $k, m \in \N$ and $\p \in (0,1)$, the following estimate holds 
		\begin{multline}
			\label{e:main_terms_to_bound}
			\Ll\vert \ab^{(k)}_m(\p) \Rr\vert \leq  \frac{4 \cdot k!}{(1-\p)^k}\sum_{i=1}^d\Big(1 + \bar{I}_{m, e_i}(0,k) + \bar{I}_{m, e_i}(0,k-1) \\
			+ \bar{I}_{m, e_i}(1,k-1) + \bar{I}_{m, e_i}(1,k-2)\Big).
		\end{multline}
	\end{lemma} 
	\begin{proof}
		Combing \eqref{e:derivative_formula}, \eqref{eq.DFmu_formula} and \eqref{eq.VmSimple}, we have
		\begin{align*}
			\Ll\vert \ab^{(k)}_m(\p) \Rr\vert &\leq \sum_{i=1}^d e_i \cdot \ab^{(k)}_m(\p) e_i \\
			& \leq  \frac{k!}{(1-\p)^k} \frac{1}{\vert \cu_m \vert} \sum_{i=1}^d \E_{\p}\Ll[ \Ll\vert \Big \langle \nabla \ell_{e_i}, \a \nabla V_{m,e_i}(\emptyset,k) \Big \rangle_{\cu_m} \Rr \vert \Rr.\\
			& \quad + \Ll\vert \sum_{e\in \Ed(\cu_m)}
			\Big \langle \nabla\ell_{e_i}, (\a^e-\a) \nabla \Ll(V_{m,e_i}(\emptyset,k-1)\Rr)\Big \rangle_{\cu_m}\Rr \vert\\
			& \quad +  \Ll.\Ll\vert \sum_{e\in \Ed(\cu_m)}
			\Big \langle \nabla\ell_{e_i}, (\a^e-\a) \nabla \Ll( V_{m,e_i}(\{e\}, k-1) - V_{m,e_i}(\{e\}, k-2)\Rr)\Big \rangle_{\cu_m}\Rr \vert\Rr].
		\end{align*}
		We  apply Young's inequality and obtain
		\begin{align*}
			& \Ll\vert \ab^{(k)}_m(\p) \Rr\vert \\
			& \leq \frac{4 \cdot k!}{(1-\p)^k}  \sum_{i=1}^d \Ll( \underbrace{\E_\p \Ll[ \frac{1}{|\cu_m|}  \sum_{e' \in \Ed(\cu_m)}  \vert \na V_{m, e_i}(\emptyset,k) (e') \vert^2 + \vert \nabla \ell_{e_i}(e') \vert^2 \Rr]}_{ \leq \bar{I}_{m, e_i}(0,k) +1} \Rr.\\
			&\quad +  \E_\p \Ll[\frac{1}{|\cu_m|}\sum_{e\in \Ed(\cu_m)} \sum_{e' \in \Ed(\cu_m)} (\aio{e}-\a)(e')  \vert \na V_{m, e_i}(\emptyset,k-1) (e')\vert^2 \Rr] \\
			&\quad + \Ll. \underbrace{\E_\p \Ll[\frac{1}{|\cu_m|}\sum_{e\in \Ed(\cu_m)} \sum_{e' \in \Ed(\cu_m)} \Ll(\vert \na V_{m, e_i}(\{e\},k-1) (e')\vert^2  + \vert \na V_{m, e_i}(\{e\},k-2) (e')\vert^2  \Rr)\Rr]}_{=  \bar{I}_{m, e_i}(1,k-1) + \bar{I}_{m, e_i}(1,k-2)} \Rr).
		\end{align*}
		It remains to clarify the equation in the third line above, which can be simplified as 
		\begin{align*}
			&\E_\p \Ll[\frac{1}{\vert \cu_m\vert}\sum_{e\in \Ed(\cu_m)}\sum_{e'\in \Ed(\cu_m)} (\aio{e}-\a)(e')   \vert \na V_{m, e_i}(\emptyset,k-1)(e') \vert^2 \Rr]\\
			&=  \E_\p \Ll[\frac{1}{\vert \cu_m\vert}\sum_{e'\in\Ed(\cu_m)} \Ll(\sum_{e\in \Ed(\cu_m)} (\aio{e}-\a)(e')\Rr)  \vert \na V_{m, e_i}(\emptyset,k-1)(e') \vert^2 \Rr]\\
			&\leq \E_\p \Ll[\frac{1}{\vert \cu_m\vert}\sum_{e' \in \Ed(\cu_m)} \vert \na V_{m, e_i}(\emptyset,k-1)(e') \vert^2 \Rr]\\
			&=  \bar{I}_{m, e_i}(0,k-1).
		\end{align*}
		From the second line to the third line, we use the fact that $\aio{e}(e')$ differ from $\a(e')$ only when $e'=e$. Thus, we have ${\sum_{e\in \Ed(\cu_m)} (\aio{e}-\a)(e') \leq 1}$. This concludes \eqref{e:main_terms_to_bound}.

	\end{proof}

	\subsection{Perturbed corrector equations}
	Thanks to Lemma~\ref{lem.abk_VFj}, we need to give an estimate for the improved $\ell^1$-$L^2$ energy uniformly with respect to $m$. The idea is to establish the following perturbed corrector equation \eqref{eq.recurrence_eq}. In the remaining part of the paper, we fix a vector $\xi \in \R^d$ with $\xi = 1$, and use the shorthand
	\begin{align}\label{eq.shorthand}
		v_{m,\xi} \equiv v_{m}, \qquad V_{m,\xi}(F,j) \equiv V_m(F,j), \qquad \bar{I}_{m,\xi}(i,j) \equiv \bar{I}_{m}(i,j).
	\end{align}
	They are defined on the percolation $\a$, and  $v^G_{m}$ refers to the one on $\a^G$ following the convention \eqref{eq.f_convention}.

	\begin{lemma}
		\label{lem.recurrence_eq}
		For any $F\subset \Ed(\cu_m)$ and $j \in \N$, we have
		\begin{equation}\label{eq.recurrence_eq}
			-\na \cdot (\aio{F} \na V_m(F,j)) = \na \cdot \W_m(F,j) \quad \text{in } \itr(\cu_m), 
		\end{equation}
		where $\W_m(F,j)$ is an anti-symmetric vector field on $\overrightarrow{\Ed}(\cu_m)$
		\begin{align}
			\label{e:induction_eqW}
			\W_m(F,j) &= \sum_{e\in \Ed(\cu_m) \setminus F}(\aio{e}-\a)\Big(\na V_m(F,j-1) - \na V_m(F\cup \{e\}, j-2) \\
			\nonumber & \qquad \qquad \qquad \qquad + \na V_m(F\cup \{e\}, j-1)\Big) \\
			\nonumber &\qquad +\sum_{e\in F}(\aio{e}-\a)\Big(\na V_m(F\setminus\{e\},j)-\na V_m(F,j-1)\Big).
		\end{align}
	\end{lemma}
	\begin{proof}
		Using \eqref{eq.V_F_j}, we develop the {\lhs} 
		\begin{align*}
			-\na \cdot (\aio{F} \na V_m(F,j)) = -\sum_{\substack{|G|=j\\G\subset \Ed(\cu_m) \setminus F}} \na \cdot \Ll(\aio{F} D_{F\cup G} v_m\Rr).
		\end{align*}
		In the rest of the proof, we will focus on the single term on the {\rhs} at first, and then apply to the sum.
		
		\textit{Step 1: study of the single term.} We focus on the single term with $G\subset \Ed(\cu_m) \setminus F$
		\begin{align*}
			-\na \cdot \Ll(\aio{F} \na D_{F\cup G}v_m\Rr)
			&= -\sum_{E \subset F\cup G}(-1)^{|F\cup G\setminus E|} \na \cdot \Ll(\aio{F} \na v^E_m\Rr)\\
			&= -\sum_{E \subset F\cup G}(-1)^{|F\cup G\setminus E|} \na \cdot \Ll((\aio{F}-\aio{E}) \na v^E_m\Rr)\\
			&= -\na \cdot D_{F\cup G}\Ll((\aio{F}-\aio{\#}) \na v^\#_m\Rr).
		\end{align*}
		Here from the first line to the second line, we use the inclusion-exclusion formula \eqref{eq.Difference}. From the second line to the third line, we add the harmonic equation ${ -\na \cdot \Ll(\aio{E} \na v^E_m\Rr) = 0}$ in $\itr (\cu_m)$ thanks to the harmonic extension in \eqref{eq.harmonicCube}. This insertion creates more cancellation and that is why the {\rhs} has lower order. The notation \eqref{eq.sharp} is applied in the last line.
		
		We continue the manipulation of the single term by Leibniz's formula \eqref{eq.Leibniz2}
		\begin{equation}\label{eq.DecomDFG}
			\begin{split}
				&-D_{F\cup G}\Ll((\aio{F}-\aio{\#}) \na v^\#_m\Rr) \\
				&= -(\aio{F}-\a)\na D_{F\cup G}v_m
				+ \sum_{e\in F\cup G}(\aio{e}-\a)\na D_{F\cup G}v_m\\
				&\qquad  + \sum_{e\in F\cup G}(\aio{e}-\a)\na D_{F\cup G\setminus \{e\}}v_m.
			\end{split}
		\end{equation}
		Here we use the fact that $D_G (\aio{F}-\aio{\#})$ vanishes when $\vert G \vert \geq 2$. Using the locality $\aio{F} -\a = \sum_{e\in F}(\aio{e}-\a)$, in the second line of \eqref{eq.DecomDFG}, the first term can compensate parts of the second term. Thus we obtain that 
		\begin{equation}\label{eq.RecurSingle}
			\begin{split}
				-D_{F\cup G}\Ll((\aio{F}-\aio{\#}) \na v^\#_m\Rr) 
				&=  \sum_{e\in G}(\aio{e}-\a)\na D_{F\cup G}v_m  \\
				& \qquad + \sum_{e\in F\cup G}(\aio{e}-\a)\na D_{F\cup G\setminus \{e\}}v_m.
			\end{split}
		\end{equation}
		
		\medskip
		
		\textit{Step 2: sum-up.} Now, we need to make the sum-up over $|G|=j, G\subset \Ed(\cu_m) \setminus F$, and we treat it term by term.
		
		\textit{Step~2.1: sum-up for the second line in \eqref{eq.RecurSingle}.} With a change of variable ${G = H \cup \{e\}}$, and we exchange the order of sum
		\begin{align*}
			&\sum_{\substack{|G|=j\\G\subset \Ed(\cu_m) \setminus F}}   \sum_{e\in G} \na \cdot\Ll((\aio{e}-\a)\na D_{F\cup G}v_m\Rr) \\
			&=     \na \cdot\Ll(\sum_{e\in  \Ed(\cu_m) \setminus F}(\aio{e}-\a) \sum_{\substack{|H|=j-1\\ H \subset \Ed(\cu_m) \setminus (F \cup \{e\})}}  \na D_{(F\cup \{e\}) \cup H}v_m\Rr)\\
			&=  \na \cdot\Ll(\sum_{e\in  \Ed(\cu_m) \setminus F}(\aio{e}-\a) \nabla V_m(F \cup \{e\}, j-1)\Rr).
		\end{align*}
		From the second line to the third line, we use directly the definition \eqref{eq.V_F_j}. This is the second line in \eqref{e:induction_eqW}. 
		
		\smallskip
		
		\textit{Step~2.2: sum-up for the third line in \eqref{eq.RecurSingle}, part $\sum_{e\in F}$.} For this term, we can exchange the order of sum directly as one does not influence each other 
		\begin{align*}
			&\sum_{\substack{|G|=j\\G\subset \Ed(\cu_m) \setminus F}}   \sum_{e\in F} \na \cdot\Ll((\aio{e}-\a)\na D_{F\cup G \setminus \{e\}}v_m\Rr) \\
			&=    \na \cdot\Ll(\sum_{e\in F}(\aio{e}-\a) \sum_{\substack{|G|=j\\G\subset \Ed(\cu_m) \setminus F}}  \na D_{(F \setminus \{e\}) \cup G }v_m\Rr). 
		\end{align*}
		However, the definition \eqref{eq.V_F_j} does not apply directly to the last line. We can decompose the second sum in the following identity
		\begin{align}\label{eq.RecurDecom}
			\sum_{\substack{|G|=j\\G\subset \Ed(\cu_m) \setminus F}}\Big(\cdots\Big) = \sum_{\substack{|G|=j\\G\subset \Ed(\cu_m) \setminus (F \setminus \{e\})}} \Big(\cdots\Big) - \sum_{\substack{|G|=j\\ e \in G\subset \Ed(\cu_m) \setminus (F \setminus \{e\})}} \Big(\cdots\Big).
		\end{align}
		For the first term on {\rhs} of the identity above, we apply \eqref{eq.V_F_j} that 
		\begin{align*}
			\sum_{\substack{|G|=j\\G\subset \Ed(\cu_m) \setminus (F \setminus \{e\})}}  D_{(F \setminus \{e\}) \cup G }v_m = V_m(F \setminus \{e\}, j).
		\end{align*}
		For the second term on {\rhs} of \eqref{eq.RecurDecom}, we make the change of variable $G = H \cup \{e\}$ 
		\begin{align*}
			\sum_{\substack{|G|=j\\ e \in G\subset \Ed(\cu_m) \setminus (F \setminus \{e\})}}  D_{(F \setminus \{e\}) \cup G }v_m 
			&= \sum_{\substack{|H|=j-1\\  H \subset \Ed(\cu_m) \setminus F}}  D_{(F \setminus \{e\}) \cup (H \cup \{e\}) }v_m \\
			&= \sum_{\substack{|H|=j-1\\  H \subset \Ed(\cu_m) \setminus F}}  D_{F \cup H }v_m\\
			&=  V_m(F, j-1).
		\end{align*}
		These two terms corresponds to the third line of \eqref{e:induction_eqW}.
		
		\smallskip
		
		\textit{Step~2.3: sum-up for the second line in \eqref{eq.RecurSingle}, part $\sum_{e\in G}$.} The manipulation of this term combines the technique from Step~2.1 and Step~2.2. Like Step~2.1, we need at first a change of variable ${G = H \cup \{e\}}$
		\begin{align*}
			&\sum_{\substack{|G|=j\\G\subset \Ed(\cu_m) \setminus F}}   \sum_{e\in G} \na \cdot\Ll((\aio{e}-\a)\na D_{F\cup G \setminus \{e\}}v_m\Rr) \\
			&=     \na \cdot\Ll(\sum_{e\in  \Ed(\cu_m) \setminus F}(\aio{e}-\a) \sum_{\substack{|H|=j-1\\ H \subset \Ed(\cu_m) \setminus (F \cup \{e\})}}  \na D_{F\cup H}v_m\Rr).
		\end{align*}
		The difference is that here the operator is $D_{F\cup H}$ instead of $D_{(F\cup \{e\}) \cup H}$, and we need to apply the technique \eqref{eq.RecurDecom} once again to conclude that 
		\begin{align*}
			\sum_{\substack{|H|=j-1\\ H \subset \Ed(\cu_m) \setminus (F \cup \{e\})}}  \na D_{F\cup H}v_m = V_m (F, j-1) - V_m(F\cup \{e\}, j-2).
		\end{align*}
		These two terms corresponds to the first line in \eqref{e:induction_eqW} and we finish the proof.
	\end{proof}

	\begin{remark}\label{rmk.OrderDecre}
		We highlight the good property of this perturbed corrector equation \eqref{e:induction_eqW}. Roughly, the Dirichlet energy for the term $V_m$ of \eqref{e:induction_eqW} should be bounded from that of the $\W_m$ on the {\rhs}. By defining the degree, 
		\begin{align}\label{eq.defdegV}
			\deg (V_m(F, j)) := \vert F \vert + 2j,
		\end{align}
		we also observe that the degree of the functions on the {\rhs} strictly decreases by at least $1$ compared to the {\lhs}. Therefore, in finite steps of induction, we can finally obtain the upper bound of the Dirichlet energy.
		
		Nevertheless, the observation above contains some technical difficulties. The control of $V_m$ by $\W_m$ is direct under the uniformly ellipticity (see \cite{GGMN, dg1}), which is not the case for the percolation. To see this, it is clear that the {\lhs} of \eqref{eq.recurrence_eq} is only supported on $\a^F$ rather than $\Ed(\cu_m)$. Therefore, the induction of the improved energy $\bar{I}_{m}(i,j)$ is not closed here.  We need a better understanding of \eqref{eq.recurrence_eq} both from analysis, geometry and combinatorial, so more techniques will be involved. Its final proof is presented in Proposition~\ref{pr:key_estimate_avg}.
		
	\end{remark}

	\section{Pyramid partitions of good cubes}\label{sec.cube}
	In this section, we will introduce the partition of good cubes. This technique rooted from the work by Pisztora \cite{pisztora1996surface}, further developed by Armstrong and Dario in \cite{armstrong2018elliptic} and then applied in \cite{dario2021corrector, dario2021quantitative}. The main novelty  here is a partition in several scales called \emph{pyramid partitions of good cubes}, i.e. every partition cube can be further refined (see Figure~\ref{fig.cube} as an illustration), so that they can treat the high-order perturbed corrector equation \eqref{eq.recurrence_eq}. We will state a general local partition of good cubes as its basis in Section~\ref{subsec:partition}, then build the pyramid partitions in Section~\ref{subsec.RenomalPerco}. Finally, we establish its important properties in Section~\ref{subsec:counting}.

	The following lemma will be frequently used in this section.
	\begin{lemma}[Lemma~2.3, \cite{armstrong2018elliptic}]
		\label{le:Borel_Contelli}
		Let $B_n$ be a series of event satisfying 
		\begin{equation*}
			\P[B_n] \leq C\exp(-C^{-1}3^{ns}),
		\end{equation*}
		with some finite positive constants $C,s$. We define 
		\begin{equation*}
			N := \sup\{3^n:B_n \text{ happens}\},
		\end{equation*}
		then there exists a finite positive constant $C' = C'(C,s)$ such that 
		\begin{align*}
			\P[N \geq t] \leq \exp(-{C'}^{-1}t^s).
		\end{align*}
	\end{lemma}

	\subsection{Local partition of good cubes}\label{subsec:partition}

	For each $\a \in \Omega$, we denote by $\gG(\a) \subset \tT$ a (random) set of $\a$-measurable ``good cubes'', which has certain properties. In the following paragraphs, we would drop the dependence on $\a$ and just write $\gG$ for simplicity. The following two conditions are usually assumed for the random set of good cubes $\gG$.
	\begin{itemize}
		\item \emph{Locality:} 
		\begin{equation}
			\label{e:local_dependence}
			\forall \cu \in \tT, \qquad \text{the event }\{\cu \not \in \gG\} \text{ is } \mathcal{F}(3\cu)\text{-measurable}. 
		\end{equation}
		\item \emph{Tail probability estimate:} there exist two finite positive constants $K, s$ such that
		\begin{equation}
			\label{e:rarely_bad}
			\forall m \in \N_+, \qquad \sup_{z\in 3^m\ZZ^d}\P[z+\cu_m \not \in \gG] \leq K\exp(-K^{-1}3^{ms}).
		\end{equation}
	\end{itemize}
	We will highlight them in the concrete statement.
	
	As its name indicates, \emph{the partition of good cubes} is to divide a subset of $\Zd$ into disjoint union of elements from $\gG$. Our first result is \emph{the local partition of good cubes}, which has already appeared in \cite[Proposition~2.1 (iv)]{armstrong2018elliptic}. This local version will be heavily used in our work, because we always work on a finite domain. We reformulate it for the completeness of proof.
	
	\begin{prop}[Local partition of good cubes]\label{prop.partition}
		Given a random set of good cubes $\gG\subset \tT$ satisfying \eqref{e:local_dependence}, then for any cube $\cu \in \gG$, there exists its local partition of good cubes denoted by $\sSl_{\gG}(\cu) \subset \gG$, such that the following properties hold:
		\begin{enumerate}
			\item All predecessors of elements of $\sSl_{\gG}(\cu)$ in $\cu$ are good: for every $\cu', \cu'' \in \tT$, 
			\begin{align*}
				\cu' \in  \sSl_{\gG}(\cu) \text{ and } \cu'  \subset \cu'' \subset \cu  \Longrightarrow \cu''  \in  \gG.
			\end{align*}
			\item Neighboring elements of $\sSl_{\gG}(\cu)$ have comparable sizes: for every $\cu', \cu'' \in \sSl_{\gG}(\cu)$ such that $\operatorname{dist}(\cu', \cu'')\leq 1$, we have
			\begin{equation}
				\label{e:neighbor_comparable_size}
				\frac{1}{3}\leq \frac{\size(\cu'')}{\size(\cu')}\leq 3.        
			\end{equation}
			\item Locality: the partition $\sSl_{\gG}(\cu)$ is $\fF(3 \cu)$-measurable.
			\item Estimate for the coarseness of $\sSl_{\gG}(\cu)$: if we also assume the tail event estimate \eqref{e:rarely_bad} for $\gG$, then there exists a finite positive constant $C(K,s)$, such that for any $\cu$ and any $x \in \cu$, we have
			\begin{equation}
				\label{e:coarseness_partition}
				 \P\Ll[\size(\cu_{\sSl_{\gG}}(x)) \geq t\Rr] \leq \exp(-C^{-1}t^s),
			\end{equation}
			where we denote by $\cu_{\sSl_{\gG}}(x)$ the unique cube in partition $\sSl_{\gG}(\cu)$ containing $x$.
		\end{enumerate}
	\end{prop}
	
	\begin{remark}
		For the notation $\cu_{\sSl_{\gG}}(x)$ in \eqref{e:coarseness_partition}, $\cu$ indicates the cube to be partitioned. A concrete example is that, for $\cu' \in \gG$, we implement the local partition $\sSl_{\gG}(\cu')$, then for any $x \in \cu'$,  the cube $\cu'_{\sSl_{\gG}}(x)$ is the one containing $x$.
	\end{remark}
	\begin{proof}
		The proof is similar as \cite[Proposition~2.1]{armstrong2018elliptic}. Without loss of generality we assume $\size(\cu) = 3^m$, and let $\tT(\cu)$ be the collection of triadic cubes which are subsets of $\cu$. For each $\cu' \in \tT(\cu)$, define
		\begin{align}
			\kK_m(\cu') := \{&\cu''\in\tT(\cu):\exists n\in \N \text{ and } \cu^0,\cdots, \cu^n \in \tT(\cu), \\
			& \nonumber \text{ such that }\cu^0 = \cu', \cu^n = \cu'', \\
			& \nonumber \text{ and } \forall 1\leq k \leq n,   \size(\cu^k) = 3^k\size(\cu'),  \dist(\Tilde{\cu^k},\cu^{k-1})\leq 1 \},
		\end{align}
		where we use $\Tilde{\cu^k}$ to denote the predecessor of $\cu^k$. Here $\kK_m(\cu)$ can be seen as the set of triadic cubes in $\cu$ obtained from transport and scaling.
		
		We consider a subset of good cubes $\Bar{\gG}(\cu)$, whose elements are candidates for the property (1)
		\begin{align}\label{eq.BargG}
			\Bar{\gG}(\cu) := \{\cu' \in \tT(\cu): \kK_m(\cu') \subset \gG\}.
		\end{align} 
		Then we introduce the partition $\sSl_{\gG}(\cu)$ as follows. For each $x\in \cu$, let $\cu_{\sSl_{\gG}}(x)$ be the largest cube in $\Bar{\gG}(\cu)$ which has a successor not belonging to $\Bar{\gG}(\cu)$, and we also define  
		\begin{equation*}
			\sSl_{\gG}(\cu) := \Ll\{\cu_{\sSl_{\gG}}(x)\Rr\}_{x \in \cu}.
		\end{equation*}
		This does form a partition of $\cu$, because every $x \in \cu$ associates a unique $\cu_{\sSl_{\gG}}(x)$, and elements in $\sSl_{\gG}(\cu)$ are disjoint or identical from the maximality in the definition. 
		
		We now verify the properties one by one. Property (1) follows from the definition of $\Bar{\gG}(\cu)$ and (3) is the result of \eqref{e:local_dependence}. For property (2), we prove by contradiction and assume that two cubes $\cu', \cu'' \in \sSl_{\gG}(\cu)$ are neighbors with $\size(\cu'')\geq 9\size(\cu')$. Since $\kK_m(\cu')$ contains all successors of $\cu''$, it follows that all successors of $\cu''$ are also in $\Bar{\gG}(\cu)$ from the construction of $\cu'$, which contradicts the construction of $\cu''$.
		
		Concerning the property (4), if we also assume \eqref{e:rarely_bad}, it is not hard to show that for any $\cu' \in \tT(\cu)$
		\begin{equation}\label{eq.BarcuTail}
			\P[\cu' \notin \Bar{\gG}(\cu)] \leq \sum_{n = \size(\cu')}^\infty 3^{d(n - \size(\cu'))}K\exp(-K^{-1} n^s) \leq K'\exp(-(K')^{-1} n^s).
		\end{equation}
		Let $\cu_n(x)$ be the unique cube in $\tT_n$ containing $x$, and we define the event 
		\begin{equation*}
			B_n(x) :=  \{\cu_n(x) \text{ or one of its successors are not in } \Bar{\gG}(\cu) \}.
		\end{equation*}
		Applying Lemma~\ref{le:Borel_Contelli} for $B_n(x)$ with the help of \eqref{eq.BarcuTail}, we conclude (4).
	\end{proof}

	\bigskip
	We then define a subset of the good cubes $\Lambda(\gG, t, C) \subset \gG$, in which the partition cubes has  a good control for the size.
	\begin{definition}\label{def.LambdaG}
		A cube $\cu$ is an element in $\Lambda(\gG, t, C)$ if it satisfies the following three conditions:
		\begin{itemize}
			\item $\cu$ is good: $\cu \in \gG$;
			\item Control of the averaged $\ell^t$ moment:
			\begin{align}\label{eq.defgGpc}
				\frac{1}{|\cu|} \sum_{x\in \cu}\vert \cu_{\sSl_{\gG}}(x)\vert^t \leq C
			\end{align}
			\item Control of the maximum partition cube:
			\begin{align}\label{eq.sizeMax}
				\sup_{x \in \cu} \size( \cu_{\sSl_{\gG}} (x)) \leq \size(\cu)^{\frac{d}{d+t}}.
			\end{align}
		\end{itemize}
	\end{definition}
	
	\begin{lemma}\label{lem.local}
		Assume \eqref{e:local_dependence} for $\gG$, then the event $\{\cu \in\Lambda(\gG, t, C)\}$ is $\mcl F(3\cu)$-measurable.
	\end{lemma}
	\begin{proof}
		This is a direct result from \eqref{e:local_dependence} and (3) in Proposition~\ref{prop.partition}.
	\end{proof}
	
	It is not hard to find cubes in $\Lambda(\gG, t, C)$. This property is in fact reformulated from \cite[Proposition~2.4]{armstrong2018elliptic}, but we do not use the notation of minimal scale in order to highlight that the tail estimate is for the local event.  
	\begin{prop}\label{prop.cubePartitionSize}
		Assume \eqref{e:local_dependence}, \eqref{e:rarely_bad} for $\gG$, then for every $t \geq 1$, there exists a finite positive constants $C$ depending on $t,s,K,d$ there, such that 
		\begin{equation}\label{eq.momentTail}
			\P[	\cu \notin\Lambda(\gG, t, C)] \leq C \exp\Ll(- C^{-1}  \size(\cu)^{\frac{ds}{d+t+s}}\Rr).
		\end{equation}
	\end{prop}
	\begin{proof}
		The proof follows exactly \cite[Proposition~2.4]{armstrong2018elliptic}. More precisely, \eqref{eq.defgGpc} is derived from Step~1, Step~2, eq.(2.15) in \cite[Proposition~2.4]{armstrong2018elliptic}, and \eqref{eq.momentTail}, \eqref{eq.sizeMax} are derived from Step~3 in \cite[Proposition~2.4]{armstrong2018elliptic}. In fact, the proposition there is stated for the global partition of good cubes, but the proof also makes use of the local version to approximate and to create independence, which covers the case in our context.
	\end{proof}

	\subsection{Construction of pyramid partitions}\label{subsec.RenomalPerco}
	In this subsection, we carry the local partition of good cubes described in Proposition~\ref{prop.partition} to the supercritical percolation clusters. As we track carefully the local event in these propositions, we can further build a sequence of partitions by recurrence.
	
	In the construction, we will have coarsened good cubes in different scales. Roughly, they satisfy four nice properties: 
	\begin{enumerate}
		\item The very-well-connectedness.
		\item Meyers' estimate.
		\item The control of partition cube size.
		\item The robustness under the perturbation of percolation.
	\end{enumerate}
	Let us introduce their definitions at first. Here the control of the partition cube size is stated in Proposition~\ref{prop.cubePartitionSize}. The very-well-connectedness is defined as follows.
	Recall $\tT$ the set of triadic cubes and the definition of predecessors, successor in \eqref{eq.predecessor}. We also recall the crossing cluster from \cite{pisztora1996surface}. 
	\begin{definition}[Crossability and crossing cluster]
		\label{def:crossing_cluster}
		We say that a cube $\cu$ is crossable if each of the $d$ pairs of opposite $(d-1)$-dimensional faces of the cube $\cu$ is joined by an open path in $\cu$. We say that a cluster $\cC\subset \cu$ is a crossing cluster for $\cu$ if $\cC$ intersects each of the $(d-1)$-dimensional faces of $\cu$.
	\end{definition}
	We adapt the word ``very-well-connected" instead of ``good cube" in \cite[Definition~2.6]{armstrong2018elliptic} in order to avoid the abuse in our context. We also add a lower bound of the size to exclude some exceptions in the small scale.
	\begin{definition}[Well-connected and very-well-connected]
		\label{def:well_connected}
		We say that a triadic cube $\cu\in \tT$ is \emph{well-connected} if there exists a crossing cluster $\cC$ for the cube $\cu$ such that:
		\begin{enumerate}
			\item Each cube $\cu'$ with $\size(\cu') \in [\frac{1}{10}\size(\cu),\frac{1}{2}\size(\cu)]$ and $\cu' \cap \frac{3}{4}\cu \neq \emptyset$ is crossable.
			\item Every open path $\gamma \subset \cu'$ with $\diam(\gamma) \geq \frac{1}{10}\size(\cu)$ is connected to $\cC$ within $\cu'$.
		\end{enumerate}
		We denote this unique crossing cluster $\cC$ by $\cC_*(\cu)$, and call it  \emph{the maximal cluster} in the cube $\cu$.
		
		A triadic cube $\cu$ is said \emph{very-well-connected} if it is well-connected, all its $3^d$ successors,  $9^d$ second order successors, and $27^d$ third order successors are well-connected, and its size is larger than $10000$.   
	\end{definition}
	
	Finally, we recall that, the very-well-connected property is valid with high probability in the supercritical regime.
	\begin{lemma}[eq.(2.24) of \cite{antal1996chemical}]
		For every $\p \in (\p_c, 1]$, there exist two finite positive constants $C(d,\p)$ and $c(d,\p) \in (0, \frac{1}{2}]$ such that 
		\begin{align}\label{eq.tailWellConnected}
			\forall \cu \in \tT, \qquad \P_\p[\cu \text{ is very-well-connected }] \geq 1 - C \exp(-c\size(\cu)).
		\end{align}
	\end{lemma}

	The second property is Meyers' estimate, which is developed in previous work \cite[Proposition~3.8]{armstrong2018elliptic}. Here we give the version involving boundary in our context. Given a well-connected cube $\cu$ and $\cu' \supset \cu$, then we define \emph{the enlarged maximal cluster} for every $c \geq 1$
	\begin{equation}\label{eq.def_clt_boundary}
		\begin{split}
			\cCp(c \cu \cap \cu') &:= \Ll\{x \in c \cu \cap \cu': x  \stackrel{\a}{\longleftrightarrow} \cCm(\cu),  \Rr.   \\
			&  \qquad \text{ or if } c \cu \cap \partial \cu' \neq \emptyset,  x  \stackrel{\a}{\longleftrightarrow}  c \cu \cap \partial \cu', \\
			&  \qquad \qquad \Ll. \text{ all the connection are realized within }  c \cu \cap \cu' \Rr\}.
		\end{split}
	\end{equation} 
	That is, $\cCp(c \cu \cap \cu')$ is the union of the cluster containing $\cCm(\cu)$ and the clusters connecting the boundary of $\cu'$ in $c \cu \cap \cu'$. We usually pick $\cu'$ to be very large and $\cu$ to be its partition cube, then $\cCp(c \cu \cap \cu')$ can be regarded as the local version of $\cCb(\cu')$ in $c \cu \cap \cu'$.


	\begin{definition}[Meyers' estimate]\label{def.Meyers}
		Given three constants ${\eps \in (0,1)}$, ${\rr > \frac{4}{3}}$, ${C \in (0, \infty)}$, we denote by $\mM(\a, \rr, \epsilon, C) \subset \tT$ the collection of cubes $\cu$ satisfying the following property under $\a$
		\begin{enumerate}
			\item $\cu, \frac{4}{3}\cu, \rr \cu$ are very-well connected;
			\item For every cube $\cu' \supset \cu$ and function $v: \rr \cu \cap \cu' \rightarrow \R$ satisfying 
			\begin{equation}\label{eq.Meyers_Poisson}
				\Ll\{
				\begin{array}{cc}
					-\na \cdot (\a \na v) = -\na\cdot \w  &\text{ in }  {\itr \cCp(\rr \cu \cap \cu')}, \\
					v = 0 & \text{ on } \rr \cap \partial \cu'.
				\end{array}
				\Rr.
			\end{equation}
			with  $\w$ an anti-symmetric vector field supported on $\overrightarrow{\Eda}$, one has the estimate
			\begin{multline}
				\label{e:Meyers}
				\left(\frac{1}{\left|\frac{4}{3}\cu\right|} \sum_{e\in \EE(\cCp(\frac{4}{3} \cu \cap \cu'))}|\nabla v(e)|^{2+\eps}\right)^{\frac{1}{2+\eps}} \\
				\leq C\left(\frac{1}{\left|\rr \cu\right|} \sum_{e\in \EE(\cCp(\rr \cu \cap \cu'))}|\nabla v(e)|^2\right)^{\frac{1}{2}}+C\left(\frac{1}{\left|\rr \cu\right|} \sum_{e\in \EE(\cCp(\rr \cu \cap \cu')}|\w(e)|^{2+\eps}\right)^{\frac{1}{2+\eps}}.
			\end{multline}
		\end{enumerate}
	\end{definition}

	The existence of cubes satisfying Meyers' estimate is stated in the following lemma. 
	\begin{lemma}\label{lem.Meyers}
		For every $\rr \in (\frac{4}{3},2)$, there exist finite positive constants $\eps_{Me}, C_{Me}$ such that $\mM(\a, \rr, \eps_{Me}, C_{Me})$ is a collection of good cubes satisfying \eqref{e:local_dependence} and \eqref{e:rarely_bad}.
	\end{lemma}
	\begin{proof}
		The result about interior Meyers' estimate was already discussed in \cite[Proposition~3.8]{armstrong2018elliptic} and \cite[Proposition~11]{dario2021corrector}. Meyers' estimate involving the boundary is similar and we refer to \cite[Appendix~C]{AKMbook}. The locality inherits that from the definition \eqref{eq.def_clt_boundary} and \eqref{e:Meyers}.
	\end{proof}

	The last property is the robustness of cube and partition defined in the following sense. Here $\mathbf{S}(\a)$ stands for a general statement $\mathbf{S}$ under the realization $\a$. 
	\begin{definition}[$N$-$\stable$]\label{def.N_stable}
		Given $N \in \N_+$, then a cube and its partition $(\cu,\pP)$ is $N$-$\stable$ for the statement $\mathbf{S}(\a^\#)$ if the following property holds: fix the realization of $(\cu,\pP)$ under $\a$,  then it satisfies $\mathbf{S}(\a^F)$ for every $F \subset \Ed(\cu)$ with $\vert F\vert \leq N$.
	\end{definition}
	\begin{remark}\label{rmk.N_stable}
		More explicitly, here we only have one partition $(\cu,\pP)$, but it satisfies the statement under  $\sum_{k=0}^N{ \vert \Ed(\cu) \vert \choose k}$ different percolation configurations. We use $\a^\#$ to indicate the admissible environment $\a^F$, in order to avoid confusion and highlight the generality in the discussion/proof about $N$-$\stable$ property.
	\end{remark}

	Now, we state the main result of this section, which defines the good cubes in this paper. They are actually a sequence of good cubes and coarsened good cubes. Here the scaling constant $\rr$ can be chosen.
	\begin{prop}[$N$-$\stable$ good cubes]
		\label{prop.good_cube_property}
		Given $d \geq 2$, $\p \in (\p_c, 1]$, $\rr \in (\frac{4}{3}, 2)$, and $N \in \N$, there exist constants $C_{Me},\eps_{Me}$ and $\{C_{Sc,k}, K_k, s_k\}_{k\geq 0}$ depending on $\p, d, \rr, N$ and a sequence of good cubes $(\gGN_k)_{k \geq 0}$ 
		\begin{align}\label{eq.good_cube_subset}
			\tT \supset \gG \supset \gGN_0 \supset \gGN_1 \supset \gGN_2 \supset \gGN_3 \cdots
		\end{align}
		For every cube $\cu \in \gGN_k$, it is $N$-$\stable$ for the statement that
		\begin{enumerate}[label=(\roman*)]
			\item \emph{Meyers' estimate}: $\cu \in \mM(\a^\#, \rr, \eps_{Me}, C_{Me})$;
			\item \emph{Connectivity}: $\cu, \frac{4}{3}\cu, \rr \cu$ are very-well connected under $\a^\#$;
		\end{enumerate}
		and it also satisfies the following properties
		\begin{enumerate}	
			\item \emph{Partition cube size}: when $k \geq 1$, then $\cu \in \Lambda\Ll(\gGN_{k-1}, d + k, C_{Sc,k}\Rr)$.
			\item \emph{Locality and tail-estimate}: $\cu$ verifies the locality \eqref{e:local_dependence} and the tail probability estimate \eqref{e:rarely_bad} with constants $K_k, s_k$.
		\end{enumerate}
		
	\end{prop}
	\begin{proof}
		We prove it by induction. 
		\smallskip

		\textit{Step of basis $k=0$:} We define that
		\begin{align}\label{eq.good_cubs_basis}
			\Ll\{\cu \in \gGN_{0}\Rr\} := \bigcap_{\substack{ F \subset \Ed(\cu) \\ \vert F\vert \leq N}} \Ll\{ \cu \in \mM(\a^F, \rr, \eps_{Me}, C_{Me})\Rr\}.
		\end{align}
		
		This definition justifies the $N$-$\stable$ property for Meyers' estimate and the connectivity. Actually, the part about connectivity was included implicitly in (1) of Definition~\ref{def.Meyers}.
		
		Then the locality in  \eqref{e:Meyers} is natural, as it only depends on the connectivity of $\rr\cu$. Viewing that every term on {\rhs} is  $\mcl F^{\a}(3\cu)$-measurable, which justifies the locality \eqref{e:local_dependence} for $\gGN_{0}$. Then the union bound gives that
		\begin{align*}
			\P_{\p}\Ll[\cu \notin  \gGN_{0}\Rr] & \leq \sum_{\substack{ F \subset \Ed(\cu) \\ \vert F\vert \leq N}} 	\P_{\p}\Ll[\cu \notin \mM(\a^F, \rr, \eps_{Me}, C_{Me})\Rr].
		\end{align*}
		Because $\a^F$ only differs from $\a$ on finite bonds, the tail estimate from \cite[Proposition~3.8, (3.17)]{armstrong2018elliptic} thus yields that of $\gGN_{0}$.
		
		
		\textit{Step of induction:}
		Assume that the case for some $k \in \N$ is already established. Then we define 
		\begin{align}\label{eq.good_cubs_Gk}
			\gGN_{k+1} :=  \Lambda\Ll(\gGN_{k}, d + k + 1, C_{Sc,k+1}\Rr).
		\end{align}
		According to \eqref{eq.defgGpc}, we have naturally $\gGN_{k+1} \subset \gGN_{k}$ above, which implies \eqref{eq.good_cube_subset}. $N$-$\stable$ property for Meyers' estimate and connectivity also inherits from $\gGN_{0}$ thanks to the inclusion $\gGN_{k+1} \subset \gGN_{k} \subset \gGN_{0}$. We then verify the other properties one by one:
		\begin{enumerate}
			\item The control of the size of the partition cube is direct from \eqref{eq.good_cubs_Gk}.
			\item By assumption,  $\gGN_k$ satisfies  \eqref{e:local_dependence} and \eqref{e:rarely_bad}, then  Lemma~\ref{lem.local}  and Proposition~\ref{prop.cubePartitionSize}  justifies respectively \eqref{e:local_dependence} and \eqref{e:rarely_bad} for the case $(k+1)$. 
		\end{enumerate}
		
	\end{proof}

	We define the pyramid partitions of good cubes. Together with the properties in Proposition~\ref{prop.good_cube_property}, they are our basic tools to resolve the degenerate geometry in \eqref{eq.recurrence_eq}. 	
	\begin{definition}[Pyramid partitions of good cubes]\label{def.pyramid}
		Given $k \in \N_+$ and $N \in \N$, for every $\cu \in \gGN_k$, we can construct a sequence partition of good cubes $(\pPN_{h})_{0 \leq h \leq k-1}$ as follows. 
		\begin{equation}\label{e.pyrmid}
			\begin{split}
				\pPN_{k-1}(\cu) &:= \sSl_{\gGN_{k-1}}(\cu), \\
				\forall 0 \leq h \leq k-2, \qquad \pPN_{h}(\cu) &:= \bigcup_{\cu' \in \pPN_{h+1}} \sSl_{\gGN_{h}}(\cu').  
			\end{split}
		\end{equation}
		In the equation \eqref{e.pyrmid} above, for $\cu' \in \pPN_{h+1}$, we call the cubes in $\sSl_{\gGN_{h}}(\cu')$ its \emph{children partition cubes}, and respectively $\cu'$ is the unique \emph{parent partition cube} for $\sSl_{\gGN_{h}}(\cu')$. For any $x \in \cu$, we denote by $\cu_{\pPN_{h}}(x)$ the unique element in $\pPN_{h}$ that contains $x$.
	\end{definition}
	\begin{proof}[Proof of the well-definedness of Definition~\ref{def.pyramid}]
		Since $\cu \in \gGN_k$, then (1) of Proposition~\ref{prop.good_cube_property} implies $\cu \in \gGN_{k-1}$, thus Proposition~\ref{prop.partition} ensures the partition $\sSl_{\gGN_{k-1}}(\cu)$. Afterwards, \eqref{eq.good_cube_subset} gives inductively the local partition.
	\end{proof}
	
	\begin{remark}\label{rmk.counting}
		For $\cu \in \gGN_k$ and $0 \leq h \leq k-1$, the partition $\pPN_{h}(\cu)$ can be different from $\sSl_{\gGN_{h}}(\cu)$, although the elements of both partitions are from $\gGN_{h}$. In particular, neighbors in $\pPN_{h}(\cu)$ may not even  have comparable sizes, thus we should be really careful in the proof; see Lemma~\ref{le:overlap_counting}. Meanwhile, we can do $(k-h)$ rounds of partitions from $\cu$ to $\pPN_{h}(\cu)$, and this is one key tool to study \eqref{eq.recurrence_eq}.  
	\end{remark}

	\begin{figure}[b]
		\centering
		\includegraphics[width=0.7\textwidth]{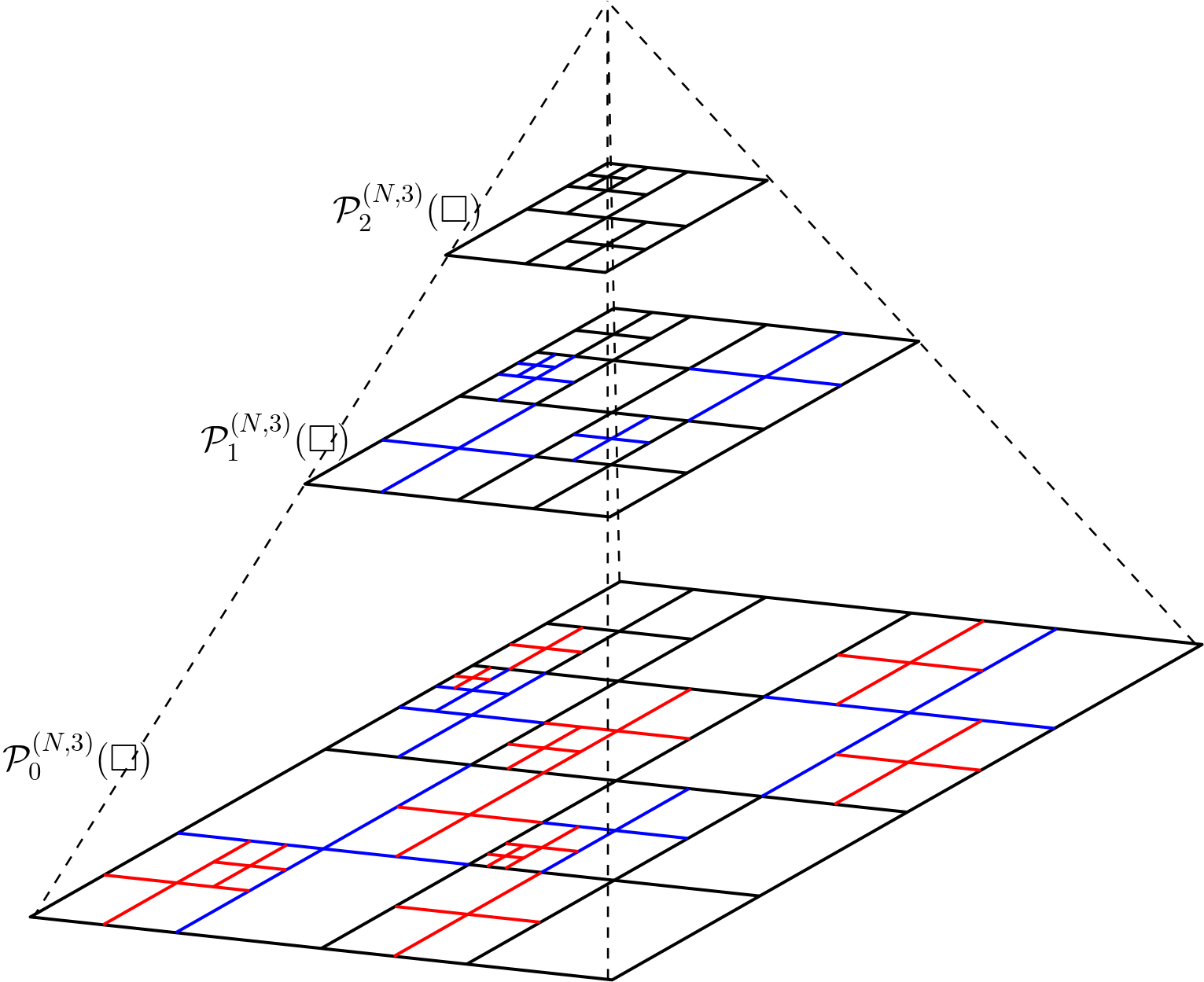}
		\caption{An illustration of the pyramid partitions $\pP^{(N,3)}_{2}(\square), \pP^{(N,3)}_{1}(\square), \pP^{(N,3)}_{0}(\square)$.}\label{fig.cube}
	\end{figure}
	
	\subsection{Basic properties of the pyramid partitions}\label{subsec:counting}
	In the end of this section, we give several basic properties about the pyramid partitions $(\pPN_{h})_{0 \leq h \leq k-1}$ in Definition~\ref{def.pyramid}. Recall Definition~\ref{def.N_stable} for the $N$-$\stable$ properties and the convention in Remark~\ref{rmk.N_stable}. We assume $k \in \N_+, N \in \N$ in all the statement without further explanation.

	Lemma~\ref{lem.maximalCluster} is a generalization of similar properties for the partition of the very-well-connected cubes in \cite{armstrong2018elliptic}.
	
	\begin{lemma}\label{lem.maximalCluster}
		Every $\cu \in \gGN_k$ is $N$-$\stable$ for the following property: for every $\cu', \cu'' \in \sSl_{\gGN_{k-1}}(\cu)$ with ${\dist(\cu',\cu'')\leq 1}$, there exists a cluster $\cC^\#$ such that
		\begin{equation}
			\label{e:neighbor_clusters}
			\cC^\#_*(\cu') \cup \cC^\#_*(\cu'') \subset \cC^\# \subset \cu' \cup \cu''.
		\end{equation} 
		Moreover, for any integer $0 \leq h \leq k-1$ and $\widehat{\cu} \in \pPN_{h}(\cu)$, we have
		\begin{equation}
			\label{e:inner_cluster_connectivity}
			\cC^\#_*(\widehat{\cu}) \subset	\cC^\#_*(\cu).
		\end{equation}
	\end{lemma}
	\begin{proof}
		We only apply Proposition~\ref{prop.good_cube_property} and Definition~\ref{def.pyramid} in the proof, which conveys the $N$-$\stable$ property. 
		
		The definition \eqref{e.pyrmid} yields a local partition of good cube $\sSl_{\gGN_{k-1}}(\cu)$. From (ii) of Proposition~\ref{prop.good_cube_property}, the cubes in $\gGN_{k-1}$ are very-well-connected satisfying \eqref{e:local_dependence} and \eqref{e:rarely_bad}  under admissible $\a^\#$, thus \eqref{e:neighbor_clusters} follows directly the proof of \cite[Lemma 2.8]{armstrong2018elliptic}. 
		
		For \eqref{e:inner_cluster_connectivity}, we prove it by induction from $h=(k-1)$ to $h=0$. The case $(k-1)$ is the basis. Applying \eqref{e:neighbor_clusters} under admissible $\a^\#$, we go over the partition cubes in $\pPN_{k-1}(\cu)$, their maximal clusters are connected to a cluster $\cC^\#$, which is crossing in $\cu$. Since $\cu \in \gGN_k$ is also well-connected, $\cC^\#$ has to be $\cC^\#_*(\cu)$ and this justifies \eqref{e:inner_cluster_connectivity} for $h= k-1$. For the induction step, suppose that \eqref{e:inner_cluster_connectivity} is already established for $(h+1)$ and we treat the case $h$. Given $\widehat{\cu} \in \pPN_{h}(\cu)$, then by the definition of pyramid partitions \eqref{e.pyrmid}, there exists its unique parent cube $\widehat{\cu}'$ such that 
		\begin{align*}
			\widehat{\cu}' \in \pPN_{h+1}(\cu), \qquad \widehat{\cu} \in \sSl_{\gGN_{h}}(\widehat{\cu}').
		\end{align*}
		Then the assumption \eqref{e:inner_cluster_connectivity} for $(h+1)$ implies
		\begin{align*}
			\cC^\#_*(\widehat{\cu}') \subset \cC^\#_*(\cu).
		\end{align*}
		Because $\widehat{\cu}' \in \gGN_{h+1}$, which is also robust when opening $N$ edges, the basic case of \eqref{e:inner_cluster_connectivity} applies to $\sSl_{\gGN_{h}}(\widehat{\cu}')$ that 
		\begin{align*}
			\cC^\#_*(\widehat{\cu}) \subset \cC^\#_*(\widehat{\cu}').
		\end{align*}
		In conclusion, we obtain that $\cC^\#_*(\widehat{\cu}) \subset \cC^\#_*(\widehat{\cu}')  \subset \cC^\#_*(\cu)$.
	\end{proof}
	
	The following lemma strengthens the result above:  we can determine whether or not a vertex is on the maximal cluster just by a slight enlargement of the partition cube. Let us explain the notations in the lemma.  Here the scaling constant $\rr$ is associated to good cubes $(\gGN_k)_{k \geq 0}$ defined in Proposition~\ref{prop.good_cube_property}, and $\OO(x)$ stands for the finite cluster containing $x$ as defined in Section~\ref{subsec.perco}, and $\cCp \Ll(\frac{4}{3} \cu' \cap \cu \Rr)$ is the boundary-connecting cluster defined in \eqref{eq.def_clt_boundary}.
	\begin{lemma}\label{lem.detouring_path}
		Every $\cu \in \gGN_k$ is $N$-$\stable$ for the statement: for every $\cu' \in \pPN_{h}(\cu) $ with $0 \leq h \leq k-1$, and every ${x \in \frac{16}{15}\cu' \cap \cu}$, the following properties hold.
		\begin{itemize}
			\item If $x \in \cC^\#_*(\cu)$, then $x \in \cCm^\# \Ll(\frac{4}{3} \cu'\Rr) \subset \cCm^\# \Ll(\rr \cu'\Rr)$.
			\item If $x \in \cCb^\#(\cu)$, then $x \in \cCp^\# \Ll(\frac{4}{3} \cu' \cap \cu \Rr) \subset \cCp^\# \Ll(\rr \cu' \cap \cu\Rr)$. As a corollary, we also have 
			\begin{align}\label{eq.clt_decomposition}
				\bigcup_{\cu' \in \pPN_{h}(\cu)} \cCp^\# \Ll(\frac{4}{3} \cu' \cap \cu \Rr) = \bigcup_{\cu' \in \pPN_{h}(\cu)} \cCp^\# \Ll(\rr \cu' \cap \cu \Rr) = \cCb^\# \Ll(\cu\Rr).
			\end{align}
			\item If $x \notin \cCb^\#(\cu)$, then $\OO^\#(x) \subset \frac{4}{3} \cu'$.
		\end{itemize}
	\end{lemma}
	\begin{proof}		
		All the arguments in the proof come from  Proposition~\ref{prop.good_cube_property}, Definition~\ref{def.pyramid}, and Lemma~\ref{lem.maximalCluster}, so this lemma is $N$-$\stable$.
		
		We study the first claim. 
		Under admissible $\a^\#$, thanks to \eqref{e:inner_cluster_connectivity}, $\cC^\#_*(\cu') \subset \cC^\#_*(\cu)$, then $\cC^\#_*(\cu')$ and $x$ should be connected in $\cu$. We claim this event is also realized in $\frac{4}{3}\cu'$. Otherwise, $x$ is connected to $\cC^\#_*(\cu')$ from outside $\frac{4}{3}\cu'$, so there must exists an open path $\gamma$ connecting $x$ to $\partial(\frac{4}{3}\cu')$. Then the proportion between the path $\gamma$ and $\frac{4}{3}\cu'$ has an estimate
		\begin{align}\label{eq.large_open_path}
			\frac{\vert \gamma\vert}{\size(\frac{4}{3} \cu')} \geq \frac{\frac{\frac{4}{3} - \frac{16}{15}}{2}\size( \cu')}{\frac{4}{3}\size( \cu')}  \geq \frac{1}{10}.
		\end{align}
		Here in the first inequality we consider the shortest possible situation from $\partial(\frac{4}{3}\cu')$ to $x$. Following (ii) in Proposition~\ref{prop.good_cube_property}, $\frac{4}{3} \cu'$ is also well-connected, so $\gamma$ should intersect $\cC^\#_*(\frac{4}{3}\cu')$ by (2) of Definition~\ref{def:well_connected}. Notice that $\cC^\#_*(\cu') \subset \cC^\#_*(\frac{4}{3}\cu')$, we result in that $x$ is connected to $\cC^\#_*(\cu')$ within $\frac{4}{3} \cu'$. 
		
		The inclusion $\cCm^\# \Ll(\frac{4}{3} \cu'\Rr) \subset \cCm^\# \Ll(\rr \cu'\Rr)$ is the result from Definition~\ref{def:well_connected} and $\rr \in (\frac{4}{3}, 2)$.
		
		For the second claim, if $x \in \cCm^\#(\cu)$, then it has been analyzed in the first claim. Otherwise, $x \in \cCb^\#(\cu) \setminus \cCm^\#(\cu)$ and it connects to the boundary of $\cu$. If this path is not contained in $\cCp^\# \Ll(\frac{4}{3} \cu' \cap \cu \Rr)$, then $x$ has to connect $\partial (\frac{4}{3} \cu') \cap \partial \cu$, which results in an open path like \eqref{eq.large_open_path} and yield contradiction. The inclusion $\cCp^\# \Ll(\frac{4}{3} \cu' \cap \cu \Rr) \subset \cCp^\# \Ll(\rr \cu' \cap \cu\Rr)$ is directly verified by the definition \eqref{eq.def_clt_boundary}. The observation above implies 
		\begin{align*}
			\cCb^\# \Ll(\cu\Rr) \subset \bigcup_{\cu' \in \pPN_{h}(\cu)} \cCp^\# \Ll(\frac{4}{3} \cu' \cap \cu \Rr).
		\end{align*}
		For the other side, the clusters connecting $\partial \cu$ in $ \cCp^\# \Ll(\frac{4}{3} \cu' \cap \cu \Rr)$  is naturally contained in $\cCb^\# \Ll(\cu\Rr)$, and the maximal clusters also have inclusion thanks to \eqref{e:inner_cluster_connectivity}.
		
		Concerning the third claim, if $\OO^\#(x) \nsubseteq \frac{4}{3} \cu'$, then it also contains an open path connecting $x$ to $\partial(\frac{4}{3}\cu')$ within $\OO^\#(x)$, which is a contradiction by the argument around \eqref{eq.large_open_path}.
		
	\end{proof}

	\bigskip
	In the pyramid partition $\pPN_{h}$ with $0 \leq h \leq k-2$, the comparability between neighbor partition cubes \eqref{e:neighbor_comparable_size} is not necessarily ensured because they can have different parent partition cubes in $\pPN_{h+1}$; see the discussion in Remark~\ref{rmk.counting}. This will bring some technical difficulty in our analysis. The following lemma called \emph{the counting argument} aims to replace the size comparability property, and will be frequently applied in the rest of paper.

	\begin{lemma}[Counting argument]\label{le:overlap_counting}
		For every real number $\aa \in (1,2)$, there exists a finite positive constant $C(\aa, d)$ such that, 
		the following estimate holds: for every $\cu \in \gGN_k$, we have
		\begin{equation}\label{eq.counting}
			\forall 0 \leq h \leq k-1, \quad \forall x \in \Zd, \qquad \sum_{\cu' \in \pPN_{h}(\cu)} \Ind{x \in \aa\cu'} \leq C^{k-h}.
		\end{equation}
	\end{lemma}
	\begin{proof} 
		The proof can be divided into 2 main steps using induction on $i$, and the first step requires a further induction on $\aa$. 
		
		\textit{Step~1: case $h=k-1$.} The case $h = k-1$ is specific, because the size comparability \eqref{e:neighbor_comparable_size} still holds. If the sum $\sum_{\cu' \in \pPN_{k-1}(\cu)} \Ind{x \in \aa\cu'}$ is zero, then it is trivial. Otherwise there exist some cube $\cu '' \in \pPN_{k-1}(\cu)$ such that $x \in \aa \cu''$. Then we relax a little the counting and observe the following inequality that 
		\begin{align}\label{eq.counting_relax}
			\sum_{\cu' \in \pPN_{k-1}(\cu)} \Ind{x \in \aa\cu'} \leq \sum_{\cu' \in \pPN_{k-1}(\cu)} \Ind{\aa\cu' \cap \aa \cu'' \neq \emptyset}.
		\end{align}
		This transforms the counting problem to the intersection problem between cubes in $\pPN_{k-1}(\cu)$. Roughly, as the neighbor partition cubes in $\pPN_{k-1}(\cu)$ are comparable, $\aa\cu' \cap \aa \cu'' \neq \emptyset$ means the distance between two partition cubes should not be very large. 
		
		We claim the following estimate, which will be justified later: for any  $\aa \in (1,2)$, there exists $L_\aa \in \N_+$ such that 
		\begin{align}\label{eq.counting_graph1}
			\forall \cu', \cu'' \in \pPN_{k-1}(\cu), \aa\cu' \cap \aa \cu'' \neq \emptyset \Longrightarrow \dist_{G_{k,k-1}(\cu)}(\cu', \cu'') \leq L_\aa.
		\end{align}
		Here we use the following graph structure $G_{k,h}(\cu)$:  $\pPN_{h}(\cu)$ serves as the vertex set, and two vertices are connected on $G_{k,h}(\cu)$ if and only if they are neighbor partitions cubes. Then we denote by $\dist_{G_{k,h}(\cu)}(\cdot, \cdot)$ for the distance on graph $G_{k,h}(\cu)$.

		Once we admit this estimate, then the {\lhs} of \eqref{eq.counting_relax} can be bounded by
		\begin{equation}\label{eq.counting_graph2}
			\#\{\cu' \in \pPN_{k-1}(\cu): x \in \aa \cu'\} \leq \# B(\cu'', L_\aa),
		\end{equation}
		where $B(\cu'', L_\aa)$ stands for a ball of radius $L_\aa$ and of center $\cu''$ in $G_{k,k-1}(\cu)$ 
		\begin{align}\label{eq.counting_ball}
			B(\cu'', L_\aa) := \{\cu' \in \pPN_{k-1}(\cu): \dist_{G_{k,k-1}(\cu)}(\cu', \cu'') \leq L_\aa\}.
		\end{align}
		Thanks to the size comparability \eqref{e:neighbor_comparable_size}, every partition cube in $\pPN_{k-1}(\cu)$ has at most $(2d \cdot 3^{d-1})$ neighbor partition cubes, thus every vertex in $G_{k,k-1}(\cu)$ has an upper bound for degree. Therefore, we obtain an upper bound for the volume of the ball $B(\cu'', L_\aa)$
		\begin{align}\label{eq.defCa}
			\# B(\cu'', L_\aa) \leq (2d \cdot 3^{d-1})^{L_\aa}.
		\end{align}
		We insert this to \eqref{eq.counting_graph2} and conclude the proof of case $i=k-1$. We also set the constant $C(\aa, d) := (2d \cdot 3^{d-1})^{L_\aa}$.
		
		It remains to verify the claim \eqref{eq.counting_graph1}.
		
		\smallskip
		\textit{Step~1.1: \eqref{eq.counting_graph1} for $\aa \in (1, \frac{5}{3}]$.}  If $\aa \leq \frac{5}{3}$, then $\aa\cu'$ is contained by the union of $\cu'$ and its neighbors following \eqref{e:neighbor_comparable_size}, which means
		\begin{equation}
			\aa\cu' \cap \aa\cu'' \neq \emptyset \Longrightarrow \dist_{G_{k,k-1}(\cu)}(\cu', \cu'') \leq 2.
		\end{equation}

		\smallskip
		\textit{Step 1.2:  \eqref{eq.counting_graph1} for $\aa \in (\frac{5}{3}, 2)$.} We generalize the argument in Step~1.1 and set
		\begin{equation}\label{eq.defan}
			\aa_1 := \frac{5}{3}, \qquad \aa_{n+1} := \frac{4}{3} + \frac{\aa_n}{3}.
		\end{equation}
		The main observation is that for any $\cu' \in \pPN_{k-1}(\cu)$,
		\begin{equation}\label{eq.neighborScale}
			\aa_{n+1} \cu' \subset \bigcup_{\widehat{\cu}' \in B(\cu',1)} \aa_{n} \widehat{\cu}'.
		\end{equation}
		Here $B(\cu',1)$ is as defined in \eqref{eq.counting_ball}, which counts $\cu'$ itself and its neighbor cubes. The observation \eqref{eq.neighborScale} is valid, because the {\rhs} covers a cube at least of the size 
		\begin{multline*}
			2\Ll(\frac{1}{2} \size(\cu') + \frac{1}{2}\cdot\frac{1}{3}\size(\cu') + \aa_n \cdot \frac{1}{2}\cdot\frac{1}{3}\size(\cu')\Rr) \\
			= \Ll(\frac{4}{3} + \frac{\aa_n}{3}\Rr) \size(\cu') = \aa_{n+1} \size(\cu').
		\end{multline*}
		\begin{figure}[b]
			\centering
			\includegraphics[height=100pt]{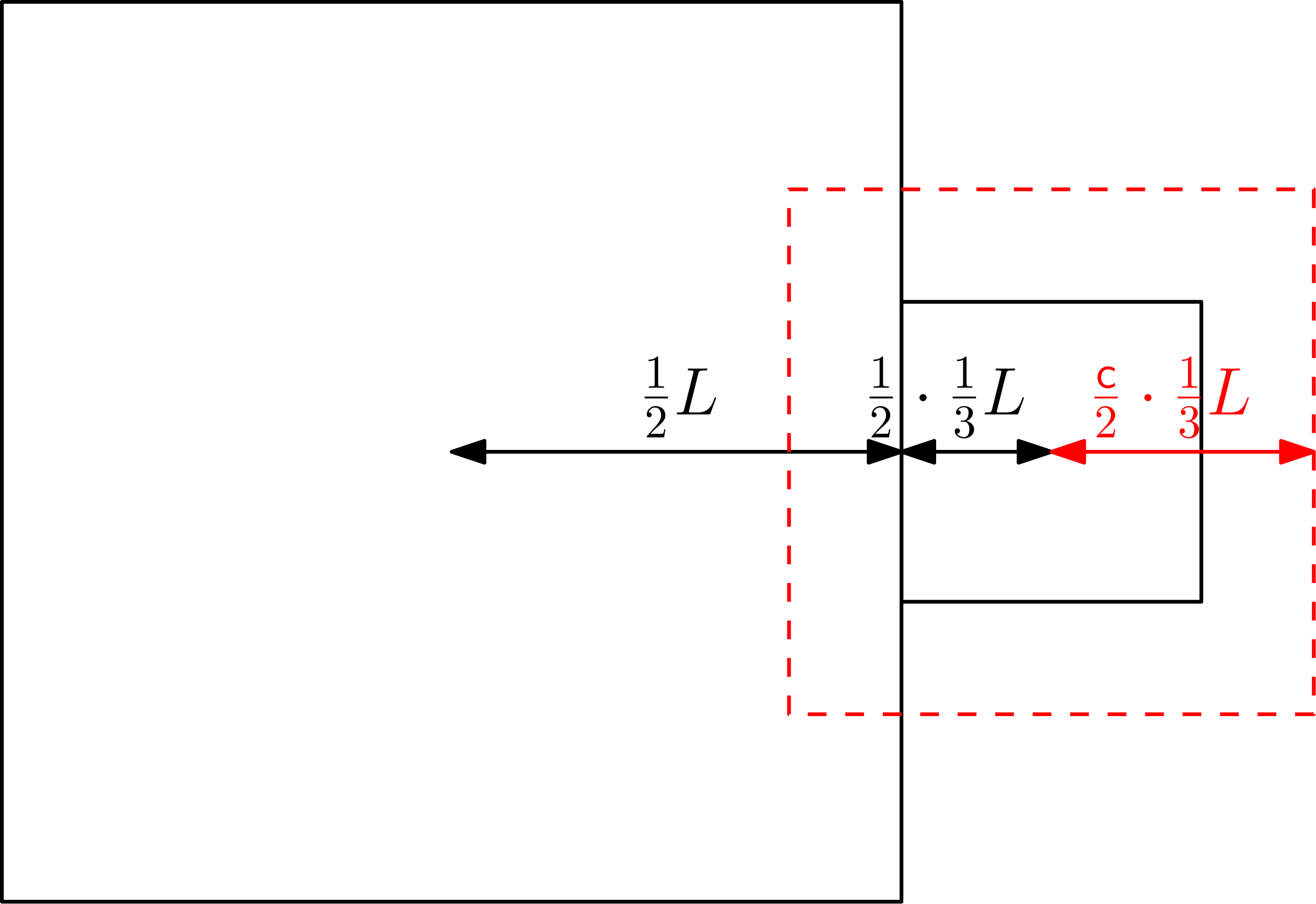}
			\caption{Illustration for the argument in \eqref{eq.neighborScale}.}\label{fig.neighborScale}
		\end{figure}
		See Figure~\ref{fig.neighborScale} for an illustration.
		
		We make an induction to prove \eqref{eq.counting_graph1} for $(\aa_n)_{n \in \N_+}$. The case $\aa_1$ is already proved in Step~1.1 and we  assume \eqref{eq.counting_graph1}  for $\aa_n$. Then for $\aa_{n+1}$ and every $\cu', \cu'' \in \pPN_{k-1}(\cu)$, we have
		\begin{align*}
			\aa_{n+1}\cu' \cap \aa_{n+1} \cu'' \neq \emptyset \Longrightarrow  \Ll(\bigcup_{\widehat{\cu}' \in B(\cu',1)} \aa_{n} \widehat{\cu}'\Rr) \bigcap \Ll(\bigcup_{\widehat{\cu}''  \in B(\cu'',1) } \aa_{n} \widehat{\cu}''\Rr) \neq \emptyset.
		\end{align*} 
		Here the relaxation of domain comes from  \eqref{eq.neighborScale}. This also implies 
		\begin{align*}
			\bigcup_{\widehat{\cu}' \in B(\cu',1)}\bigcup_{\widehat{\cu}''  \in B(\cu'',1) }  \aa_{n} \widehat{\cu}' \cap \aa_{n} \widehat{\cu}''  \neq \emptyset.
		\end{align*}
		For such neighbor partition cubes $\widehat{\cu}', \widehat{\cu}''$ with non-empty intersection after scaling, the assumption on \eqref{eq.counting_graph1} with $\aa_n$ applies, thus we have 
		\begin{align*}
			\dist_{G_{k,k-1}(\cu)}(\cu', \cu'') \leq 2 + \dist_{G_{k,k-1}(\cu)}(\widehat{\cu}', \widehat{\cu}'') \leq 2 + L_{\aa_n}.
		\end{align*}
		This induction implies that $L_{\aa_n} \leq 2n$, and the monotonicity justifies the existence of $L_\aa$ for all $\aa \in (1,2)$.

		\smallskip
		\textit{Step 2: case $0 \leq h \leq k-2$.} We have finished \eqref{eq.counting} for $h=k-1$, and for general $0 \leq h \leq k-2$, suppose that \eqref{eq.counting} is proved for $(h+1)$ with constant $C(\aa, d)$ defined in \eqref{eq.defCa}. Note that for every $\cu' \in \pPN_{h}(\cu)$ satisfying $x \in \aa \cu'$, his unique parent partition cube in $\pPN_{h+1}(\cu)$ must also contain $x$ after the scaling $\aa$. Therefore, we can count at first the scaled partition cube containing $x$ in $\pPN_{h+1}(\cu)$, then go further to count the children partition cubes in $\pPN_{h}(\cu)$
		\begin{equation*}
			\{\cu' \in \pPN_{h}(\cu): x\in \aa \cu'\} = \bigcup_{\widehat{\cu}'\in \pPN_{h+1}(\cu): x \in \aa\widehat{\cu}'} \{\cu'\in \sSl_{\G_h}(\widehat{\cu}'): x\in \aa \cu'\}.
		\end{equation*}
		Notice that 
		\begin{align*}
			\# \{\cu'\in \sSl_{\G_h}(\widehat{\cu}'): x\in \aa \cu'\} \leq C(\aa, d ),
		\end{align*}
		because it is just an application of the result in Step~1 to the partition $\pP^{(N, h+1)}_{h}(\widehat{\cu}')$. Thus, we use this bound to obtain 
		\begin{equation}\label{eq.countInduction}
			\begin{split}
				\#\{\cu' \in \pPN_{h}(\cu): x\in \aa \cu'\} &= \sum_{\widehat{\cu}'\in \pPN_{h+1}(\cu): x \in \aa\widehat{\cu}'} \#\{\cu'\in \sSl_{\G_h}(\widehat{\cu}'): x\in \aa \cu'\}\\
				&\leq C(\aa, d) \#\{\widehat{\cu}'\in \pPN_{h+1}(\cu): x \in \aa\widehat{\cu}'\},
			\end{split}
		\end{equation}
		and insert the assumption \eqref{eq.counting} with $(h+1)$ 
		\begin{align*}
			\#\{\widehat{\cu}'\in \pPN_{h+1}(\cu): x \in \aa\widehat{\cu}'\} \leq  C(\aa, d)^{k-(h+1)},
		\end{align*}
		then lemma is proved.
	\end{proof}

	\section{Weighted estimate for Poisson equation on clusters}\label{sec.Poisson} 
	This section focuses on the PDE analysis in the perturbed corrector equation \eqref{eq.recurrence_eq}. We extract its structure and reformulate it as 
	\begin{equation}\label{eq.PoissonCluster}
		\Ll\{
		\begin{array}{cc}
			-\na \cdot (\a^\# \na v_\#) = -\na\cdot \w_\#  &\text{ in }  {\itr (\cu_m)}, \\
			v_\# = 0 & \text{ on }  \partial \cu_m.
		\end{array}
		\Rr.
	\end{equation}
	Here we write $\a^\#$ following the convention \eqref{eq.sharp}. Differently, we denote by $(v_\#, \w_\#)$ the functions for the equation under $\a^\#$, but they are not necessarily equal to $(v(\a^\#), \w(\a^\#))$. Furthermore, $v_\# \in C_0(\cu_m)$ and $\w_\#$ is an anti-symmetric vector field introduced in Section~\ref{subsec.space}.
	
	Section~\ref{subsec.fieldEda} studies \eqref{eq.PoissonCluster} when $\w_\#$ is supported on $\overrightarrow{\Ed^{\a^\#}}(\cu_m)$, and then Section~\ref{subsec.fieldEd} handles more challenges when $\w_\#$ is supported on $\overrightarrow{\Ed}(\cu_m)$. They correspond respectively the case  $j=0$ and $j \geq 1$ in \eqref{eq.recurrence_eq}.


	\subsection{Source field on $\Eda$}\label{subsec.fieldEda}
	We state the weighted Dirichlet energy estimate on the intrinsic geometry of  percolation at first. In the statement, $\eps_{Me}$ is the exponent of Meyers' estimate in Proposition~\ref{prop.good_cube_property}.
	
	\begin{prop}
		\label{prop.weighted_L2}
		Given $\alpha \geq 0$ and $N \in \N$, then for all integers $h,k$ satisfying 
		\begin{align}\label{eq.condition_hk}
			h  \geq  \frac{(\alpha+1)(2+\eps_{Me})}{\eps_{Me}}, \qquad k \geq h+2,
		\end{align}
		there exists a finite positive constant $C = C(\p, d, N, h, k)$, such that every ${\cu_m \in \gGN_k}$ is $N$-$\stable$ for the statement: for $(v_\#, \w_\#)$ solving \eqref{eq.PoissonCluster} and $\w_\#$ an anti-symmetric vector field  supported on $\overrightarrow{\Ed^{\a^\#}}(\cu_m)$, then the following weighted estimate holds
		\begin{multline}\label{eq.weighted_L2}
			\sum_{\substack{\cu \in \pPN_{h}(\cu_m)}} |\cu|^{\alpha}\sum_{e\in \EE\Ll(\cCp^\#(\rr\cu \cap \cu_m)\Rr)} |\na v_\#(e)|^2\\
			\leq  C\sum_{\substack{\cu \in \pPN_{h+1}(\cu_m)}}|\cu|\sum_{e\in \EE\Ll(\cCp^\#(\rr\cu \cap \cu_m)\Rr)} |\w_\#(e)|^2.
		\end{multline}
	\end{prop}
	\begin{remark}
		The sum $\sum_{\cu \in \pPN_{h+1}(\cu_m)} \sum_{e\in \EE\Ll(\cC_*(\rr\cu)\Rr)}$ in \eqref{e:key_estimate} will appear frequently in the paper. We need to shrink the cube as $\rr\cu$ in order cover the bonds $\EE\Ll(\cCb(\cu_m)\Rr)$; see Lemma~\ref{lem.detouring_path} and \eqref{eq.clt_decomposition}. Otherwise, if we do the sum only over $\cC_*(\cu)$ for each term, some bonds of $\cC_*(\cu_m)$ can be missing.
		
	\end{remark}
	
	\begin{proof}
		The proof only uses Proposition~\ref{prop.good_cube_property}, Definition~\ref{def.pyramid} and the properties in Section~\ref{subsec:counting}, so it is also $N$-$\stable$. We skip the superscript or subscript $\#$ for convenience. The proof can be divided into three steps.

		\textit{Step 1: up-grade to coarsened good cube.} By definition of the partitions $\pPN_{h+1}(\cu_m)$ in \eqref{e.pyrmid}, we do the sum on the {\lhs} of \eqref{eq.weighted_L2} at first over $\pPN_{h+1}(\cu_m)$, and then subdivide each cube into the partition of good cube $\gGN_{h}$
		\begin{multline}\label{eq.L2Subdivision}
			\sum_{\substack{\cu' \in \pPN_{h}(\cu_m)}} \sum_{e\in \EE\Ll(\cCp(\rr\cu' \cap \cu_m)\Rr)} \Big(\cdots\Big) \\
			= \sum_{ \substack{\cu \in \pPN_{h+1}(\cu_m) }}\sum_{\substack{\cu' \in \sSl_{\gGN_{h}}(\cu)}} \sum_{e\in \EE(\cCp(\rr\cu' \cap \cu_m))} \Big(\cdots\Big).
		\end{multline}

		Our first step aims to up-grade the weighted sum over $\pPN_{h}(\cu_m)$ on the {\lhs} of \eqref{eq.weighted_L2} to a normal sum over $\pPN_{h+1}(\cu_m)$. We focus on one term $\cu \in \pPN_{h+1}(\cu_m)$ and apply the H\"older inequality with $\eps = \eps_{Me}$ in Definition~\ref{def.Meyers}, then we obtain that
		\begin{equation}\label{eq.PoissonOneTerm}
			\begin{split}
				&\sum_{\substack{\cu' \in \sSl_{\gGN_{h}}(\cu)}}|\cu'|^{\alpha}	 \sum_{e\in \EE(\cCp(\rr\cu' \cap \cu_m))} |\na v(e)|^2\\
				&\leq \sum_{\substack{\cu' \in \sSl_{\gGN_{h}}(\cu)}}|\cu'|^{\alpha+1}	 \Ll(\frac{1}{\vert \cu'\vert}\sum_{e\in \EE(\cCp(\rr\cu' \cap \cu_m))} |\na v(e)|^2\Rr)\\
				& \leq \vert \cu \vert \cdot \mathbf{I} \cdot \mathbf{II},
			\end{split}
		\end{equation}
		with two terms defined as 
		\begin{align*}
			\mathbf{I} &= \Ll(\frac{1}{\vert \cu \vert}\sum_{\substack{\cu' \in \sSl_{\gGN_{h}}(\cu)}}|\cu'|^{\frac{(\alpha+1)(2+\eps)}{\eps}}\Rr)^{\frac{\eps}{2+\eps}},\\
			\mathbf{II} &= \Ll(\frac{1}{\vert \cu \vert}\sum_{\substack{\cu' \in \sSl_{\gGN_{h}}(\cu)}}
			\Ll(\frac{1}{\vert \cu'\vert}\sum_{e\in \EE(\cCp(\rr\cu' \cap \cu_m))} |\na v(e)|^2\Rr)^{\frac{2+\eps}{2}}\Rr)^{\frac{2}{2+\eps}}.
		\end{align*}
		
		One needs more analysis on these two terms. Thanks to the property (1) in Proposition~\ref{prop.good_cube_property}, the condition $\cu \in \gGN_{h+1}$ implies $\cu \in  \Lambda\Ll(\gGN_{h}, d + h + 1, C_{Sc,h+1}\Rr)$. Moreover, the condition $h \geq  \frac{(\alpha+1)(2+\eps)}{\eps}$ in \eqref{eq.condition_hk} ensures enough integrability in \eqref{eq.defgGpc}, which yields a bound  for the term $\mathbf{I}$
		\begin{align}\label{eq.cubeBound}
			\mathbf{I}  \leq \Ll(\frac{1}{\vert \cu \vert}\sum_{\cu' \in \sSl_{\gGN_{h}}(\cu)}|\cu'|^{\frac{(\alpha+1)(2+\eps)}{\eps}}\Rr)^{\frac{\eps}{2+\eps}} \leq (C_{Sc,h+1})^{\frac{\eps}{2+\eps}}.
		\end{align}

		For the term $\mathbf{II}$, we have the following estimate
		\begin{equation}\label{eq.cubeBound2}
			\begin{split}
				\mathbf{II} &\leq \Ll(\frac{1}{\vert \cu \vert}\sum_{\substack{\cu' \in \sSl_{\gGN_{h}}(\cu)}}
				\frac{1}{\vert \cu'\vert}\sum_{e\in \EE(\cCp(\rr\cu' \cap \cu_m))} |\na v(e)|^{2+\eps}\Rr)^{\frac{2}{2+\eps}}\\
				&\leq \Ll(\frac{1}{\vert \cu \vert}\sum_{\substack{\cu' \in \sSl_{\gGN_{h}}(\cu)}}
				\sum_{e\in \EE(\cCp(\rr\cu' \cap \cu_m))} |\na v(e)|^{2+\eps}\Rr)^{\frac{2}{2+\eps}}\\
				&\leq \Ll(\frac{2C(\rr,d)}{\vert \cu \vert}
				\sum_{e\in \EE(\cCp(\frac{4}{3}\cu \cap \cu_m))} |\na v(e)|^{2+\eps}\Rr)^{\frac{2}{2+\eps}}.
			\end{split}
		\end{equation}
		In the first line, we apply Jensen's inequality to the sum in the interior. In the second line, 
		we drop the factor $\frac{1}{\vert \cu'\vert}$  because $|\cu'|\geq 1$. Then we exchange the order of sum. The last line makes use of the observation (which will be justified in the next paragraph)
		\begin{align}\label{eq.43inclusion}
			\cCp(\rr\cu' \cap \cu_m) \subset \cCp(\frac{4}{3}\cu \cap \cu_m).
		\end{align}
		We also count how many times every edge is covered, which can be answered by Lemma \ref{le:overlap_counting}
		\begin{align*}
			\forall e \in \Ed, \qquad	\sum_{\substack{\cu' \in \sSl_{\gGN_{h}}(\cu)}} \Ind{e \in \EE(\cCp(\rr\cu' \cap \cu_m))} \leq 2 C(\rr, d).
		\end{align*}
		
		We now justify \eqref{eq.43inclusion}. Observing Lemma~\ref{lem.maximalCluster} and Lemma~\ref{lem.detouring_path}, we obtain
		\begin{equation}\label{eq.choice43}
			\cC_*(\rr\cu') \xleftrightarrow{\a} \cC_*(\cu') \xleftrightarrow{\a} \cC_*(\cu) \xleftrightarrow{\a} \cC_*\Ll(\frac{4}{3}\cu\Rr).
		\end{equation} 
		Moreover, the size comparison in the maximal cluster in \eqref{eq.sizeMax} implies every maximal cluster in the scaled cube $\rr \cu'$ in $\sSl_{\gGN_{h}}(\cu)$ is actually contained in that of $\frac{4}{3}\cu$. Similarly, $\rr \cu' \subset  \frac{4}{3}\cu$ also implies the inclusion of the clusters connecting to the boundary $\partial \cu_m$. Thus, the sum in the last line of \eqref{eq.cubeBound2} is over $\EE\Ll(\cCp(\frac{4}{3}\cu \cap \cu_m)\Rr)$. 
		
		In conclusion, finally in Step~1 we coarsen the partition cube, and make use of the control of the partition cube size to handle the extra weight in \eqref{eq.weighted_L2}. The price is a larger exponent $(2+\epsilon)$. But until the moment, we did not use the elliptic equation \eqref{eq.PoissonCluster} and the procedure above works for any general function. In the next step,  we will make use of Meyers' estimate to reduce this exponent. 

		\smallskip
		\textit{Step 2: Meyers' estimate.} 
		For each $\cu \in \pPN_{h+1}(\cu_m)$, as $\cu \in \gGN_{h+1} \subset \gGN_{0}$, and $\w$ is supported on $\EE(\cCp(\rr\cu \cap \cu_m))$,  following Proposition~\ref{prop.good_cube_property}, Meyers' estimate in Definition~\ref{def.Meyers} works for $v$
		\begin{equation}\label{eq.PoissonOneTerm_Meyers}
			\begin{split}
				&\Ll(\frac{1}{|\frac{4}{3}\cu|}\sum_{e\in \EE(\cCp(\frac{4}{3}\cu \cap \cu_m))}|\na v(e)|^{2+\eps}\Rr)^{\frac{2}{2+\eps}}\\
				&\leq C^2_{Me} \Ll[\frac{1}{|\rr\cu|}\sum_{e\in \EE\Ll(\cCp(\rr\cu \cap \cu_m)\Rr)}|\na v(e)|^2 + \Ll(\frac{1}{|\rr\cu|}\sum_{e\in \EE\Ll(\cCp(\rr\cu \cap \cu_m)\Rr)}|\w(e)|^{2+\eps}\Rr)^{\frac{2}{2+\eps}}\Rr]\\
				&\leq C^2_{Me} \Ll[\frac{1}{|\rr\cu|}\sum_{e\in \EE\Ll(\cCp(\rr\cu \cap \cu_m)\Rr)}|\na v(e)|^2 + \frac{1}{|\rr\cu|^{\frac{2}{2+\eps}}}\sum_{e\in \EE\Ll(\cCp(\rr\cu \cap \cu_m)\Rr)}|\w(e)|^{2}\Rr],
			\end{split}
		\end{equation}
		where in the third line we use the bound $(\sum_j a_j)^{t} \leq (\sum_j a_j^t)$ for positive $a_j$ and $t \in [0,1]$; in our case we we set $t=\frac{2}{2+\eps}$. 
		
		We put this estimate back to \eqref{eq.cubeBound2}, together with \eqref{eq.cubeBound}, we obtain the estimate of $\mathbf{I, II}$ in \eqref{eq.PoissonOneTerm}. Then the sum over one term $\cu \in \pPN_{h+1}(\cu_m)$ is 
		\begin{multline}\label{eq.PoissonOneTerm2}
			\sum_{\substack{\cu' \in \sSl_{\gGN_{h}}(\cu)}}|\cu'|^{\alpha}	 \sum_{e\in \EE(\cCp(\rr\cu' \cap \cu_m))} |\na v(e)|^2\\
			\leq C\Ll(\sum_{e\in \EE\Ll(\cCp(\rr\cu \cap \cu_m)\Rr)}|\na v(e)|^2 + \vert \cu \vert\sum_{e\in \EE\Ll(\cCp(\rr\cu \cap \cu_m)\Rr)}|\w(e)|^{2}\Rr).
		\end{multline}
		Notice that the factor $\vert \cu \vert$ in the last line of \eqref{eq.PoissonOneTerm} compensates that from the normalization in front of $\na v$ in Meyers' estimate, and we just relax the volume in front of $\w$.

		\smallskip
		\textit{Step 3: basic Dirichlet energy estimate.} Now we sum up \eqref{eq.PoissonOneTerm2}. The term $\w$ is of the correct expression as desired, and we need to apply the counting argument again for the term of $v$
		\begin{align*}
			&\sum_{\substack{\cu \in \pPN_{h+1}(\cu_m)}} \sum_{e\in \EE\Ll(\cCp(\rr\cu \cap \cu_m)\Rr)} |\na v(e)|^2  \\
			&= \sum_{e\in \EE(\cC_*(\cu_m))}|\na v(e)|^2  \Ll(\sum_{\substack{\cu \in \pPN_{h+1}(\cu_m)}} \Ind{e \in \EE\Ll(\cCp(\rr\cu \cap \cu_m)\Rr)}\Rr)\\
			&\leq 2C(\rr, d) \sum_{e\in \EE(\cCb(\cu_m))}|\na v(e)|^2 . 
		\end{align*}
		From the first line to the second line, we change the order of sum with \eqref{eq.clt_decomposition}. Then we apply Lemma~\ref{le:overlap_counting} from the second line to the third line.

		We finally need to give a Dirichlet energy estimate for the elliptic equation on the maximal cluster. We test \eqref{eq.PoissonCluster} with $v$  
		\begin{align*}
			\bracket{\na v, \a \na v}_{\cu_m}  = \bracket{\na v, \w}_{\cu_m} = \bracket{\na v, \a \w}_{\cu_m}.
		\end{align*}
		We are allowed to add $\a$ in the third equation above, because $\w$ is supported on $\overrightarrow{\Eda}$. Then  by Young's inequality we will get
		\begin{equation*}
			\bracket{\na v, \a \na v}_{\cu_m} \leq \bracket{\w, \a \w}_{\cu_m}.
		\end{equation*}
		This concludes the proof.
	\end{proof}
	\begin{remark}
		In the step \eqref{eq.PoissonOneTerm_Meyers} above, we apply a naive bound ${(\sum_j a_j)^{t} \leq (\sum_j a_j^t)}$ for $t \in [0,1]$ to $\w$, but we do need the Meyers' estimate for $v$. The key contribution of Meyers' estimate is to reduce the weight to $1$ in front of $v$  in \eqref{eq.PoissonOneTerm2}, otherwise we cannot close the estimate using the basic Dirichlet energy in Step~3.
	\end{remark}
	
	\subsection{Source field on $\Ed$}\label{subsec.fieldEd}
	Notice that the source term $\W_m(F,j)$ in our perturbed corrector equation \eqref{eq.recurrence_eq} is not necessarily supported on $\Ed^{\a^F}$ when $j \geq 1$, thus we develop a generalized version of Proposition~\ref{prop.weighted_L2}, in which	$\w_\#$ is a vector field  supported on $\overrightarrow{\Ed}(\cu_m)$. For this case, we also need a \emph{coarse-controlled function}. We call function $\g_\# :\overrightarrow{\Ed}(\cu_m) \to \R$ a coarse-controlled function for $v_\#$, if it satisfies the following statement when $\cu_m \in \gGN_{k}$ 
	\begin{multline}\label{eq.vg}
		\forall \cu \in \pPN_0(\cu_m), \quad {and} \quad \forall e \cap \cu \neq \emptyset, \\
		\vert \nabla v_\#(e) \vert \leq  \vert \g_\#(e) \vert + \sum_{e' \in \cCp^\#(\rr \cu \cap \cu_m)} \vert \nabla v_\#(e')\vert. 	
	\end{multline}
	
	
	In the following statement of the generalization of Proposition~\ref{prop.weighted_L2}, the two positive constants $\tilde{\eps}_{Me}$ and $\tilde{h}_{Me}$ will be further explained in Lemma~\ref{lem.Meyers_NEW}.
	\begin{prop}\label{prop.weighted_L2_NEW}
		Given $\alpha \geq 0$ and $N \in \N$, then for all integers $h,k$ satisfying 
		\begin{align}\label{eq.condition_hk_NEW}
			h \geq  \max \Ll\{\frac{(\alpha+3)(2+\tilde{\eps}_{Me})}{\tilde{\eps}_{Me}}, \tilde{h}_{Me} \Rr\}, \qquad k \geq h+2,
		\end{align}
		there exists a finite positive constant $C = C(d, N, h, k)$, such that every ${\cu_m \in \gGN_k}$ is $N$-$\stable$ for the statement: for every $(v_\#, \w_\#, \g_\#)$ satisfying \eqref{eq.PoissonCluster} and \eqref{eq.vg}, and $\w_\#$ an anti-symmetric vector field  supported on $\overrightarrow{\Ed}(\cu_m)$, then the following weighted estimate holds
		\begin{multline}\label{eq.weighted_L2_NEW}
			\sum_{\substack{\cu \in \pPN_{h}(\cu_m)}} |\cu|^{\alpha}\sum_{e\in \EE\Ll(\cCp^\#(\rr\cu \cap \cu_m)\Rr)} |\na v_\#(e)|^2\\
			\leq  C\sum_{\substack{\cu \in \pPN_{h+1}(\cu_m)}}|\cu|^3\sum_{e\in \Ed\Ll(\rr\cu \cap \cu_m\Rr)} \Ll(|\w_\#(e)|^2 + \vert \g_\#(e) \vert^2\Rr).
		\end{multline}
	\end{prop}
	
	Let us explain the importance of the coarse-controlled function $\g_\#$ here. We notice the different supports on the two sides of \eqref{eq.PoissonCluster}. Especially, because the {\lhs} of \eqref{eq.PoissonCluster} does not have the uniform ellipticity, $v_\#$ cannot be uniquely determined, and we need an auxiliary function $\g_\#$ and \eqref{eq.vg} to calibrate it.

	As we have seen, one important input in the proof of Proposition~\ref{prop.weighted_L2} is  Meyers' estimate on percolation cluster,  developed in \cite{armstrong2018elliptic} and \cite{dario2021corrector} and assumed for $\gGN_{0}$. Similarly, we also need a \emph{coarse-grained  Meyers' estimate}, stated in Lemma~\ref{lem.Meyers_NEW}, to prove Proposition~\ref{prop.weighted_L2_NEW}. As it involves the controlled function $\g$ and a weight function $A$, we give the details of the proof. In particular, the weight $A^{-1}$ is usually chosen to be very small, so this inequality is still non-trivial. Similar spirit roots from Widman's hole-filling technique, and an application can be found in \textit{the modified Caccioppoli inequality} of \cite[Proposition~3.6]{bulk}.
	\begin{lemma}[Coarse-grained Meyers' estimate]\label{lem.Meyers_NEW}
		There exist finite positive constants $\tilde{\eps}_{Me}, \tilde{C}_{Me}$ and $\tilde{h}_{Me} \in \N_+$, such that for every $\cu \subset \cu_m$ satisfying $\cu \in \gGN_{\tilde{h}_{Me}}$, every $(v_\#, \w_\#, \g_\#)$ satisfying \eqref{eq.PoissonCluster} and \eqref{eq.vg}, and every weight function ${A:\Ed(\cu_m) \to [1, \infty)}$, one has the estimate
		\begin{multline}\label{eq.Meyers_New}
			\norm{\nabla v_\# \Ind{\a^\#(\cdot) \neq 0} }_{\underline{L}^{2+\tilde{\eps}_{Me}}(\frac{4}{3} \cu \cap \cu_m)} 
			\leq \tilde{C}_{Me}\norm{\nabla v_\# \Ind{\a^\#(\cdot) \neq 0} }_{\underline{L}^{2}(\rr \cu \cap \cu_m)} \\ + \tilde{C}_{Me} \norm{\vert A^{-1}\nabla v_\# \Ind{\a^\#(\cdot) \neq 0}\vert + \vert \cu \vert\vert A\w_\# \vert +\vert \g_\# \vert}_{\underline{L}^{2+\tilde{\eps}_{Me}}(\rr \cu \cap \cu_m)}.
		\end{multline}
	\end{lemma}
	\begin{proof}
		The proof of Meyers' estimate starts from the Caccioppoli inequality, and then the Sobolev--Poincar\'e inequality yields a reverse H\"older inequality. Afterward, Gehring’s lemma will improve the integrability.  In the proof, we just assume $\rr \cu \subset \cu_m$ for convenience, as the version near the boundary is similar. We also skip the superscript and subscript in the proof.
		
		\textit{Step~1: the Caccioppoli inequality.} 	We start from a cube $\cu' \in \gGN_{0}$ and also propose a cutoff function $\eta$, which vanishes outside $\cu'$ but may keep constant in $\frac{1}{2} \cu'$. We test \eqref{eq.PoissonCluster} with $\eta^2 ( v - c)$, then we have
		\begin{align*}
			\bracket{\nabla (\eta^2 ( v - c)), \a \nabla v}_{\cu'} = \bracket{\nabla (\eta^2 ( v - c)), \w}_{\cu'}.
		\end{align*}
		We denote by $R:= \size(\cu')$, then we obtain 
		\begin{align}\label{eq.Caccioppoli_NEW}
			\norm{ \eta \nabla v \Ind{\a(\cdot) \neq 0}}^2_{\aL^2(\cu')} 
			\leq \frac{1}{R^2} \norm{(v - c)}^2_{\aL^2(\cu')} +  \norm{\tilde{A} \w }^2_{\aL^2(\cu')} +  \norm{ \tilde{A}^{-1} \eta \nabla v}^2_{\aL^{2}(\cu')}.
		\end{align}
		Here $\tilde{A} \geq 1$ is a pre-weight function, which is related to $A$ and to be fixed later.
		\smallskip
		
		\textit{Step~2: the reverse H\"older inequality.} Choosing $c = (v)_{\cu'}$, then we apply the Sobolev--Poicar\'e inequality to the term $\norm{(v - c)}^2_{\aL^2(\cu')}$ with $2_*= \frac{2d}{d+2}$ and obtain 
		\begin{align*}
			\frac{1}{R^2} \norm{(v - c)}^2_{\aL^2( \cu')} \leq C_{2,d} \norm{\nabla v}^2_{\aL^{2_*}( \cu')}
		\end{align*}
		The controlled function $\g$ is applied to the term $\norm{\nabla v}^2_{\aL^{2}( \cu')}$ and $\norm{ \tilde{A}^{-1} \eta \nabla v}^2_{\aL^{2}(\cu')}$ with the coarse-graining argument in \cite[Lemma~3.6]{armstrong2018elliptic}: 
		\begin{align}\label{eq.coarse_grain_control}
			\begin{split}
				&\sum_{e \in \Ed(\cu')}\vert \nabla v (e)\vert^{2_*} \\
				&\leq  \sum_{e \in \Ed(\cu')}\Ll(\vert \g(e) \vert^{2_*} + \vert \rr \cu' \vert^{2_*}  \Ll(\frac{1}{\vert \rr \cu' \vert}\sum_{e' \in \cCp(\rr \cu')} \vert \nabla v(e')\vert\Rr)^{2_*}\Rr) \\
				&\leq \sum_{e \in \Ed(\cu')} \Ll(\vert \g(e) \vert^{2_*} + \vert \rr \cu' \vert^{2_*}  \Ll(\frac{1}{\vert \rr \cu' \vert}\sum_{e' \in \cCp(\rr \cu')} \vert \nabla v(e')\vert^{2_*} \Rr)\Rr)\\
				&= \sum_{e \in \Ed(\cu')} \vert \g(e) \vert^{2_*} + \vert \rr \cu' \vert^{2_*}  \sum_{e' \in \cCp(\rr \cu')} \vert \nabla v(e')\vert^{2_*}. 
			\end{split}
		\end{align}
		Here Jensen's inequality is applied from the first line to the second line. To absorb the weight $\vert \rr \cu' \vert^{2_*}$, we choose a large scale $h'$, and further assume that $\cu' \in \gGN_{h'} \subset \gGN_{0}$. Then we coarsen the domain and apply \eqref{eq.Caccioppoli_NEW} to $\cu'' \in \gGN_{h'+1}$, where the term $\norm{\nabla v}_{\aL^{2_*}( \cu'')}$ can be treated as
		\begin{align*}
			\norm{\nabla v}_{\aL^{2_*}( \cu'')} &= \Ll(\frac{1}{\vert \cu ''\vert}\sum_{\cu' \in \sSl_{\gGN_{h'}}(\cu'')} \sum_{e \in \Ed(\cu')}\vert \nabla v (e)\vert^{2_*}\Rr)^{\frac{1}{2_*}}\\
			&\leq  \norm{\g}_{\aL^{2_*}( \cu'')} + \Ll(\frac{1}{\vert \cu ''\vert}\sum_{\cu' \in \sSl_{\gGN_{h'}}(\cu'')} \vert \rr \cu' \vert^{2_*}  \sum_{e \in \Ed(\rr \cu')} \vert \nabla v(e)\vert^{2_*}\Ind{\a(e) \neq 0}\Rr)^{\frac{1}{2_*}}.
		\end{align*}
		We apply H\"older's inequality to the weighted sum, with $\frac{1}{s} + \frac{1}{s'} = 1, s>1$ to be determined later 
		\begin{equation}\label{eq.coarse_grain_weight_argument}
			\begin{split}
				&\Ll(\frac{1}{\vert \cu ''\vert}\sum_{\cu' \in \sSl_{\gGN_{h'}}(\cu'')} \vert \rr \cu' \vert^{2_*}  \sum_{e \in \Ed(\rr \cu')} \vert \nabla v(e)\vert^{2_*}\Ind{\a(e) \neq 0}\Rr)^{\frac{1}{2_*}}\\
				& \leq \Ll(\frac{1}{\vert \cu ''\vert}\sum_{\cu' \in \sSl_{\gGN_{h'}}(\cu'')} \vert \rr \cu' \vert^{(2_* + 1)s'}  \Rr)^{\frac{1}{2_* s'}} \\
				& \qquad \times \Ll(\frac{1}{\vert \cu''\vert}\sum_{\cu' \in \sSl_{\gGN_{h'}}(\cu'')} \Ll(\frac{1}{\vert \rr \cu'\vert} \sum_{e \in \Ed(\rr \cu')} \vert \nabla v(e)\vert^{2_*}\Ind{\a(e) \neq 0}\Rr)^s  \Rr)^{\frac{1}{2_* s}}\\
				&\leq C  \Ll(\frac{1}{\vert \cu''\vert}\sum_{\cu' \in \sSl_{\gGN_{h'}}(\cu'')} \frac{1}{\vert \rr \cu'\vert} \sum_{e \in \Ed(\rr \cu')} \vert \nabla v(e)\vert^{2_* s}\Ind{\a(e) \neq 0} \Rr)^{\frac{1}{2_* s}}\\
				&\leq C \norm{\nabla v \Ind{\a(\cdot) \neq 0}}_{\aL^{2_* s}( \frac{4}{3}\cu'')}.
			\end{split}
		\end{equation}
		From the third line to the forth line, Jensen's inequality is applied once again. We hope $2_* s < 2$, so this is still a reverse H\"older inequality. This implies the choice
		\begin{align*}
			s = \frac{2+2_*}{2 \cdot 2_*} \Longrightarrow 2_* s = \frac{2+2_*}{2} < 2.
		\end{align*}
		The moment condition (1) of Proposition~\ref{prop.good_cube_property} is also applied to the third line of \eqref{eq.coarse_grain_weight_argument}, so we make the following choice of $h'$ to ensure the integrability
		\begin{align*}
			h' := \lfloor (2_* + 1)s' \rfloor + 1.
		\end{align*} 
		In the last line, we also shrink the domain $\rr \cu' \subset \frac{4}{3}\cu''$ thanks to the fact \eqref{eq.sizeMax}. Especially, we should notice that, there is no hope to compensate the outcome of \eqref{eq.coarse_grain_weight_argument} with the {\lhs} of \eqref{eq.Caccioppoli_NEW}, because the coarse-graining argument uses the information outside $\cu'$ there. 
		
		A similar calculation in \eqref{eq.coarse_grain_control} is also applied to the term $\norm{ \tilde{A}^{-1} \eta \nabla v}^2_{\aL^{2}(\cu')}$ with $\cu' \in \gGN_{h'}$, which yields
		\begin{align*}
			\sum_{e \in \Ed(\cu')}\vert \tilde{A}^{-1} \eta \nabla v (e)\vert^{2} \leq \sum_{e \in \Ed(\cu')} \vert \tilde{A}^{-1} \g(e) \vert^{2} + \vert \rr \cu' \vert^{2}  \sum_{e' \in \cCp(\rr \cu')} \vert \tilde{A}^{-1} \nabla v(e')\vert^{2}. 
		\end{align*}
		Choosing $\tilde{A} = \vert \cu' \vert A$, then for every $\cu'' \in \gGN_{h''+1}, \frac{4}{3} \cu'' \subset \cu_m$, we have a reverse H\"older inequality
		\begin{multline}\label{eq.revere_Holder}
			\norm{ \nabla v \Ind{\a(\cdot) \neq 0}}_{\aL^2(\cu'')} \leq K \norm{\nabla v \Ind{\a(\cdot) \neq 0}}_{\aL^{\frac{2+2_*}{2}}(\frac{4}{3} \cu'')} \\
			+ C\Ll(\norm{A^{-1} \nabla v \Ind{\a(\cdot) \neq 0} }_{\aL^{2}(\frac{4}{3} \cu'')}
			+ \norm{\vert \cu''\vert A \w }_{\aL^2(\frac{4}{3} \cu'')} +  \norm{\g }_{\aL^2(\frac{4}{3} \cu'')} \Rr).
		\end{multline}
		\smallskip
		
		\textit{Step~3: Gehring’s lemma.} We aim to apply Gehring’s lemma to \eqref{eq.revere_Holder}. Under the setting $\w = 0$, then \eqref{eq.revere_Holder} becomes
		\begin{align*}
			\norm{ \nabla v \Ind{\a(\cdot) \neq 0}}_{\aL^2(\cu'')}
			\leq K \norm{\nabla v \Ind{\a(\cdot) \neq 0}}_{\aL^{\frac{2+2_*}{2}}(\frac{4}{3} \cu'')} + K \norm{\g }_{\aL^2(\frac{4}{3} \cu'')}.
		\end{align*}
		Then the argument of the coarsened function in \cite[Proposition~3.8]{armstrong2018elliptic} helps get rid of the minimal scale, and the classical Gehring’s lemma applies: there exists finite positive constants $C, \epsilon$ and integer $\tilde{h}_{Me} \in \N_+$, such that
		for $\cu \in \gGN_{\tilde{h}_{Me}}$, we have 
		\begin{align*}
			\norm{ \nabla v \Ind{\a(\cdot) \neq 0}}_{\aL^{2+\epsilon}(\frac{4}{3} \cu)}
			\leq C \norm{\nabla v \Ind{\a(\cdot) \neq 0}}_{\aL^{2}(\rr \cu)} + C \norm{\g }_{\aL^{2+\epsilon}(\rr \cu)}.
		\end{align*}

		For the general case $\w \neq 0$ is, we only keep the term $K \norm{\nabla v \Ind{\a(\cdot) \neq 0}}_{\aL^{\frac{2+2_*}{2}}(\frac{4}{3} \cu'')}$ as the term of improvement of integrability, and treat all the other terms
		\begin{align*}
			\vert A^{-1}\nabla v \Ind{\a(\cdot) \neq 0}\vert +  \vert \cu''\vert  \vert A \w  \vert +\vert \g \vert,
		\end{align*}
		as the inhomogeneous terms, i.e.
		\begin{multline*}
			\norm{ \nabla v \Ind{\a(\cdot) \neq 0}}_{\aL^2(\cu'')} \leq K \norm{\nabla v \Ind{\a(\cdot) \neq 0}}_{\aL^{\frac{2+2_*}{2}}(\frac{4}{3} \cu'')} \\
			+ C\norm{\vert A^{-1}\nabla v \Ind{\a(\cdot) \neq 0}\vert + \vert \cu''\vert  \vert A\w \vert +\vert \g \vert}_{\aL^2(\frac{4}{3} \cu'')}.
		\end{multline*}
		Gehring’s lemma yields an inhomogeneous version (see for example the proof in \cite[Lemma~C.4]{AKMbook}, \cite[Proposition~6.1]{Gehring95}, and  \cite[Section~3.2]{navarro2023variants})
		\begin{multline*}
			\norm{\nabla v \Ind{\a(\cdot) \neq 0} }_{\underline{L}^{2+\eps}(\frac{4}{3} \cu )} 
			\leq C\norm{\nabla v\Ind{\a(\cdot) \neq 0} }_{\underline{L}^{2}(\rr \cu )} \\ + C\norm{\vert A^{-1}\nabla v \Ind{\a(\cdot) \neq 0}\vert + \vert \cu\vert \vert A\w \vert +\vert \g \vert}_{\underline{L}^{2+\eps}(\rr \cu )}.
		\end{multline*}
		This bound will be trivial when $A \equiv 1$,  but in practice we let $A^{-1}$ be very small to convey the regularity of $v$. This is the desired result as \eqref{eq.Meyers_New}. 
	\end{proof}

	\begin{proof}[Proof of Proposition~\ref{prop.weighted_L2_NEW}]
		Under the condition \eqref{eq.condition_hk_NEW}, we replace \eqref{eq.Meyers_Poisson} by the coarse-grained Meyers' estimate in Lemma~\ref{lem.Meyers_NEW}, and follow the proof of Proposition~\ref{prop.weighted_L2} until the step \eqref{eq.PoissonOneTerm2}: for one term $\cu \in \pPN_{h+1}(\cu_m)$, we have
		\begin{multline*}
			\sum_{\substack{\cu' \in \sSl_{\gGN_{h}}(\cu)}}|\cu'|^{\alpha}	 \sum_{e\in \Ed(\rr\cu' \cap \cu_m)} |\na v(e)|^2 \Ind{\a(e) \neq 0}\\
			\leq  C\vert \cu \vert \Ll(\sum_{e\in \Ed(\rr\cu \cap \cu_m)}\Ll(\vert A^{-1}\nabla v(e) \vert \Ind{\a(e) \neq 0}+ \vert \cu\vert  \vert A \w (e)\vert  +\vert \g (e)\vert \Rr)^{2}\Rr) \\
			+ C\Ll(\sum_{e\in \Ed(\rr\cu \cap \cu_m)}|\na v(e)|^2 \Ind{\a(e) \neq 0}\Rr).
		\end{multline*}
		We choose the weight $A(e) = \vert \cu \vert^{-1}$ in this equation, which yields
		\begin{multline*}
			\sum_{\substack{\cu' \in \sSl_{\gGN_{h}}(\cu)}}|\cu'|^{\alpha}	 \sum_{e\in \Ed(\rr\cu' \cap \cu_m)} |\na v(e)|^2 \Ind{\a(e) \neq 0} \\
			\leq C\Ll(\sum_{e\in \Ed(\rr\cu \cap \cu_m)}|\na v(e)|^2 \Ind{\a(e) \neq 0}\Rr)
			+ C\vert \cu \vert^3 \Ll(\sum_{e\in \Ed(\rr\cu \cap \cu_m)}\Ll(\vert\w (e)\vert  +\vert \g (e)\vert \Rr)^{2}\Rr).
		\end{multline*}
		We sum over $\cu \in \pPN_{h+1}(\cu_m)$ and apply counting argument \eqref{eq.counting} to $|\na v(e)|^2$, then we get
		\begin{multline}\label{prop.weighted_L2_NEW_middle}
			\sum_{\cu' \in \pPN_{h}}|\cu'|^{\alpha}	 \sum_{e\in \Ed(\rr\cu' \cap \cu_m)} |\na v(e)|^2 \Ind{\a(e) \neq 0} \\
			\leq C \norm{\na v \Ind{\a(\cdot) \neq 0} }^2_{L^2(\cu_m)}
			+ C \sum_{\cu \in \pPN_{h+1}} \vert \cu \vert^3 \Ll(\sum_{e\in \Ed(\rr\cu \cap \cu_m)}\Ll(\vert\w (e)\vert  +\vert \g (e)\vert \Rr)^{2}\Rr).
		\end{multline}
		It remains to absorb the term $\norm{\na v \Ind{\a(\cdot) \neq 0} }^2_{L^2(\cu_m)}$. Since $v \equiv 0$ on $\partial \cu_m$, we test \eqref{eq.PoissonCluster} with $v$ to get the basic Dirichlet energy estimate
		\begin{equation*}
			\bracket{\na v, \a \na v}_{\cu_m} = \bracket{\w, \nabla v}_{\cu_m} 
			\leq B\norm{\w}^2_{L^2(\cu_m)} + B^{-1}\norm{\nabla v}^2_{L^2(\cu_m)}
		\end{equation*}
		A coarse-graining argument in \eqref{eq.coarse_grain_control} is then applied to $\pPN_{h}(\cu_m)$, which implies 
		\begin{multline*}
			\norm{\na v \Ind{\a(\cdot) \neq 0} }^2_{L^2(\cu_m)} \leq B\norm{\w}^2_{L^2(\cu_m)} +  B^{-1}\norm{\g}^2_{L^2(\cu_m)} \\
			+ B^{-1}\sum_{\cu' \in \pPN_{h}(\cu_m)} \sum_{e \in \Ed(\cu')} \vert \rr \cu' \vert^{2}   \vert \nabla v(e)\vert^{2} \Ind{\a(e) \neq 0}.
		\end{multline*}
		We insert this estimate back to \eqref{prop.weighted_L2_NEW_middle}, and choose $B$ such that $	\rr^{2d}	B^{-1}  C \leq \frac{1}{2}$, then the term $C \norm{\na v \Ind{\a(\cdot) \neq 0} }^2_{L^2(\cu_m)}$ on its {\rhs} compensates part of its {\lhs}, if $\alpha \geq 2$ is assumed here. Then we conclude desired result \eqref{eq.weighted_L2_NEW} under the condition $\alpha \geq 2$.
		
		The $0 \leq \alpha < 2$ is dominated by the case $\alpha = 2$, then we complete the proof.
	\end{proof}
	
	In the remaining of the paper, our task is to find a reasonable coarse-controlled function $\g$ for the perturbed corrector equation \eqref{eq.recurrence_eq}.

	\section{Cluster-growth decomposition}\label{sec.HarmonicExtension}	
	
	We keep the convention of notation \eqref{eq.shorthand} in the following sections. As discussed in \eqref{eq.harmonicBC}, the function $v_{m}$ follows a constant extension rule in order to be $\a$-harmonic in the whole domain $\cu_m$. However, because $V_m(F,j)$ involves more functions like $v^\#_{m}$, its behavior of extension is more complicated. This section is devoted to the structure  hidden in $D_F v_m$.

	We propose at first a canonical grain extension, which is more robust when opening edges in percolation.
	Recall that $\OO^F(x)$ stands for the hole containing $x$ if $x$ is not on the infinite cluster of $\a^F$, and $\cCb^F(\cu) \equiv \cCb^{\a^F}(\cu)$ is the boundary-connecting cluster under $\a^F$.
	
	\begin{definition}[Canonical constant extension]\label{def.ConstantExtension_canonical}
		Given a realization $\a$ and ${F \subset \Ed(\cu)}$, we define \emph{the canonical grain}
		\begin{align}\label{eq.def_center}
			[x]^F:= \Ll\{
			\begin{array}{ll}
				x, & x \in \cCb^F(\cu),\\
				\arg \min_{y \in \cCb^\a(\cu)} \dist(y, \OO^F(x)), & x \notin \cCb^F(\cu) .
			\end{array}\Rr.
		\end{align} 
		Here we fix a lexicographic order for $\arg \min$. We also keep the convention $[x] \equiv [x]^{\emptyset}$.	
		
		A function $u : \{0,1\}^{\Ed(\cu)} \times \cu \to \R$ follows the canonical grain extension rule if 	
		\begin{align}\label{eq.def_extension_plan}
			u(\a^F, x) = u(\a^F, [x]^F).
		\end{align}
	\end{definition}
	We skip the dependence of domain in the notation $[\cdot]$, but it will be clear in the context. Roughly speaking,  a function $u$ satisfying the canonical grain extension can be constructed in two steps:
	\begin{enumerate}
		\item Define its value	at first on $\cCb^F(\cu)$.
		\item Choose some vertices on $\cCb^\a(\cu)$ to extend the value on holes.
	\end{enumerate}
	We highlight this  boundary-connecting cluster in step 2 is under $\a$ rather than $\a^F$, so it is ``canonical" in some sense. This rule increases the stability, and works perfectly for $v_{m}$ defined in \eqref{eq.harmonicBC}.

	This canonical grain extension is robust, because the operator $[\cdot]$ remains unchanged if the there do not exist two holes connecting by the edges in $F$. We prove at first this observation as a warm-up.
	
	\begin{lemma}\label{lem.decom_caonical_easy}
		Given a realization $\a$, if $F \subset \EE(\cCb^F(\cu))$, then for every function $u$ satisfying Definition~\ref{def.ConstantExtension_canonical}, we have
		\begin{align}\label{eq.decom_caonical_easy_1}
			\forall G \subset F, \forall x \in \cu \setminus  \cCb^F(\cu), \qquad u^G(x) = u^G([x]).
		\end{align}
		As a consequence, we have 
		\begin{align}\label{eq.decom_caonical_easy_2}
			\forall x \in \cu \setminus  \cCb^F(\cu), \qquad (D_F u)(x) = (D_F u)([x]).
		\end{align}
	\end{lemma} 
	\begin{proof}
		
		For $x \in \cu \setminus  \cCb^F(\cu)$, we look at its associated hole. Because ${G \subset F \subset \EE(\cCb^F(\cu))}$, no edges in $G$ are added to $\OO(x)$. Thus we obtain 
		\begin{align*}
			\forall G \subset F, \qquad \OO^G(x) = \OO(x).
		\end{align*}
		Since \eqref{eq.def_center} only relies on vertices of the boundary-connecting cluster under $\a$,  we obtain
		\begin{align}\label{eq.dropG}
			[x]^G = [x].
		\end{align}
		We can  justify \eqref{eq.decom_caonical_easy_1} as 
		\begin{align*}
			u^G(x) \stackrel{\eqref{eq.f_convention}}{=} u(\a^G, x)  \stackrel{\eqref{eq.def_extension_plan}}{=}  u(\a^G, [x]^G) \stackrel{\eqref{eq.dropG}}{=} u(\a^G, [x])  \stackrel{\eqref{eq.f_convention}}{=} u^G([x]).
		\end{align*}
		Then, for every $x \in \cu \setminus  \cCb^F(\cu)$
		\begin{align*}
			(D_F u)(x) &= \sum_{G \subset F} (-1)^{\vert F \setminus G \vert} u^G([x]^G) \\
			&= \sum_{G \subset F} (-1)^{\vert F \setminus G \vert} u^G([x]) \\
			&=  (D_F u)([x]).
		\end{align*}
	\end{proof}

	To generalize Lemma~\ref{lem.decom_caonical_easy}, we need to understand which bonds contribute to the growth of the boundary-connecting cluster in the passage from $\a$ to $\a^F$. This relies on a key observation called \emph{cluster-growth decomposition}:
	\begin{equation}\label{eq.def_WIP_new}
		\begin{split}
			F & = F_* \sqcup F_\circ, \\
			F_* &:= F \cap \EE(\cCb^F(\cu)), \\
			F_\circ &:= F \setminus F_*.
		\end{split}
	\end{equation}
	We will see that, the two sets $F_*$ and $F_\circ$ are more or less independent:
	\begin{itemize}
		\item $F_*$ enlarges the boundary-connecting cluster;
		\item $F_\circ$ builds ``bridges" and shortcuts between the holes disconnecting to the boundary under $\a^{F_*}$. Here ``bridge" indicates the bond connecting two disjoint holes.
	\end{itemize}
	
	\begin{figure}[b]
		\centering
		\begin{subfigure}{0.4\textwidth}
			\includegraphics[width=\textwidth]{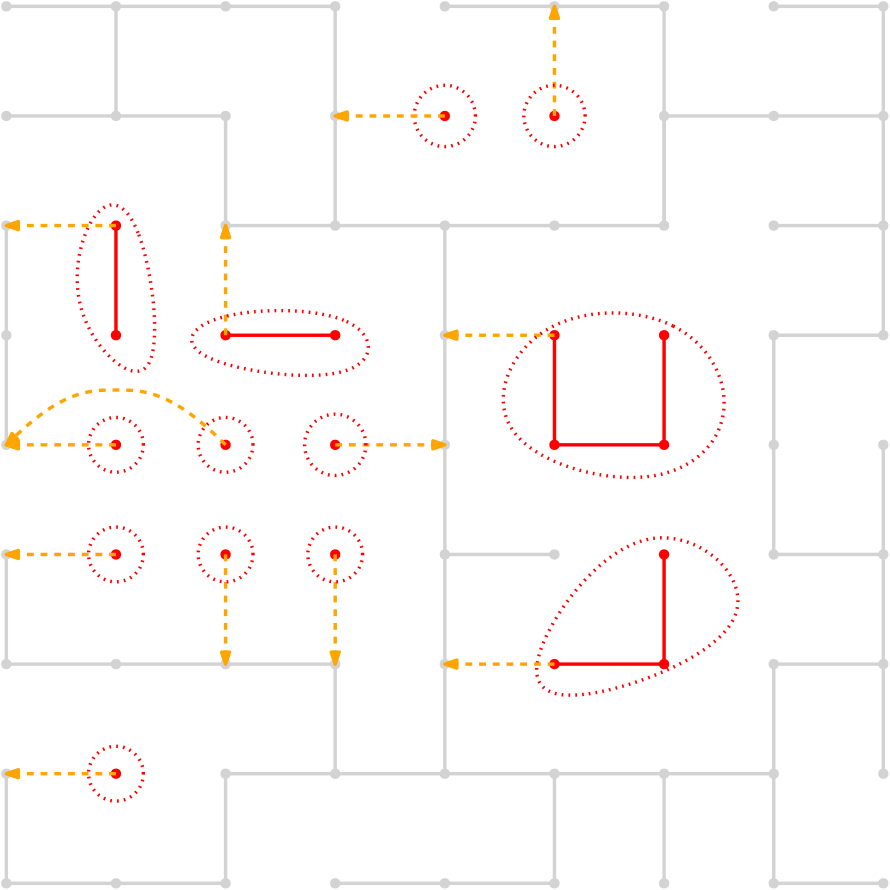}
			\caption{Percolation of $\a$}
			\label{fig.cluster_growth_0}
		\end{subfigure}
		\hfill
		\begin{subfigure}{0.4\textwidth}
			\includegraphics[width=\textwidth]{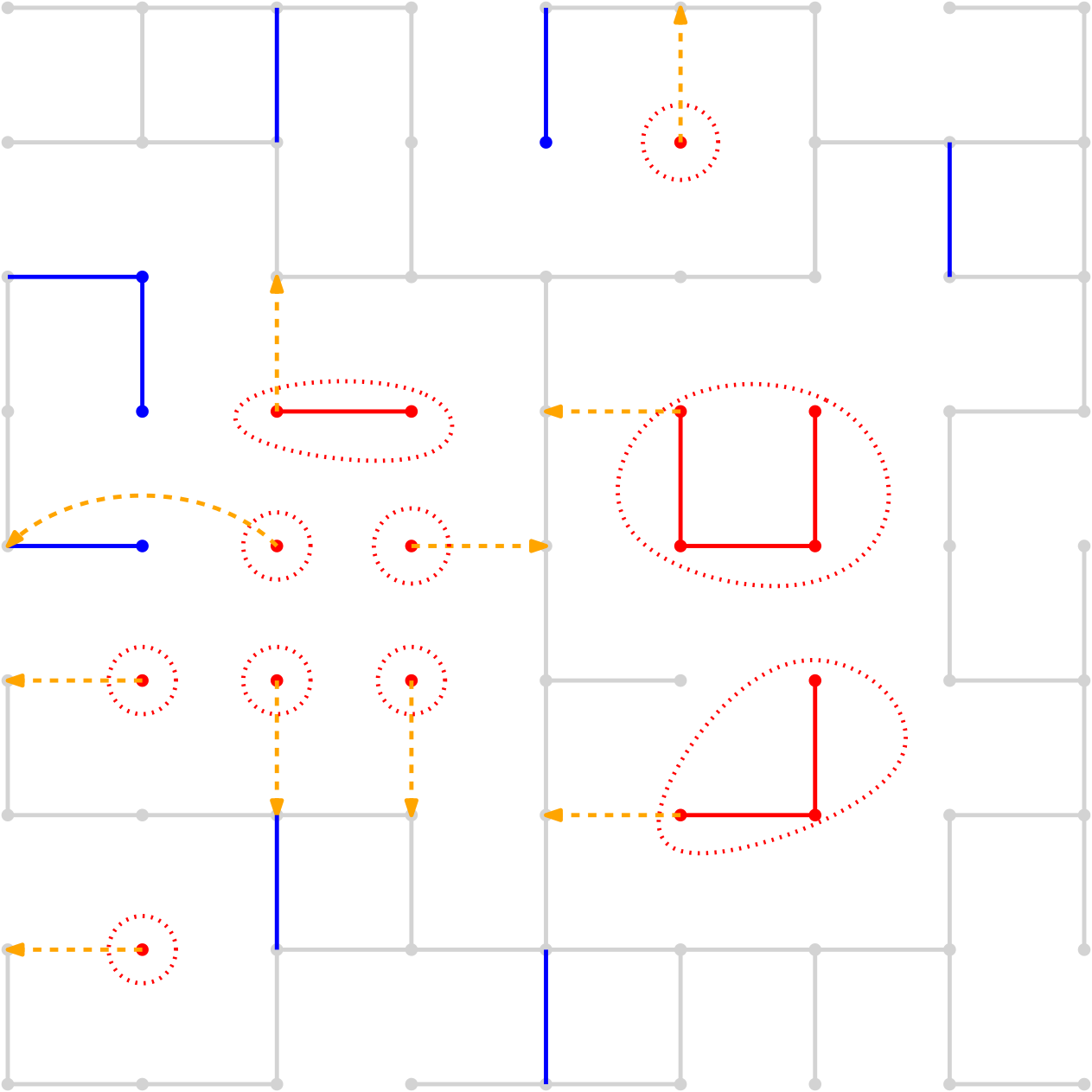}
			\caption{Percolation of $\a^{F_*}$}
			\label{fig.cluster_growth_1}
		\end{subfigure}
		\hfill
		\begin{subfigure}{0.4\textwidth}
			\includegraphics[width=\textwidth]{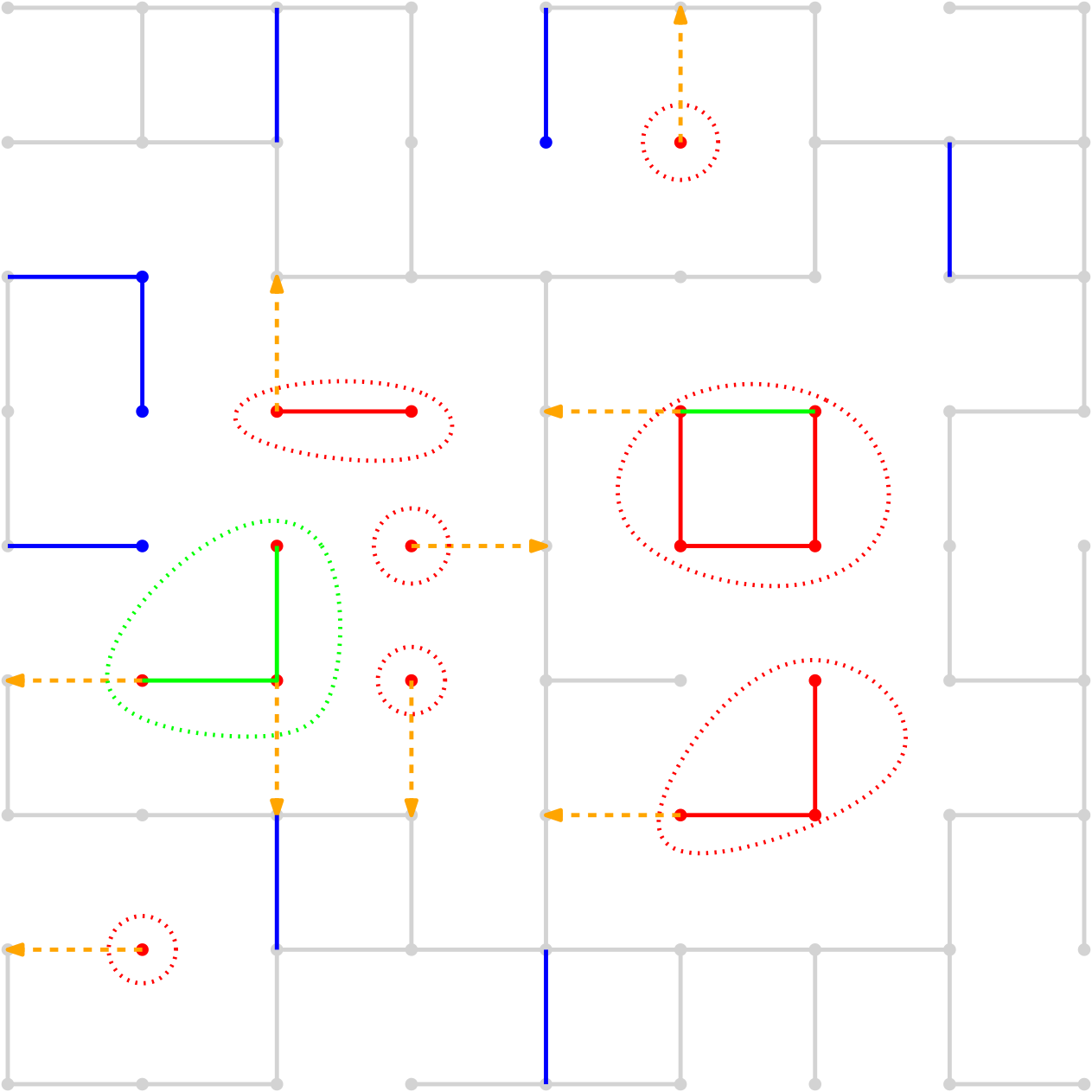}
			\caption{Percolation of $\a^{F_* \cup F_\circ}$}
			\label{fig.cluster_growth_2}
		\end{subfigure}
		
		\caption{The clusters in gray are boundary-connecting, while the others in red are holes. The edges in blue are in $F_*$, and the ones in green are of $F_\circ$. The arrows in orange mark the vertices for the canonical grain extension: among the nearest vertices on $\a$, they choose the one following the left-top order.}
		\label{fig.cluster_growth}
	\end{figure}

	See Figure~\ref{fig.cluster_growth} for an illustration. We also define another higher-order spatial differentiation operator
	\begin{align}\label{eq.def_diff_shift}
		(\nabla_{F} f)(x) := \sum_{G \subset F} (-1)^{\vert F \setminus G \vert} f([x]^G). 
	\end{align}
	
	The notations \eqref{eq.def_WIP_new} and \eqref{eq.def_diff_shift} then give a nice decomposition of $D_F v_m$, which can be summarized as: 
	\begin{center}
		\textit{``$F_*$ changes the harmonic function value, while $F_\circ$ changes the grain"}.
	\end{center}
	Afterwards, the following lemma establishes the link between the Glauber derivative and the spatial derivative.

	\begin{lemma}\label{lem.decom_DF}
		Let $v_{m}$ defined in \eqref{eq.harmonicBC} follow the canonical grain extension \eqref{eq.def_extension_plan} on $\cu_m$, then for every $F \subset \Ed(\cu_m)$ we have
		\begin{equation}\label{eq.decom_DF}
			\forall x \in \cu_m, \qquad (D_F v_m)(x) = (\nabla_{F_\circ} D_{F_*} v_m)(x).
		\end{equation}
		Here $F_*, F_\circ$ follows the decomposition \eqref{eq.def_WIP_new}.  
	\end{lemma}
	\begin{proof}
		
		The proof is divided into two steps.
		
		\textit{Step~1: case $x \in \cu_m  \setminus \cCb^F(\cu_m)$.} We make the calculation using $D_F u = D_{F_\circ} D_{F_*} u$ 
		\begin{equation}\label{eq.double_difference}
			\begin{split}
				(D_F v_m)(x) &= \sum_{G_2 \subset F_\circ} (-1)^{\vert F_\circ \setminus G_2 \vert} \sum_{G_1 \subset F_*} (-1)^{\vert F_* \setminus G_1 \vert} v_m^{G_1 \cup G_2}([x]^{G_1 \cup G_2}) \\
				&= \sum_{G_2 \subset F_\circ} (-1)^{\vert F_\circ \setminus G_2 \vert} \sum_{G_1 \subset F_*} (-1)^{\vert F_* \setminus G_1 \vert} v_m^{G_1}([x]^{G_1 \cup G_2})\\
				&= \sum_{G_2 \subset F_\circ} (-1)^{\vert F_\circ \setminus G_2 \vert} \sum_{G_1 \subset F_*} (-1)^{\vert F_* \setminus G_1 \vert} v_m^{G_1}([x]^{G_2})\\
				&= (\nabla_{F_\circ} D_{F_*} v_m)(x).
			\end{split}
		\end{equation}
		Let us explain why each equality holds: 
		\begin{itemize}[label=---]
			\item  The first line is just the definition \eqref{eq.def_extension_plan} and \eqref{eq.Difference}.
			\item  In the second line, because $F_\circ$ only creates bridges in $\cu_m \setminus  \EE(\cCb^F(\cu_m))$, opening bonds in $G_2$ will not change the geometry of the boundary-connecting cluster, i.e.
			\begin{align*}
				\Ll(\cCb^{G_1 \cup G_2}(\cu_m), \EE(\cCb^{G_1 \cup G_2}(\cu_m))\Rr) = \Ll(\cCb^{G_1}(\cu_m), \EE(\cCb^{G_1}(\cu_m))\Rr).
			\end{align*}  
			Therefore, by definition of $v_m$ in \eqref{eq.harmonicBC}, we have 
			\begin{align*}
				\forall y \in \cCb^{G_1}(\cu_m), \qquad v_m^{G_1 \cup G_2}(y) =  v_m^{G_1}(y).
			\end{align*}
			Because $[x]^{G_1 \cup G_2} \in \cCb(\cu_m) \subset \cCb^{G_1}(\cu_m)$, we can then replace $v_m^{G_1 \cup G_2}([x]^{G_1 \cup G_2})$ by  $ v_m^{G_1}([x]^{G_1 \cup G_2})$.
			\item  In the third line, because $G_1 \subset \EE(\cCb^F(\cu_m))$, it does not provide any bond connecting holes in $\cu_m \setminus  \EE(\cCb^F(\cu_m))$. Therefore, we have
			\begin{align*}
				\forall x \in \cu_m \setminus \cCb^{F}(\cu_m), \qquad	\OO^{G_2}(x) = \OO^{G_1 \cup G_2}(x).
			\end{align*}
			Using the definition of the canonical grain \eqref{eq.def_center}
			\begin{align*}
				\forall x \in \cu_m \setminus \cCb^{F}(\cu_m), \qquad   [x]^{G_2} = [x]^{G_1 \cup G_2},
			\end{align*}
			and this yields the equality in the third line.
			\item The forth line is just the expansion from \eqref{eq.def_diff_shift}.	
		\end{itemize}
		
		\medskip
		
		\textit{Step~2: case $x \in \cCb^F(\cu_m)$.} For this case, if $F_\circ = \emptyset$, then \eqref{eq.decom_DF} is obvious. Otherwise, we claim
		\begin{align}\label{eq.Null_Maximal_Cluster_pre}
			F_\circ \neq \emptyset \Longrightarrow \forall x \in \cCb^F(\cu_m), \quad (D_F v_m)(x) = (\nabla_{F_\circ} D_{F_*} v_m)(x) = 0.
		\end{align}
		Let $e \in F_\circ$, then we have 
		\begin{align*}
			D_F v_m &= \sum_{G \subset F \setminus \{e\}} (-1)^{\vert F \setminus (G \cup \{e\})\vert}v_m^{G \cup \{e\}} + (-1)^{\vert F \setminus G \vert}v_m^{G} \\
			&= \sum_{G \subset F \setminus \{e\}}  (-1)^{\vert F \setminus G \vert}\Ll(v_m^{G} - v_m^{G \cup \{e\}}\Rr).
		\end{align*}
		We then compare $v_m^{G}$ and $v_m^{G \cup \{e\}}$ for $x \in \cCb^F(\cu_m)$. Using the canonical grain extension \eqref{eq.def_extension_plan}, we have
		\begin{align}\label{eq.DF_cancel_0}
			\Ll(v_m^{G} - v_m^{G \cup \{e\}}\Rr)(x) = v_m^{G}([x]^{G}) - v_m^{G \cup \{e\}}([x]^{G \cup \{e\}}).
		\end{align}
		Because  $e \notin \EE(\cCb^F(\cu_m))$, its perturbation does not change the geometry of boundary-connecting cluster under $\a^G$, i.e.
		\begin{align}\label{eq.DF_cancel_observation}
			\Ll(\cCb^{G}(\cu_m), \EE(\cCb^{G}(\cu_m))\Rr) = \Ll(\cCb^{G \cup \{e\}}(\cu_m), \EE(\cCb^{G \cup \{e\}}(\cu_m))\Rr),
		\end{align}  
		which yields (by \eqref{eq.harmonicBC})
		\begin{align}\label{eq.DF_cancel_1}
			v_m^{G} \equiv v_m^{G \cup \{e\}} \text{ on } \cCb^{G}(\cu_m).
		\end{align}
		Moreover, we claim
		\begin{align}\label{eq.DF_cancel_2}
			[x]^{G} = [x]^{G \cup \{e\}}.
		\end{align}
		This can be justified in two cases:
		\begin{itemize}[label=---]
			\item If $x \in \cCb^{G}(\cu_m)$, then $[x]^{G} = [x]^{G \cup \{e\}} = x$ by \eqref{eq.def_center}.
			\item If $x \notin \cCb^{G}(\cu_m)$, then $e$ cannot be connected to the hole $\OO^G(x)$, otherwise it will be merged to $\cCb^{F}(\cu_m)$ when opening all bonds in $F$. Therefore, opening $e$ will not enlarge $\OO^G(x)$, i.e.
			\begin{align*}
				\OO^G(x) = \OO^{G \cup \{e\}}(x),
			\end{align*}
			and we have $[x]^{G} = [x]^{G \cup \{e\}}$ by \eqref{eq.def_center}.
		\end{itemize}
		Then we put \eqref{eq.DF_cancel_1} and \eqref{eq.DF_cancel_2} back to \eqref{eq.DF_cancel_0}, which concludes $(D_F v_m)(x) = 0$ for $x \in \cCb^F(\cu_m)$. The part $(\nabla_{F_\circ} D_{F_*} v_m)(x) = 0$ is similar and we skip the proof.
	\end{proof}

	Built on these notations, actually there are a lot of cancellation $D_F v_m$. They are quite natural as the perturbation on disjoint components will annihilate the function.
	\begin{corollary}\label{cor.DF_cancel}
		Under the setting of Lemma~\ref{lem.decom_DF}, the following observations hold.
		\begin{enumerate}
			\item If $F \nsubseteq \EE(\cCb^F(\cu_m))$, i.e. $F_\circ \neq \emptyset$, then 
			\begin{align}\label{eq.Null_Maximal_Cluster}
				\forall x \in \cCb^F(\cu_m), \qquad (D_F v_m)(x) = 0.
			\end{align}
			\item For all $x \in  \cu_m \setminus  \cCb^F(\cu_m)$, if $F_\circ \nsubseteq \EE(\OO^{F_\circ}(x))$, then $(D_F v_m)(x) = 0$.
		\end{enumerate}
	\end{corollary}
	\begin{proof}
		(1) is the result of \eqref{eq.Null_Maximal_Cluster_pre}.
		
		\smallskip
		
		For (2), then there exists an edge $e \in F_\circ$ but $e \notin \EE(\OO^{F_\circ}(x))$. We apply the formula in \eqref{eq.decom_DF} and the definition \eqref{eq.def_diff_shift}
		\begin{equation}\label{eq.DF_cancel_3}
			\begin{split}
				&(D_F v_m)(x) \\
				&= (\nabla_{F_\circ} D_{F_*} v_m)(x) \\
				&= \sum_{G \subset F_\circ \setminus \{e\}} (-1)^{\vert F_\circ \setminus (G \cup \{e\})\vert}  (D_{F_*} v_m)([x]^{G \cup \{e\}}) + (-1)^{\vert F_\circ \setminus G \vert} (D_{F_*} v_m)([x]^{G}) \\
				&= \sum_{G \subset F_\circ \setminus \{e\}}  (-1)^{\vert F_\circ \setminus G \vert}\Ll((D_{F_*} v_m)([x]^{G})- (D_{F_*} v_m)([x]^{G \cup \{e\}})\Rr).
			\end{split}
		\end{equation}
		The edge $e$ should not be connected to $\OO^G(x)$, otherwise it is in $\EE(\OO^{F_\circ}(x))$. This implies $\OO^{G \cup \{e\}}(x) = \OO^{G}(x)$, and we have $[x]^{G} = [x]^{G \cup \{e\}}$ by the definition \eqref{eq.def_center}. We put this identity back to \eqref{eq.DF_cancel_3} and finish the proof.
		
	\end{proof}

	\section{Weighted $\ell^1$-$L^2$ estimate $I_m(i,0)$}\label{sec.KeyEstimate1}
	This section and next section are the main part to prove a uniform bound \eqref{e:main_terms_to_bound} with respect to $m \in \N_+$. The key result is the following quenched weighted $\ell^1$-$L^2$ estimate, which is deterministic once we know the cube is good enough with the desired pyramid partition $\pPN_{h}$. We keep the conventions \eqref{eq.shorthand} for $V_{m,\xi}(F, j)$. The parameter $\rr$ in the statement roots from Definition~\ref{def.pyramid}, whose range is $(\frac{4}{3}, 2)$ by default. We also define the following parameter $h_*(\alpha,i,j)$ for the convenience, which is the minimal scale in the pyramid partition
	\begin{align}\label{eq.condition_h_star}
		h_*(\alpha,i,j) := \max \Ll\{\tilde{h}_{Me} ,\frac{(3+\alpha+12(i+j+2)^2)(2+\tilde{\eps}_{Me})}{\tilde{\eps}_{Me}}, \frac{(\alpha+5)(2+\eps_{Me})}{\eps_{Me}}\Rr\}.
	\end{align}
	Here the constants $\tilde{\eps}_{Me}, \tilde{h}_{Me}$ come from Lemma~\ref{lem.Meyers_NEW}, and $\eps_{Me}$ comes from Definition~\ref{def.Meyers}.
	
	\begin{prop}[Quenched improved weighted $\ell^1$-$L^2$ estimate]
		\label{pr:key_estimate}
		Given $i,j\in \N$ and $\alpha \geq 0$, then for every $\rr \in (\frac{50}{27}, 2)$, and every $N,k \in \N$ satisfying
		\begin{align}\label{eq.condition_Nk_main}
			N \geq i+j, \qquad k \geq \max\{h, h_*(\alpha,i,j)\} + 12(i+j) + 2,
		\end{align}
		there exists a constant $C(i,j,\alpha,h,N,k,d,\rr)$ independent of $m \in \N_+$ such that the following statement holds: for every $\cu_m \in \gGN_{k}$, we have the estimate when the canonical grain \eqref{def.ConstantExtension_canonical} is applied to $v_m$ in $\cu_m$ 
		\begin{multline}
			\label{e:key_estimate}
			\frac{1}{|\cu_m|}\sum_{\substack{|F|=i\\F\subset \Ed(\cu_m)}}\Ll(\sum_{\substack{\cu \in \pPN_{h}(\cu_m)}} |\cu|^{\alpha}\sum_{e\in \Ed(\cu)} |\na V_{m}(F, j)(e)|^2 \Rr) \\ \leq C(i,j,\alpha,h,N,k,d,\rr) .
		\end{multline}           
	\end{prop}

	The proof is divided into several subsections, so we can attack the technical difficulties separately. In the proof, we will make use of the nice properties of the perturbed corrector equations and the pyramid partitions respectively discussed in Lemma~\ref{lem.recurrence_eq} and Definition~\ref{def.pyramid}.
	


	\subsection{Coarse-graining argument for $V_m(F,0)$}

	We develop the coarse-graining argument for $V_m(F,0)$ using the structure discussed in Section~\ref{sec.HarmonicExtension}.

	The following lemma implies that partition cube is large enough to capture the information of the canonical grain introduced in \eqref{eq.def_center}.
	\begin{lemma}\label{lem.center_cube}
		Given $\rr \in (\frac{50}{27}, 2)$, then every $\cu_m \in \gGN_k$ is $N$-$\stable$ for the following property: 
		for every $\cu' \in \pPN_{h}(\cu_m)$ with $0 \leq h \leq k-1$, we have
		\begin{align}\label{eq.center_cube}
			\forall x \in \frac{16}{15} \cu' \cap \Ll(\cu_m \setminus \cCb^\#(\cu_m)\Rr),  \qquad  [x]^\# \in \cCp\Ll(\rr\cu' \cap \cu_m\Rr).
		\end{align}
	\end{lemma}
	\begin{proof}
		
		Thanks to the very-well connectivity in Definition~\ref{def:well_connected}, one of the $27^d$ third order successors of $\frac{4}{3}\cu'$ already provides a candidate of vertex on $\a$ with distance 
		\begin{align*}
			\frac{1}{27}  \cdot \frac{4}{3} \size(\cu') = \frac{4}{81} \size(\cu').
		\end{align*}
		We claim that the vertex outside $\frac{40}{27}\cu'$ cannot do better. Lemma~\ref{lem.detouring_path} implies 
		\begin{align*}
			\OO^\#(x) \subset \frac{4}{3}\cu'.
		\end{align*}
		Therefore, we have the estimate 
		\begin{multline*}
			\min_{y \notin \frac{40}{27} \cu'} \dist(y, \OO^\#(x)) \geq \min_{y \notin \frac{40}{27} \cu'} \dist\Ll(y, \frac{4}{3}\cu'\Rr) \geq \frac{\frac{40}{27} - \frac{4}{3}}{2} \size(\cu')\\
			= \frac{4}{54} \size(\cu') > \frac{4}{81} \size(\cu'),
		\end{multline*}
		and this implies that $[x]^\# \in \frac{40}{27}\cu'$. 
		
		Finally, a similar argument as \eqref{eq.large_open_path} justifies that $[x]^\# \in \cCp\Ll(\rr\cu' \cap \cu_m\Rr)$: if it is not the case, then there exists an open path $\gamma$ connecting $\partial \rr \cu'$ without touching $\cCm(\rr \cu')$. Then we calculate its proportion,
		\begin{align}\label{eq.large_open_path_2}
			\frac{\vert \gamma\vert}{\size(\rr \cu')} \geq \frac{\frac{\rr - \frac{40}{27}}{2}\size( \cu')}{\rr\size( \cu')}  \geq \frac{1}{2}\Ll(1 - \frac{\frac{40}{27}}{\rr}\Rr) \geq \frac{1}{2}\Ll(1 - \frac{40/27}{50/27}\Rr) \geq \frac{1}{10}.
		\end{align}
		In the last inequality, we use the condition $\rr \in (\frac{50}{27}, 2)$. Because $\rr \cu'$ is well-connected, the proportion above contradicts Definition~\ref{def:well_connected}. We thus finish the proof of \eqref{eq.center_cube}.
	\end{proof}

	We prove the coarse-graining argument for $V_m(F,0)$. Here the notation $F_*, F_\circ$ are defined as \eqref{eq.def_WIP_new}.
	\begin{lemma}[Coarse-graining argument for $V_m(F,0)$]\label{lem.coarse_grained_VF0}
		Given $\cu_m \in \gGN_k$, then the following estimates hold for every integer $0 \leq h \leq k-2$, and $1 \leq i = \vert F\vert \leq N$ with $F \subset \Ed(\cu_m)$, and every $\rr \in (\frac{50}{27}, 2)$: 
		\begin{itemize}
			\item For every $\cu \in \pPN_{h}(\cu_m)$ and edge $e \cap \cu \neq \emptyset$, we have
			\begin{align}\label{eq.coarse_grained_VF0_point}
				|\na V_m(F,0)(e)| \leq 2^{\vert F_\circ \vert} \sum_{e' \in \EE(\cCp^F(\rr \cu \cap \cu_m))} |\na V_m(F_*,0)(e')|\Ind{F_\circ \subset \Ed(\rr \cu)}.
			\end{align}
			\item There exists a finite positive constant $C(\rr,N,d,h,k)$, such that for every real number $\alpha > 0$, we have
			\begin{multline}\label{eq.coarse_grained_VF0}
				\sum_{\substack{|F|=i\\F\subset \Ed(\cu_m)}}\sum_{\substack{\cu \in \pPN_{h}(\cu_m)} }|\cu|^{\alpha} \sum_{e\in \Ed\Ll(\rr\cu \cap \cu_m \Rr)} |\na V_m(F,0)(e)|^2 \\
				\leq C \sum_{i_1 = 0}^i \sum_{\substack{|F|=i_1\\F\subset \Ed(\cu_m)}}\sum_{\substack{\cu \in \pPN_{h+1}(\cu_m)} }|\cu|^{\alpha+4+i-i_1} \sum_{e\in \EE\Ll(\cCp^{F}(\rr\cu \cap \cu_m)\Rr)} |\na V_m(F,0)(e)|^2. 
			\end{multline}
		\end{itemize}
	\end{lemma}
	\begin{proof}
		Concerning the {\lhs} of \eqref{eq.coarse_grained_VF0}, we use the decomposition \eqref{eq.decom_DF}, and write the sum as 
		\begin{align}\label{eq.sum_F1_F2}
			\sum_{\substack{|F|=i\\F\subset \Ed(\cu_m)}} \Big(\cdots\Big) =  \sum_{i_1 = 0}^i \sum_{\substack{|F_*|=i_1\\F_* \subset \EE\Ll(\cCb^{F_*}(\cu_m)\Rr)}} \sum_{\substack{|F_\circ|=i-i_1\\F_\circ \subset \Ed\Ll(\cu_m \setminus \cCb^{F_*}(\cu_m)\Rr)}} \Big(\cdots\Big).
		\end{align}
		Besides, we always have the following calculation for a general function $v$
		\begin{multline}\label{eq.decom_classic}
			\sum_{\substack{\cu \in \pPN_{h}(\cu_m)} }|\cu|^{\alpha}\sum_{e\in \Ed(\rr\cu)} |\na v(e)|^2 \\
			\leq \sum_{\substack{\cu' \in \pPN_{h}(\cu_m)} } \sum_{\substack{e \cap \cu' \neq \emptyset}} |\na v(e)|^2 \Ll( \sum_{\substack{\cu \in \pPN_{h}(\cu_m)} } |\cu|^{\alpha} \Ind{\rr \cu \supset e}\Rr).
		\end{multline}
		We combine \eqref{eq.sum_F1_F2} and \eqref{eq.decom_classic}, and  study 
		\begin{align*}
			\nabla v(e) = \nabla V_m(F_* \cup F_\circ,0)(e), \qquad e \cap \cu' \neq \emptyset.
		\end{align*}

		\textit{Step~1: point-wise estimate of gradient.} 	
		We denote by $e=\{x, y\}$ in the discussion, and recall that $\vert F_* \vert = i_1$, $\cu' \in \pPN_{h}(\cu_m)$.
		
		\textit{Case~1: $i_1 = i$, i.e. $F_\circ = \emptyset$.} For this case, $F=F_*$. We further discuss the status of $x,y$ as the two endpoints of $e$:
		\begin{itemize}[label=---]
			\item  If both $x$ and $y$ are in $\cCb^{F}(\cu_m)$, then they are also in $\cCp^F(\rr \cu' \cap \cu_m)$  by the observation in Lemma~\ref{lem.detouring_path}. 
			\item  If some endpoint in $e$, saying $x$, is not in $\cCb^{F}(\cu_m)$, then observation \eqref{eq.decom_caonical_easy_2} applies that 
			\begin{align*}
				V_m(F,0)(x) =  V_m(F,0)([x]).
			\end{align*}
			Moreover, \eqref{eq.center_cube} suggests that the grain $[x] \in \cCp(\rr \cu' \cap \cu_m)$.
		\end{itemize}
		In conclusion, we can carry the estimate of $\nabla V_m(F,0)(e)$ on $\cCp^F(\rr \cu' \cap \cu_m)$. As $\cCp^F(\rr \cu' \cap \cu_m)$ may contain more than one cluster, there are still two possibilities:
		\begin{itemize}[label=---]
			\item  If $x$ and $y$ or their associated grains live on the same cluster, then there must be an open path $\gamma_e \subset \EE(\cCp^F(\rr \cu' \cap \cu_m))$ connecting them, and we have
			\begin{align}\label{eq.shortcutPath}
				|\na V_m(F,0)(e)|
				&\leq \sum_{e'\in \gamma_e}|\na V_m(F,0)(e')|\\
				\nonumber &\leq \sum_{e'\in \EE \Ll(\cCp^F(\rr\cu' \cap \cu_m)\Rr)} |\na V_m(F,0)(e')|.
			\end{align}
			Here in the second line we relax the sum to $\EE \Ll(\cCp^F(\rr\cu' \cap \cu_m)\Rr)$.   
			\item Otherwise, $x$ and $y$ or their associated grains live on the different clusters. This case only happens near the boundary of $\partial \cu_m$ according to the definition of \eqref{eq.def_clt_boundary}, and then there are two disjoint open paths $\gamma_x, \gamma_y$ connecting $x, y$ to $\partial \cu_m$ within $\EE \Ll(\cCp^F(\rr\cu' \cap \cu_m)\Rr)$. We then find a third path $\gamma'$ on $\partial \cu_m$ to connect $\gamma_x$ and  $\gamma_y$ with the following estimate
			\begin{equation}\label{eq.shortcutPath2}
				\begin{split}
					& |\na V_m(F,0)(e)|\\
					&= \Ll\vert \sum_{e'\in \gamma_x \cup \gamma' \cup \gamma_y} \na V_m(F,0)(e') \Rr\vert \\
					&\leq \Ll\vert \sum_{e'\in \gamma_x} \na V_m(F,0)(e') \Rr\vert + \underbrace{\Ll\vert \sum_{e'\in  \gamma' } \na V_m(F,0)(e') \Rr\vert}_{=0} + \Ll\vert \sum_{e'\in \gamma_y} \na V_m(F,0)(e') \Rr\vert\\
					&\leq \sum_{e'\in \EE \Ll(\cCp^F(\rr\cu' \cap \cu_m)\Rr)} |\na V_m(F,0)(e')|.
				\end{split}
			\end{equation}
			Here because we assume $\vert F \vert = i \geq 1$, then $V_m(F,0) \equiv 0$ on $\partial \cu_m$, and the contribution along $\gamma'$ is null. 
		\end{itemize}
		In conclusion, for the case $F_\circ = \emptyset$, we obtain
		\begin{align}\label{eq.shortcutPath_Fo_empty}
			|\na V_m(F,0)(e)| \leq \sum_{e'\in \EE \Ll(\cCp^F(\rr\cu' \cap \cu_m)\Rr)} |\na V_m(F,0)(e')|.
		\end{align}

		\textit{Case~2: $0 \leq i_1 \leq i-1$, i.e. $F_\circ \neq \emptyset$.} We use the trivial bound 
		\begin{align}\label{eq.shortcutPath_Fo_trivial}
			|\na V_m(F,0)(\{x,y\})| \leq  |V_m(F,0)(x)| + |V_m(F,0)(y)|.
		\end{align}
		Viewing the cancellation in \eqref{eq.Null_Maximal_Cluster}, the only term which contributes is when ${x \in \cu_m \setminus \cCb^{F}(\cu_m)}$. Then we can apply the structure \eqref{eq.decom_DF}, the definition \eqref{eq.def_diff_shift}
		\begin{align*}
			|V_m(F,0)(x)| = |\nabla_{F_\circ} V_m(F_*,0)(x)| =  \Ll|\sum_{G \subset F_\circ} (-1)^{\vert F_\circ \setminus G\vert} V_m(F_*,0)([x]^G)\Rr|.
		\end{align*}
		The observation \eqref{eq.center_cube} suggests that $[x]^G \in \cCp(\rr \cu' \cap \cu_m)$ for all $G$ thanks to the $N$-$\stable$ property. We can treat $\nabla_{F_\circ}$ just as $2^{i-i_1-1}$ pairs of difference on $\cCp(\rr\cu' \cap \cu_m)$ and apply the similar idea as \eqref{eq.shortcutPath} and \eqref{eq.shortcutPath2}
		\begin{align*}
			|V_m(F,0)(x)| &\leq 2^{i-i_1-1}  \sum_{e'\in \EE \Ll(\cCp(\rr\cu' \cap \cu_m) \Rr)}|\na V_m(F_*,0)(e')|.
		\end{align*}
		We also recall the cancellation in (2) of Corollary~\ref{cor.DF_cancel}, so this only contributes when 
		\begin{align*}
			F_\circ \subset  \EE(\OO^{F_\circ}(x)) \subset \Ed(\rr \cu').
		\end{align*}
		Here the second inclusion comes from Lemma~\ref{lem.detouring_path}. Therefore, combining \eqref{eq.shortcutPath_Fo_trivial}, we strengthen the result as
		\begin{align}\label{eq.coarse_grained_VF0_Case2}
			|\nabla V_m(F,0)(e)| &\leq 2^{i-i_1}  \sum_{e'\in \EE \Ll(\cCp(\rr\cu' \cap \cu_m) \Rr)}|\na V_m(F_*,0)(e')|\Ind{F_\circ \subset \Ed(\rr \cu')}.
		\end{align}
		The result \eqref{eq.shortcutPath_Fo_empty} can be unified with \eqref{eq.coarse_grained_VF0_Case2}, which gives the point-wise estimate \eqref{eq.coarse_grained_VF0_point}.
		\smallskip
		
		\textit{Step~2: weighted estimate.} We apply Cauchy--Schwarz inequality to \eqref{eq.coarse_grained_VF0_Case2} and obtain
		\begin{align*}
			|\nabla V_m(F,0)(e)|^2 &\leq 2^{2(i-i_1)} \vert \rr\cu' \vert  \sum_{e'\in \EE \Ll(\cCp(\rr\cu' \cap \cu_m) \Rr)}|\na V_m(F,0)(e)|^2 \Ind{F_\circ \subset \Ed(\rr \cu')}.
		\end{align*}
		Combining this with  \eqref{eq.decom_classic}, the {\rhs} of \eqref{eq.sum_F1_F2} is bounded by
		\begin{align*}
			&\sum_{i_1 = 0}^{i} \sum_{\substack{|F_*|=i_1\\F_* \subset \EE\Ll(\cCb^{F_*}(\cu_m)\Rr)}} \sum_{\substack{|F_\circ|=i-i_1\\F_\circ \subset \Ed\Ll(\cu_m \setminus \cCb^{F_*}(\cu_m)\Rr)}}\sum_{\substack{\cu \in \pPN_{h}(\cu_m)} }|\cu|^{\alpha} \sum_{e\in \Ed\Ll(\rr\cu\Rr)} |\na V_m(F,0)(e)|^2 \\
			&\leq \sum_{i_1 = 0}^{i}2^{2(i-i_1)}   \sum_{\substack{|F_*|=i_1\\F_* \subset  \Ed(\cu_m)}} \sum_{\substack{|F_\circ|=i-i_1\\F_\circ \subset  \Ed(\cu_m)}} \sum_{\substack{\cu' \in \pPN_{h}(\cu_m)} } \sum_{\substack{e \in \Ed(\cu_m) \\ e \cap \cu' \neq \emptyset}}  \sum_{e'\in \EE \Ll(\cC_*(\rr\cu')\Rr)}\\
			&\qquad  |\cu'|  |\na V_m(F_*,0)(e')|^2 \Ind{F_\circ \subset \Ed(\rr \cu')} \times \Ll( \sum_{\substack{\cu \in \pPN_{h}(\cu_m)} } |\cu|^{\alpha} \Ind{\rr \cu \supset e}\Rr)\\
			&\leq C \sum_{i_1 = 0}^{i} \sum_{\substack{|F_*|=i_1\\F_* \subset  \Ed(\cu_m)}}\sum_{e'\in \EE \Ll(\cCb^{F_*}(\cu_m)\Rr)} A_{\alpha, i-i_1+1}(e') |\na V_m(F_*,0)(e')|^2.
		\end{align*}
		with the weight function $A_{\alpha, \beta}(e')$ defined by
		\begin{align}\label{eq.defA}
			A_{\alpha, \beta}(e') := \sum_{\substack{\cu,\cu'\in \pPN_{h}(\cu_m) \\ e \in \Ed(\cu_m) }} |\cu|^\alpha|\cu'|^\beta\boldsymbol{1}_{\{e\subset \rr\cu \cap \rr\cu', e' \subset \rr\cu' \neq \emptyset\}}.
		\end{align}
		Here $\sum_{\substack{|F_\circ|=i-i_1\\F_\circ \subset  \Ed(\cu_m)}} \Ind{F_\circ \subset \Ed(\rr \cu')}$ contributes a weight $ |\rr \cu'|^{i-i_1}$. This concludes
		\begin{multline*}
			\sum_{\substack{|F|=i\\F\subset \Ed(\cu_m)}}\sum_{\substack{\cu \in \pPN_{h}(\cu_m)} }|\cu|^{\alpha} \sum_{e\in \Ed\Ll(\rr\cu\Rr)} |\na V_m(F,0)(e)|^2 \\
			\leq  C \sum_{i_1 = 0}^{i} \sum_{\substack{|F_*|=i_1\\F_* \subset  \Ed(\cu_m)}}\sum_{e'\in \EE \Ll(\cCb^{F_*}(\cu_m)\Rr)} A_{\alpha, i-i_1+1}(e') |\na V_m(F_*,0)(e')|^2 .
		\end{multline*}
		The weighted sum on the {\rhs} is treated in Lemma~\ref{lem.weight_estimate}, which yields
		\begin{multline}
			\sum_{\substack{|F|=i\\F\subset \Ed(\cu_m)}}\sum_{\substack{\cu \in \pPN_{h}(\cu_m)} }|\cu|^{\alpha} \sum_{e\in \Ed\Ll(\rr\cu\Rr)} |\na V_m(F,0)(e)|^2 \\
			\leq  C \sum_{i_1 = 0}^{i} \sum_{\substack{|F_*|=i_1\\F_* \subset  \Ed(\cu_m)}} \sum_{\substack{\cu \in \pPN_{h+1}(\cu_m)} }|\cu|^{\alpha+4+i-i_1} \sum_{e' \in \EE\Ll(\cCp^{F_*}(\rr\cu \cap \cu_m)\Rr)}  |\na V_m(F_*,0)(e')|^2.
		\end{multline}
		Here we also use the $N$-$\stable$ property, and the condition $\vert F_* \vert \leq i \leq N$. This is the desired result.
	\end{proof}
	\begin{remark}\label{rmk.plus1}
		The counterpart of \eqref{eq.coarse_grained_VF0} for the case $F = \emptyset$, i.e. $ V_m(\emptyset,0) = v_m$ is
		\begin{multline}\label{eq.coarse_grained_v}
		\sum_{\substack{\cu \in \pPN_{h}(\cu_m)} }|\cu|^{\alpha} \sum_{e\in \Ed\Ll(\rr\cu \cap \cu_m \Rr)} |\na v_m (e)|^2 \\
			\leq C  \sum_{\substack{\cu \in \pPN_{h+1}(\cu_m)} }|\cu|^{\alpha+4+i-i_1} \sum_{e\in \EE\Ll(\cCp^{F}(\rr\cu \cap \cu_m)\Rr)} (|\na v_m(e)|^2 + 1). 
		\end{multline}
		The only difference is ``+1" there, and it comes from the contribution of $\gamma'$ in \eqref{eq.shortcutPath2}. Recall the definition \eqref{eq.harmonicCube} that $v_{m,\xi} = \ell_{\xi}$ on $\partial \cu_m$ and the convention $\vert \xi \vert = 1$, so we have 
		\begin{align*}
			\Ll\vert \sum_{e'\in  \gamma' } \na  v_m(e') \Rr\vert = \Ll\vert \sum_{e'\in  \gamma' } \na \ell_{\xi}(e') \Rr\vert \leq \vert \gamma' \vert \leq \sum_{e'\in \EE \Ll(\cCp^F(\rr\cu' \cap \cu_m)\Rr)} 1.
		\end{align*} 
		The last inequality relies on the fact that $\cu' \in \pPN_{h}(\cu_m)$ is well-connected.
	\end{remark}

	The following technical estimate can be considered as a generalization of Lemma~\ref{le:overlap_counting}.
	\begin{lemma}[Weight estimate]\label{lem.weight_estimate}
		Given $k \in \N_+, N\in \N$ and $0 \leq h \leq k-1$, we define a weight function
		\begin{align}\label{eq.defA_another}
			A_{\alpha, \beta}(e') := \sum_{\substack{\cu,\cu'\in \pPN_{h}(\cu_m) \\ e \in \Ed(\cu_m) }} |\cu|^\alpha|\cu'|^\beta\boldsymbol{1}_{\{e\subset \rr\cu \cap \rr\cu', e' \subset \rr\cu'\}},
		\end{align}
		and there exists a finite positive constant $C(k,h,\rr,d)$, such that every cube ${\cu_m \in \gGN_k}$ is $N$-$\stable$ for the following estimate: for every $\alpha, \beta >0$ and every positive function ${w : \Ed(\cu_m) \to \R_+}$, we have
		\begin{align}\label{eq.weight_estimate}
			\sum_{e'\in \EE \Ll(\cCb^\#(\cu_m)\Rr)} A_{\alpha, \beta}(e')w(e')  \leq C \sum_{\substack{\cu \in \pPN_{h+1}(\cu_m)} }|\cu|^{\alpha+\beta+3} \sum_{e' \in \EE\Ll(\cCp^{\#}(\rr\cu \cap \cu_m)\Rr)} w(e'). 
		\end{align}
	\end{lemma}
	\begin{proof}
		We decompose $A_{\alpha, \beta}$ as the sum of two parts $A_{\alpha, \beta} = A^{(1)}_{\alpha, \beta} + A^{(2)}_{\alpha, \beta}$
		\begin{align*}
			A^{(1)}_{\alpha, \beta}(e') &:= \sum_{\substack{\cu,\cu'\in \pPN_{h}(\cu_m) \\ e \in \Ed(\cu_m) }} |\cu|^\alpha|\cu'|^\beta\boldsymbol{1}_{\{e\subset \rr\cu \cap \rr\cu', e' \subset \rr\cu'\}} \Ind{\frac{\size(\cu)}{\size(\cu')} \leq 1} ,\\ 
			A^{(2)}_{\alpha, \beta}(e') &:= \sum_{\substack{\cu,\cu'\in \pPN_{h}(\cu_m) \\ e \in \Ed(\cu_m) }} |\cu|^\alpha|\cu'|^\beta\boldsymbol{1}_{\{e\subset \rr\cu \cap \rr\cu', e' \subset \rr\cu'\}} \Ind{\frac{\size(\cu)}{\size(\cu')} > 1}.
		\end{align*}
		
		\smallskip
		\textit{Step~1: bound of $A^{(1)}_{\alpha, \beta}(e') $.}
		We apply at first the domination $\size(\cu) \leq \size(\cu')$
		\begin{align*}
			A^{(1)}_{\alpha, \beta}(e')  &\leq \sum_{\cu' \in \pPN_{h}(\cu_m)} \Ind{ e' \subset \rr\cu'} \sum_{\substack{\cu \in \pPN_{h}(\cu_m) }}\Ind{\rr \cu \cap \rr \cu' \neq \emptyset, \frac{\size(\cu)}{\size(\cu')} \leq 1} \sum_{e \in \Ed(\rr\cu)} \vert \cu\vert^\alpha \vert \cu'\vert^\beta \\
			&\leq \sum_{\cu' \in \pPN_{h}(\cu_m)} \Ind{ e' \subset \rr\cu'} \sum_{\substack{\cu \in \pPN_{h}(\cu_m) }}\Ind{\rr \cu \cap \rr \cu' \neq \emptyset, \frac{\size(\cu)}{\size(\cu')} \leq 1}   \rr^d  \vert \cu' \vert^{\alpha+\beta+1}\\
			&\leq  C \sum_{\substack{\cu' \in \pPN_{h}(\cu_m) }} \Ind{ e' \subset \rr\cu'} \vert \cu' \vert^{\alpha+\beta+1} \sum_{\substack{\cu \in \pPN_{h}(\cu_m)}}\Ind{\rr \cu \cap \rr \cu' \neq \emptyset}
		\end{align*}
		As we see, the weight of volume is now dominated by $ \vert \cu' \vert^{\alpha+\beta+1}$. We then naturally hope to reduce the final expression using $\cu'$, and some constrain is dropped in the third line.  We then apply the counting argument in Lemma~\ref{le:overlap_counting} that 
		\begin{align*}
			\sum_{\cu \in \pPN_{h}(\cu_m)}\Ind{\rr\cu \cap \rr\cu' \neq \emptyset} \leq \sum_{x \in \rr \cu'} \sum_{\cu \in \pPN_{h}(\cu_m)} \Ind{x \in \rr \cu} \leq C \rr \vert \cu' \vert,  
		\end{align*}
		which results in 
		\begin{align}\label{eq.A1Bound}
			A^{(1)}_{\alpha, \beta}(e') \leq  C \sum_{\substack{\cu' \in \pPN_{h}(\cu_m) }} \Ind{ e' \subset \rr\cu'} \vert \cu' \vert^{\alpha+\beta+2} .
		\end{align}

		\smallskip
		\textit{Step~2: bound of $A^{(2)}_{\alpha, \beta}(e')$.} This term can be done similarly 
		\begin{align*}
			A^{(2)}_{\alpha, \beta}(e') &\leq \sum_{\substack{\cu \in \pPN_{h}(\cu_m) }} \sum_{\cu' \in \pPN_{h}(\cu_m)} \Ind{ e' \subset \rr\cu', \rr\cu \cap \rr\cu' \neq \emptyset, \frac{\size(\cu)}{\size(\cu')} > 1}  \sum_{e \in \Ed(\rr\cu')} \vert \cu\vert^\alpha \vert \cu'\vert^\beta \\
			&\leq \sum_{\substack{\cu \in \pPN_{h}(\cu_m) }} \sum_{\cu' \in \pPN_{h}(\cu_m)} \Ind{ e' \subset \rr\cu', \rr\cu \cap \rr\cu' \neq \emptyset, \frac{\size(\cu)}{\size(\cu')} > 1} \rr^d \vert \cu\vert^{\alpha+\beta+1}.
		\end{align*}
		Here because $\size(\cu)$ is larger, we will reduce the sum over $\cu$. Different from $A^{(1)}_{\alpha, \beta}(e')$, we need to carry the condition $ {e' \subset \rr\cu'}$ to $\cu$. This relies on the three conditions 
		\begin{align*}
			{e' \subset \rr\cu'},  \qquad {\rr \cu \cap \rr \cu' \neq \emptyset}, \qquad{\size(\cu) > \size(\cu')}.
		\end{align*}
		They yield the comparability of size
		\begin{align*}
			\dist(e', \operatorname{center}(\cu)) &\leq \dist(e', \operatorname{center}(\cu')) + \dist(\operatorname{center}(\cu'), \operatorname{center}(\cu))\\
			& \leq \frac{\rr}{2} \size(\cu') + \frac{\rr}{2} \Ll(\size(\cu') + \size(\cu)\Rr)  \\
			& \leq 2\rr  \size(\cu),
		\end{align*}
		which implies
		\begin{align*}
			e' \subset 2\rr \cu .
		\end{align*}
		This allows us to simplify the estimate of $A^{(2)}_{\alpha, \beta}(e')$ as 
		\begin{align*}
			A^{(2)}_{\alpha, \beta}(e') \leq C \sum_{\substack{\cu \in \pPN_{h}(\cu_m) }} \Ind{e' \subset 2\rr \cu} \vert \cu\vert^{\alpha+\beta+1}\sum_{\cu' \in \pPN_{h}(\cu_m)} \Ind{ \rr\cu \cap \rr\cu' \neq \emptyset}  .
		\end{align*}
		We apply Lemma~\ref{le:overlap_counting}  
		\begin{align*}
			\sum_{\cu' \in \pPN_{h}(\cu_m)}\Ind{\rr\cu \cap \rr\cu' \neq \emptyset} \leq \sum_{x \in \rr \cu} \sum_{\cu' \in \pPN_{h}(\cu_m)} \Ind{x \in \rr \cu'} \leq C \rr \vert \cu \vert,
		\end{align*}
		which gives us the upper bound 
		\begin{align}\label{eq.A2Bound}
			A^{(2)}_{\alpha, \beta}(e') \leq C \sum_{\substack{\cu \in \pPN_{h}(\cu_m) }} \Ind{e' \subset 2\rr \cu} \vert \cu\vert^{\alpha+\beta+2}.
		\end{align}
		
		\smallskip

		\textit{Step 3: shrink of parameter in coarsened cubes.} 
		Combining the estimates  \eqref{eq.A1Bound} and \eqref{eq.A2Bound} above for $A^{(1)}_{\alpha, \beta}(e') , A^{(2)}_{\alpha, \beta}(e') $, and the fact $w$ is positive, we will get
		\begin{align}\label{eq.ABound}
			\sum_{e'\in \EE \Ll(\cCb^\#(\cu_m)\Rr)} A_{\alpha, \beta}(e') w(e')   
			\leq   C \sum_{\substack{\cu \in \pPN_{h}(\cu_m) }}  \vert \cu\vert^{\alpha+\beta+2}  \sum_{\substack{e' \in \EE \Ll(\cCb^\#(\cu_m)\Rr) \\ e' \subset 2\rr \cu}}  w(e') .
		\end{align}
		The result \eqref{eq.ABound} differs slightly from that in \eqref{eq.weight_estimate} for scale factor $2\rr$ instead of $\rr$. We will absorb this tiny factor using the good cubes in a larger scale. Let us develop the {\rhs} of \eqref{eq.ABound} with the partition $\pPN_{h+1}(\cu_m)$
		\begin{multline*}
			\sum_{\substack{\cu \in \pPN_{h}(\cu_m) }}  \vert \cu\vert^{\alpha+\beta+2}  \sum_{\substack{e' \in \EE \Ll(\cCb^\#(\cu_m)\Rr) \\ e' \subset 2\rr \cu}}  w(e') \\
			= \sum_{\widehat{\cu} \in \pPN_{h+1}(\cu_m)} \sum_{ \substack{\cu \in \sSl_{\gGN_{h}}(\widehat{\cu}) } }   \vert \cu\vert^{\alpha+\beta+2}  \sum_{\substack{e' \in \EE \Ll(\cCb^\#(\cu_m)\Rr) \\ e' \subset 2\rr \cu}}  w(e')
		\end{multline*}
		For every $\cu \in \sSl_{\gGN_{h}}(\widehat{\cu})$, thanks to \eqref{eq.sizeMax} we have $\size(\cu) \leq \frac{1}{100} \size (\widehat{\cu})$, thus 
		\begin{align*}
			2\rr \cu \subset 4 \cu \subset \frac{103}{100} \widehat{\cu}.
		\end{align*}

		Then we obtain
		\begin{align*}
			&\sum_{\substack{\cu \in \pPN_{h}(\cu_m) }}  \vert \cu\vert^{\alpha+\beta+2}  \sum_{\substack{e' \in \EE \Ll(\cCb^\#(\cu_m)\Rr) \\ e' \subset 2\rr \cu}}  w(e') \\
			&\leq  \sum_{\substack{\widehat{\cu} \in \pPN_{h+1}(\cu_m)  }} \sum_{ \substack{\cu \in \sSl_{\gGN_{h}}(\widehat{\cu})} }   \vert \cu\vert^{\alpha+\beta+2}  \sum_{\substack{e' \in \EE \Ll(\cCb^\#(\cu_m)\Rr) \\ e' \subset \frac{103}{100} \widehat{\cu}}}  w(e')\\
			&\leq \sum_{\substack{\widehat{\cu} \in \pPN_{h+1}(\cu_m)  }} \vert \widehat{\cu}\vert^{\alpha+\beta+3}  \sum_{\substack{e' \in \EE \Ll(\cCp^\#(\rr \widehat{\cu} \cap \cu_m)\Rr)}}  w(e').
		\end{align*}
		In the last line, Lemma~\ref{lem.detouring_path} implies that the edge connecting to $\EE \Ll(\cCb^\#(\cu_m)\Rr)$ in $\frac{103}{100} \widehat{\cu}$ should also connected to $\EE \Ll(\cCp^\#(\rr \widehat{\cu} \cap \cu_m)\Rr)$. This concludes the proof.
	\end{proof}
	\begin{remark}
		In the last Step~3 of the proof, we insist on reducing the factor for two reasons. Firstly, our pyramid partition is associated to the factor $\rr$, and general $2\rr \cu$ is not necessarily a well-connected cube. Secondly, without this reduction, the factor will increase in every implementation, which breaks the structure of induction. 
	\end{remark}

	\subsection{Proof of Proposition~\ref{pr:key_estimate} with $j=0$}\label{subsec.basis}

	In the proof, we fix some $N, k \in \N_+$ large enough, and denote by
	\begin{align}\label{eq.defIalpha}
		I^{\,(N,k)}_{m,\alpha, h}(i,j) := \frac{1}{\vert \cu_m \vert}\sum_{\substack{|F|=i\\F\subset \Ed(\cu_m)}}\Ll(\sum_{\substack{\cu \in \pPN_{h}(\cu_m)}} |\cu|^{\alpha}\sum_{e\in \Ed(\rr \cu \cap \cu_m)} |\na V_m(F, j)(e)|^2 \Rr),
	\end{align}
	and its version on the maximal cluster
	\begin{align}\label{eq.defIalpha_cluster}
		I^{*  \, (N,k)}_{m,\alpha, h}(i,j) 
		:= \frac{1}{\vert \cu_m \vert}\sum_{\substack{|F|=i\\F\subset \Ed(\cu_m)}}\Ll(\sum_{\substack{\cu \in \pPN_{h}(\cu_m)}} |\cu|^{\alpha}\sum_{e\in \EE(\cCp^{F}(\rr \cu \cap \cu_m))} |\na V_m(F, j)(e)|^2 \Rr).
	\end{align}
	These quantities are exactly the normalized version on the {\lhs} and {\rhs} of \eqref{eq.coarse_grained_VF0}. Our object is to give a uniform bound with respect to $m$. The idea is an induction, and in the procedure we always reduce $i,j$ with the price to increase $\alpha, h$.
	
	\begin{proof}[Proof of Proposition \ref{pr:key_estimate} with $j=0$]
		We divide the proof into two steps.
		
		\textit{Step~1: basis.}
		Using the coarse-graining argument in \eqref{eq.coarse_grained_v}, we have 
		\begin{align*}
			I^{\,(N,k)}_{m, \alpha, h}(0,0) \leq C \Ll(I^{*  \, (N,k)}_{m, \alpha+4, h+1} (0,0) + 1\Rr) \leq C.
		\end{align*}
		We explain further the second inequality. We need to estimate the weighted Dirichlet energy \eqref{eq.defIalpha_cluster} for $V_{m,\xi}(\emptyset, 0) = v_{m,\xi} = \ell_{\xi} + \phi_{m,\xi}$. The contribution by $\ell_{\xi}$ can be calculated directly, while \eqref{eq.harmonicCube} establishes the equation for the local corrector  $\phi_{m,\xi}$
		\begin{equation*}
			\begin{cases}
				-\nabla \cdot (\a \nabla \phi_{m,\xi})= \nabla \cdot (\a \nabla \ell_{\xi}) & \text { in } int(\cu_m), \\ 
				\phi_{m,\xi} = 0 & \text { on } \partial \cu_m.
			\end{cases}
		\end{equation*}
		Therefore, Proposition~\ref{prop.weighted_L2} applies and yields a uniform bound.
		
		\textit{Step~2: set-up of the induction.} We apply the coarse-graining estimate \eqref{eq.coarse_grained_VF0} at first
		\begin{equation}\label{eq.V0_coarse_grain}
			\begin{split}
				& I^{\,(N,k)}_{m, \alpha, h}(i,0)\\
				&\leq C \sum_{i_1 = 0}^{i} I^{*  \, (N,k)}_{m,\alpha+4+i-i_1, h+1}(i_1,0) \\
				& \leq C \Ll(I^{*  \, (N,k)}_{m, \alpha+4, h+1}(i,0)  + \sum_{\ell = 1}^{i} I^{\,(N,k)}_{m, \alpha + \ell + 4, h+1}(i-\ell,0)\Rr).
			\end{split}
		\end{equation} 
		From the second line to the third line, we dominate $I^*_m$ trivially by $I_m$ with the term of lower degree $i_1 \leq i-1$, with a change of variable $\ell = i - i_1$. It remains to estimate $I^{*  \, (N,k)}_{m, \alpha+4, h+1}(i,0)$ in the last line. Viewing that $V_m(F,0)$ follows the perturbed corrector equation \eqref{eq.recurrence_eq}, where $\W_m(F,0)$ has the explicit expression as indicated in \eqref{e:induction_eqW}
		\begin{align*}
			\W_m(F,0) = \sum_{e'\in F}(\a^{e'}-\a)\na V_m(F\setminus\{e'\},0).
		\end{align*}
		Thus, it is exactly a field on $\a^F$ and \eqref{eq.weighted_L2} applies,
		\begin{multline}\label{eq.V0_Poisson}
			I^*_m(\alpha+4, h+1; i,0) \\
			\leq  C\frac{1}{\vert \cu_m \vert}\sum_{\substack{|F|=i \\ F\subset \Ed(\cu_m)}}\sum_{\substack{\cu \in \pPN_{k,h+2}(\cu_m)}}|\cu|\sum_{e\in \EE\Ll(\cCp^F(\rr\cu \cap \cu_m)\Rr)} |\W_m(F,0)(e)|^2
		\end{multline}
		The conditions $h \geq \frac{(\alpha+5)(2+\eps_{Me})}{\eps_{Me}}$ and $N \geq i$ are also required here, in order to ensure the integrability and stability in Proposition~\ref{prop.weighted_L2}. We then have the following observation
		\begin{align*}
			\W_m(F,0)(e) = \na V_m(F\setminus\{e\},0)(e) \Ind{e \in F},
		\end{align*}
		which simplifies the expression in \eqref{eq.V0_Poisson} with the change of variable $F' = F \setminus \{e\}$
		\begin{equation}\label{eq.V0_W0}
			\begin{split}
				& I^{*  \, (N,k)}_{m, \alpha+4, h+1}(i,0) \\
				&\leq \frac{C}{\vert \cu_m \vert}\sum_{\substack{|F|=i \\ F\subset \Ed(\cu_m)}}\sum_{\substack{\cu \in \pPN_{k,h+2}(\cu_m)}}|\cu|\sum_{e\in \EE\Ll(\cCp^F(\rr\cu \cap \cu_m)\Rr)} |\na V_m(F\setminus\{e\},0)(e) |^2 \Ind{e \in F}\\
				&\leq  \frac{C}{\vert \cu_m \vert}\sum_{e \in \Ed(\cu_m)}\sum_{\substack{|F'|=i-1 \\ F'\subset \Ed(\cu_m)}}\sum_{\substack{\cu \in \pPN_{k,h+2}(\cu_m)}}|\cu| |\na V_m(F',0)(e) |^2 \Ind{e\in \EE\Ll(\cCp^{F' \cup \{e\}}(\rr\cu \cap \cu_m)\Rr)} \\
				&\leq \frac{C}{\vert \cu_m \vert} \sum_{\substack{|F'|=i-1 \\ F'\subset \Ed(\cu_m)}}\sum_{\substack{\cu \in \pPN_{k,h+2}(\cu_m)}}|\cu| \sum_{e\in \Ed(\rr \cu \cap \cu_m)} |\na V_m(F',0)(e) |^2  \\
				& \leq C I^{\,(N,k)}_{m,1, h+2}(i-1,0).
			\end{split}
		\end{equation}
		We combine \eqref{eq.V0_coarse_grain} and \eqref{eq.V0_W0} 
		\begin{align*}
			I^{\,(N,k)}_{m, \alpha, h}(i,0) \leq C \Ll(I^{\,(N,k)}_{m,1, h+2}(i-1,0) + \sum_{\ell = 1}^{i} I^{\,(N,k)}_{m, \alpha + \ell + 4, h+1}(i-\ell,0)\Rr).
		\end{align*}
		Therefore, we can always control the $I^{\,(N,k)}_{m, \alpha, h}(i,0)$ using the lower order term. Since the index $h$ can increase by $2$ when $i$ decreases by $1$, the condition \eqref{eq.condition_Nk_main} is enough to complete the induction for $j=0$.
		
	\end{proof}

	\section{Weighted  $\ell^1$-$L^2$ estimate $I_m(i,j), j \geq 1$}\label{sec.KeyEstimate2}
	Now we turn to the general weighted $\ell^1$-$L^2$ estimate for $V_m(F,j)$ with $j \geq 1$. Compared to $V_m(F,0)$, this case contains two additional challenges:
	\begin{enumerate}
		\item Firstly, although we assume that $v_m$ follows the constant extension on holes, we cannot expect the same property for $V_m(F, j)$, because  $G$ goes over all the possible configurations in the definition $V_m(F, j) = \sum_{\substack{|G|=j\\G\subset \Ed(\cu_m) \setminus F}} D_{F\cup G} v_m$.
		\item Secondly, a naive application of Cauchy--Schwarz inequality gives  
		\begin{equation}\label{eq.CS_naive}
			\begin{split}
				&\sum_{\substack{|F|=i\\F\subset \Ed(\cu_m)}} \sum_{e\in \Ed(\cu)} |\na V_m(F, j)(e)|^2 \\
				&= \sum_{\substack{|F|=i\\F\subset \Ed(\cu_m)}} \sum_{e\in \Ed(\cu)} \Ll|  \sum_{\substack{|G|=j\\G\subset \Ed(\cu_m) \setminus F}} \na D_{F\cup G} v_m (e) \Rr|^2 \\
				&\leq  \sum_{\substack{|H|=i+j\\ H \subset \Ed(\cu_m)}} \sum_{e\in \Ed(\cu)} |\na V_m(H, 0)(e)|^2 \cdot { {\vert \Ed(\cu_m)\vert - i} \choose j}.
			\end{split}
		\end{equation}
		The weighted $\ell^1$-$L^2$ estimate of $V_m(H, 0)$ is bounded, but there remains a huge factor ${ {\vert \Ed(\cu_m)\vert - i} \choose j}$ which explodes when $m$ grows.
	\end{enumerate}
	For these reasons, we introduce a technique called \emph{hole-separation} in this section, and it will help overcome the difficulties mentioned above.

	\subsection{Hole separation}\label{subsec.Hole}
	We propose at first a generalized version of $V_m(F,j)$. Given $E \subset \Ed(\cu_m)$ as a subset of edges, then we define
	\begin{equation}\label{eq.V_F_j_Generalized}
		V_m(F, j \vert E) := \sum_{\substack{|G|=j\\G\subset (\Ed(\cu_m) \setminus F) \setminus E}} D_{F\cup G} v_m.
	\end{equation}
	The set $E$ in \eqref{eq.V_F_j_Generalized} means ``no Glauber derivatives involve $E$". The following elementary observation will be useful. In practice, the size of the set $E$ will be of constant order, thus the factor in \eqref{eq.hole_separation_bound} is much smaller than the naive bound in \eqref{eq.CS_naive}.
	
	\begin{lemma}[Hole separation]
		For every $E \subset \Ed(\cu_m)$, we have 		
		\begin{align}\label{eq.hole_separation_id}
			V_m(F,j) = \sum_{\ell=0}^j \sum_{\substack{|G|= \ell \\ G\subset E \setminus F}} V_m(F \cup G,j-\ell \vert E),
		\end{align}
		thus it satisfies
		\begin{align}\label{eq.hole_separation_bound}
			\vert V_m(F,j) \vert^2 \leq (j+1)\sum_{\ell=0}^j \vert E \vert^\ell \sum_{\substack{|G|= \ell \\ G\subset E \setminus F}} \vert V_m(F \cup G,j-\ell \vert E) \vert^2.
		\end{align}
	\end{lemma}
	\begin{proof}
		Recall the definition of $V_m(F,j)$ in \eqref{eq.V_F_j}, then we decompose it by counting the Glauber derivatives involving $E$
		\begin{align*}
			V_m(F, j) &= \sum_{\substack{|G|=j\\G\subset \Ed(\cu_m) \setminus F}} D_{F\cup G} v_m\\
			&= \sum_{\ell=0}^j \sum_{\substack{|G|=\ell\\ G \subset E \setminus F}} \sum_{\substack{|H|=j-\ell\\ H \subset (\Ed(\cu_m) \setminus F) \setminus E  }} D_{F  \cup G \cup H}  v_m\\
			&= \sum_{\ell=0}^j \sum_{\substack{|G|= \ell \\ G\subset E \setminus F}} V_m(F \cup G,j-\ell \vert E).
		\end{align*}
		We apply the definition \eqref{eq.V_F_j_Generalized} in the last line and obtain \eqref{eq.hole_separation_id}. Then \eqref{eq.hole_separation_bound} is just an application of Cauchy--Schwarz inequality 
		\begin{align*}
			\vert V_m(F,j) \vert^2 &= \Ll \vert \sum_{\ell=0}^j  \sum_{\substack{|G|= \ell \\ G\subset E \setminus F}} V_m(F \cup G,j-\ell \vert E) \Rr \vert^2 \\
			&\leq (j+1)  \sum_{\ell=0}^j  \Ll \vert \sum_{\substack{|G|= \ell \\ G\subset E \setminus F}} V_m(F \cup G,j-\ell \vert E) \Rr \vert^2 \\
			&\leq (j+1)  \sum_{\ell=0}^j \vert E \vert^{\ell} \sum_{\substack{|G|= \ell \\ G\subset E \setminus F}}  \Ll \vert V_m(F \cup G,j-\ell \vert E) \Rr \vert^2 .
		\end{align*}
	\end{proof}
	
	In the following paragraph, we will use $V_m(F, j \vert E)$ to dominate the weighted $\ell^1$-$L^2$ estimate. With a slight abuse of notation, we define 
	\begin{align}\label{eq.V_F_j_Generalized_cube}
		V_m(F, j \vert \cu) := V_m(F, j \vert \Ed(\rr \cu)) =   \sum_{\substack{|G|=j\\G\subset (\Ed(\cu_m) \setminus F) \setminus \Ed(\rr \cu)}} D_{F\cup G} v_m.
	\end{align}
	We also define the weighted norm like \eqref{eq.defIalpha} using the hole separation
	\begin{align}\label{eq.defIalpha_hole}
		\tilde{I}^{\, (N,k)}_{m, \alpha, h}(i,j) := \frac{1}{\vert \cu_m \vert}\sum_{\substack{|F|=i\\F\subset \Ed(\cu_m)}}\Ll(\sum_{\substack{\cu \in \pPN_{h}(\cu_m)}} |\cu|^{\alpha}\sum_{e\in \Ed(\frac{16}{15} \cu \cap \cu_m)} \Ll|\na V_m(F, j \vert \cu)(e) \Rr|^2 \Rr).
	\end{align}
	Following the spirit of \eqref{eq.hole_separation_bound}, one can dominate the quantity $I_m$ using $\tilde{I}_m$.
	\begin{lemma}\label{lem.I_by_It}
		For every $\cu_m \in \gGN_{k}$ with $0 \leq h \leq k-2$, the following estimate holds
		\begin{equation}\label{eq.I_by_It}
			I^{\,(N,k)}_{m,\alpha, h}( i,j)  \leq (j+1)\sum_{\ell=0}^j \tilde{I}^{\, (N,k)}_{m, \alpha+\ell+1, h+1}(i+\ell,j-\ell).
		\end{equation}
	\end{lemma}
	\begin{proof}
		We study one term in \eqref{eq.defIalpha} using the pyramid partition like \eqref{eq.L2Subdivision}
		\begin{multline}\label{eq.I_by_It_step1}
			\sum_{\substack{\cu' \in \pPN_{h}(\cu_m)}} |\cu'|^{\alpha}\sum_{e\in \Ed(\rr \cu' \cap \cu_m)} |\na V_m(F, j)(e)|^2 \\
			=  \sum_{ \substack{\cu \in \pPN_{h+1}(\cu_m)  }}\sum_{\substack{\cu' \in \sSl_{\gGN_{h}}(\cu)}}  |\cu'|^{\alpha}\sum_{e\in \Ed(\rr \cu' \cap \cu_m)} |\na V_m(F, j)(e)|^2.\\
		\end{multline}
		We then apply \eqref{eq.hole_separation_bound} to every partition cube $\cu \in \pPN_{h+1}(\cu_m)$: for every ${\cu' \in \sSl_{\gGN_{h}}(\cu)}$, we have
		\begin{multline}\label{eq.I_by_It_step2}
			|\cu'|^{\alpha}\sum_{e\in \Ed(\rr \cu' \cap \cu_m)} |\na V_m(F, j)(e)|^2 \\
			\leq |\cu'|^{\alpha}  \sum_{e\in \Ed(\rr \cu' \cap \cu_m)} (j+1)\sum_{\ell=0}^j \vert \cu \vert^\ell \sum_{\substack{|G|= \ell \\ G\subset \Ed(\rr \cu) \setminus F}} \vert \nabla  V_m(F \cup G,j-\ell \vert \cu) (e)\vert^2.
		\end{multline}
		Combing \eqref{eq.I_by_It_step1} and \eqref{eq.I_by_It_step2}, we obtain that
		\begin{multline*}
			\sum_{\substack{\cu' \in \pPN_{h}(\cu_m)}} |\cu'|^{\alpha}\sum_{e\in \Ed(\rr \cu' \cap \cu_m)} |\na V_m(F, j)(e)|^2 \\
			\leq (j+1) \sum_{\ell=0}^j \sum_{ \substack{\cu \in \pPN_{h+1}(\cu_m)  }} \sum_{e\in \Ed(\frac{16}{15} \cu \cap \cu_m)}  \vert \cu \vert^{\alpha + \ell +1} \sum_{\substack{|G|= \ell \\ G\subset \Ed(\rr \cu) \setminus F}} \vert \nabla  V_m(F \cup G,j-\ell \vert \cu) (e)\vert^2.
		\end{multline*}
		In the last line, we absorb $\vert \cu' \vert^\alpha$ using its parent cube $\vert \cu \vert^\alpha$. The sum $\sum_{\substack{\cu' \in \sSl_{\gGN_{h}}(\cu)}}$ in \eqref{eq.I_by_It_step1} contributes another factor $\vert \cu \vert$. Moreover, \eqref{eq.sizeMax} implies $\rr \cu' \subset \frac{16}{15} \cu$ and we shrink the domain of integration. We put this bound back to the definition of \eqref{eq.defIalpha} and conclude \eqref{eq.I_by_It}.
	\end{proof}

	In the definition of $\tilde{I}_m$, we insist on using the factor $\frac{16}{15}$, which allows us to use the coarse-graining argument later. Lemma~\ref{lem.I_by_It} solves the second difficulty mentioned at the beginning of the section, then we turn to the first difficulty. We will see that, the function $V_m(F,j \vert \cu)$ is better than $V_m(F,j)$, since all the Glauber derivatives in $\rr \cu$ come from $F$. It can further recover similar results like Lemma~\ref{lem.decom_DF} and Corollary~\ref{cor.DF_cancel} locally in small cube.
	
	\begin{lemma}\label{lem.decom_DF_j}
		Let $F_*, F_\circ$ follow the decomposition \eqref{eq.def_WIP_new}, and $v_{m}$ follow the canonical grain extension in Definition~\ref{def.ConstantExtension_canonical}. Given $\cu_m \in  \gGN_{k}$ and $\cu \in \pPN_{h}$, we have the identity
		\begin{equation}\label{eq.decom_DF_j}
			\forall x \in \frac{16}{15} \cu, \qquad V_m(F,j \vert \cu)(x) = \nabla_{F_\circ} V_m(F_*,j \vert \cu)(x),
		\end{equation}
		and following observations:
		\begin{itemize}
			\item  $V_m(F,j)$ is then constant on every holes in $\frac{16}{15} \cu$.
			\item For  $x \notin  \cCb^F(\cu_m) \cap \frac{16}{15} \cu$, if $F_\circ \nsubseteq \EE(\OO^{F_\circ}(x))$, then $V_m(F,j \vert \cu)(x) = 0$.
			\item For $x \in \cCb^F(\cu_m) \cap \frac{16}{15} \cu$, if $F_\circ \neq \emptyset$, then $V_m(F,j \vert \cu)(x) = 0$.
		\end{itemize}
	\end{lemma}
	\begin{proof}
		We apply the definition \eqref{eq.V_F_j_Generalized_cube} to $x \in \frac{16}{15} \cu$
		\begin{equation}\label{eq.decom_DF_dropG}
			\begin{split}
				&V_m(F, j \vert \cu)(x)	\\
				&= \sum_{\substack{|G|=j\\G\subset (\Ed(\cu_m) \setminus F) \setminus \Ed(\rr \cu)}}  \sum_{G' \subset G} (-1)^{\vert G \setminus G' \vert}  \sum_{F' \subset F} (-1)^{\vert F \setminus F' \vert}  v_m^{F' \cup G'}([x]^{F' \cup G'})\\
				&= \sum_{\substack{|G|=j\\G\subset (\Ed(\cu_m) \setminus F) \setminus \Ed(\rr \cu)}}  \sum_{G' \subset G} (-1)^{\vert G \setminus G' \vert}  \sum_{F' \subset F} (-1)^{\vert F \setminus F' \vert}  v_m^{F' \cup G'}([x]^{F'}).
			\end{split}
		\end{equation}
		We can replace  $[x]^{F' \cup G'}$ by $[x]^{F'}$ because the opening of edges in $G$ only influence the domain $\cu_m \setminus \rr \cu$, which will not change the connectivity of holes in $\frac{16}{15} \cu$. Once we drop $G'$ in the operator of grain, the remaining analysis is similar as Lemma~\ref{lem.decom_DF} and we have 
		\begin{align*}
			\sum_{F' \subset F} (-1)^{\vert F \setminus F' \vert}  v_m^{F' \cup G'}([x]^{F'}) = (D_F v_m^{G'})(x) = \nabla_{F_\circ} D_{F_*} v_m^{G'}(x).
		\end{align*}
		Then we put the equation above back to the last line of \eqref{eq.decom_DF_dropG} and obtain \eqref{eq.decom_DF_j}
		\begin{align*}
			V_m(F, j \vert \cu)(x) &= \sum_{\substack{|G|=j\\G\subset (\Ed(\cu_m) \setminus F) \setminus \Ed(\rr \cu)}}  \sum_{G' \subset G} (-1)^{\vert G \setminus G' \vert} \nabla_{F_\circ} D_{F_*} v_m^{G'}(x) \\
			&=  \nabla_{F_\circ} \Ll(\sum_{\substack{|G|=j\\G\subset (\Ed(\cu_m) \setminus F) \setminus \Ed(\rr \cu)}} D_{F_* \cup G} v_m \Rr)(x)  \\
			&= \nabla_{F_\circ} V_m(F_*,j \vert \cu)(x).
		\end{align*}
		The constant  property on holes comes directly from \eqref{eq.V_F_j_Generalized_cube} and Definition~\ref{def.ConstantExtension_canonical}. Then only $F$ contributes to the Glauber derivatives in $\frac{16}{15} \cu$, so the other two observations can be deduced from the step \eqref{eq.decom_DF_j} following Corollary~\ref{cor.DF_cancel}.
	\end{proof}

	Lemma~\ref{lem.decom_DF_j} allows us to integrate the hole separation into the coarse-graining estimate. However, $V(F,j \vert \cu)$ also contains disadvantages: Lemma~\ref{lem.decom_DF_j} only applies to the associated local small cube $\cu$, while all the results in Section~\ref{sec.Poisson} are stated globally on $\cu_m$. We will not develop the analogue of corrector equation about $V(F,j \vert \cu)$, but establish a coarse-graining argument for $\tilde{I}_m$, and then bring the analysis back to $V(F,j)$.
	
	\begin{lemma}[Coarse-graining estimate of $\tilde{I}_m$]\label{lem.coarse_grained_VFj}
		Let $v_{m}$ follow the canonical grain extension in Definition~\ref{def.ConstantExtension_canonical} and fix $\rr \in (\frac{50}{27}, 2)$. Every $\cu_m \in  \gGN_{k}$  satisfies the following estimates:
		\begin{itemize}
			\item For every $F \subset \Ed(\cu_m)$ with $\vert F \vert \leq N - j$, every $\cu' \subset \cu$ satisfying 
			\begin{align*}
				0 \leq h' \leq h \leq k-1, \quad &\cu \in \pPN_{h}(\cu_m), \quad \cu' \in  \pPN_{h'}(\cu_m), 
			\end{align*}
			and every edge $e \in  \Ed\Ll(\frac{16}{15} \cu' \cap \cu_m\Rr)$, we have
			\begin{multline}\label{eq.coarse_grained_VFj_point}
				|\na V_m(F,j \vert \cu)(e)| \leq \sum_{e' \in \EE(\cCp^F(\rr \cu' \cap \cu_m))} \vert  \nabla V_m(F,j)(e')\vert  \Ind{F_\circ = \emptyset}   \\
				+ \sum_{\ell=1}^j \sum_{\substack{|G|= \ell \\ G\subset \Ed(\rr \cu) \setminus F}} \sum_{e' \in \EE(\cCp^F(\rr \cu' \cap \cu_m))} \Ll\vert \nabla V_m(F \cup G,j-\ell \vert \cu)(e') \Rr\vert \Ind{F_\circ = \emptyset} \\
				+2^{\vert F_\circ \vert} \sum_{e' \in \EE(\cCp^F(\rr \cu' \cap \cu_m))}  |\na V_m(F_*,j \vert \cu)(e')| \Ind{F_\circ \subset \Ed(\rr \cu')}\Ind{F_\circ \neq \emptyset}.
			\end{multline}
			\item There exists a finite positive constant $C \equiv C(i,j) = 2^i(j+1)$, such that for all positive integers satisfying $N \geq i+j$ and $k \geq h+2$, we have
			\begin{align}\label{eq.coarse_grained_VFj}
				\tilde{I}^{\, (N,k)}_{m,\alpha, h}(i,j) \leq C \Ll(I^{*  \, (N,k)}_{m,\alpha+1, h}(i,j) + \sum_{\substack{0 \leq \ell' \leq i, 0 \leq \ell \leq j\\ (\ell',\ell) \neq (0,0)}}\tilde{I}^{\, (N,k)}_{m,\alpha+\ell+ \ell' + 2, h+1}(i-\ell'+\ell,j-\ell)\Rr).
			\end{align}
		\end{itemize}
	\end{lemma}
	\begin{proof}
		We divide the proof into 2 steps.
		
		\textit{Step~1: pre-treatment.} Concerning the point-wise estimate \eqref{eq.coarse_grained_VFj_point}, we observe at first that
		\begin{align}\label{eq.coarse_grained_VFj_point_pre}
			|\na V_m(F,j \vert \cu)(e)| 
			\leq 2^{\vert F_\circ \vert} \sum_{e' \in \EE(\cCp^F(\rr \cu' \cap \cu_m))} |\na V_m(F_*,j \vert \cu)(e')|  \Ind{F_\circ \subset \Ed(\rr \cu')}.
		\end{align} 
		The proof follows exactly Step~1 in the proof of Lemma~\ref{lem.coarse_grained_VF0}, especially \eqref{eq.shortcutPath} and \eqref{eq.coarse_grained_VF0_Case2}. We remark that Lemma~\ref{lem.decom_DF_j} serves as a counterpart of Lemma~\ref{lem.decom_DF} and Corollary~\ref{cor.DF_cancel}. Meanwhile, Lemma~\ref{lem.detouring_path}, \eqref{eq.center_cube} together with the condition $\rr \in (\frac{50}{27}, 2)$ ensures that, all the open paths inside $\frac{16}{15} \cu'$ can be seen locally as parts of the maximal cluster $\cCp^{F}(\rr\cu' \cap \cu_m)$.
		
		The estimate \eqref{eq.coarse_grained_VFj_point_pre} with the choice  $\cu'=\cu$ also helps develop $\tilde{I}^{\, (N,k)}_{m, \alpha, h}(i,j)$
		\begin{equation}\label{eq.coarse_grained_VFj_Step1}
			\begin{split}
				& \vert \cu_m \vert \tilde{I}^{\, (N,k)}_{m, \alpha, h}(i,j) \\
				&=  \sum_{\substack{|F|=i\\F\subset \Ed(\cu_m)}}\sum_{\substack{\cu \in \pPN_{h}(\cu_m)}} |\cu|^{\alpha}\sum_{e\in \Ed(\frac{16}{15} \cu \cap \cu_m)} \Ll|\na V_m(F, j \vert \cu)(e) \Rr|^2\\
				&\leq 2^i \sum_{\substack{F=F_* \sqcup F_\circ\\|F|=i, F\subset \Ed(\cu_m)}}\sum_{\substack{\cu \in \pPN_{h}(\cu_m)}} |\cu|^{\alpha+1}\sum_{e\in \EE (\cCp^F(\rr\cu \cap \cu_m))} \Ll|\na V_m(F_*, j \vert \cu)(e) \Rr|^2  \Ind{F_\circ \subset \Ed(\rr \cu) } \\
				&\leq 2^i \sum_{\ell' = 0}^i \sum_{\substack{|F_*|=i-\ell'\\F_*\subset \Ed(\cu_m)}}\sum_{\substack{\cu \in \pPN_{h}(\cu_m)}} |\cu|^{\alpha+\ell'+1}\sum_{e\in \EE (\cCp^{F_*}(\rr\cu \cap \cu_m))} \Ll|\na V_m(F_*, j \vert \cu)(e) \Rr|^2.
			\end{split}
		\end{equation}
		The passage from the third line to the forth line is just a integration over $F_{\circ}$. When $\vert F_{\circ} \vert = \ell'$, this contributes a factor $ |\cu|^{\ell'}$ thanks to  the indicator $\Ind{F_\circ \subset \Ed(\rr \cu) }$.
		
		\smallskip
		
		\textit{Step~2: leading term.}	We then need to bring the upper bound back to $V_m(F,j)$. The idea is to reformulate \eqref{eq.hole_separation_id} as follows
		\begin{align}\label{eq.hole_separation_id_reformulate}
			V_m(F,j \vert \cu) =  V_m(F,j) - \sum_{\ell=1}^j \sum_{\substack{|G|= \ell \\ G\subset \Ed(\rr \cu) \setminus F}} V_m(F \cup G,j-\ell \vert \cu).
		\end{align}
		We apply the equation above to the term $F_\circ = \emptyset$ in \eqref{eq.coarse_grained_VFj_point_pre}, then a triangle inequality will entail \eqref{eq.coarse_grained_VFj_point}.

		The identity \eqref{eq.hole_separation_id_reformulate} also applies to the leading order term $\ell' = 0$ in the last line of \eqref{eq.coarse_grained_VFj_Step1}
		\begin{align*}
			&\sum_{\substack{|F|=i\\F\subset \Ed(\cu_m)}}\sum_{\substack{\cu \in \pPN_{h}(\cu_m)}} |\cu|^{\alpha+1}\sum_{e\in \EE (\cCp^F(\rr\cu \cap \cu_m))} \Ll|\na V_m(F, j \vert \cu)(e) \Rr|^2 \\
			&\leq  \sum_{\substack{|F|=i\\F\subset \Ed(\cu_m)}}\sum_{\substack{\cu \in \pPN_{h}(\cu_m)}} |\cu|^{\alpha+1}\sum_{e\in \EE (\cCp^F(\rr\cu \cap \cu_m))} \Ll|\na V_m(F, j)(e) \Rr|^2 \\
			& \qquad + \sum_{\ell=1}^j \sum_{\substack{|F|=i\\F\subset \Ed(\cu_m)}}\sum_{\substack{\cu \in \pPN_{h}(\cu_m)}}\sum_{\substack{|G|= \ell \\ G\subset \Ed(\rr \cu)  \setminus F}}  |\cu|^{\alpha+\ell'+1}\sum_{e\in \EE (\cCp^F(\rr\cu \cap \cu_m))} \Ll|\na  V_m(F \cup G,j-\ell \vert \cu)(e) \Rr|^2.
		\end{align*}
		Here the term $I^{*  \, (N,k)}_{m,\alpha+1, h}(i,j)$ already appears and the other terms are of lower order.  
		We insert this estimate back to \eqref{eq.coarse_grained_VFj_Step1} and obtain
		\begin{equation*}
			\begin{split}
				&  \tilde{I}^{\, (N,k)}_{m, \alpha, h}(i,j) \\
				& \leq C I^{*  \, (N,k)}_{m,\alpha+1, h}(i,j) \\
				& \quad + \frac{1}{\vert \cu_m \vert}\sum_{\ell' = 1}^i \sum_{\substack{|F|=i-\ell'\\F\subset \Ed(\cu_m)}}\sum_{\substack{\cu \in \pPN_{h}(\cu_m)}} |\cu|^{\alpha+\ell'+1}\sum_{e\in \EE (\cCp^F(\rr\cu \cap \cu_m))} \Ll|\na V_m(F, j \vert \cu)(e) \Rr|^2 \\
				& \quad + \frac{1}{\vert \cu_m \vert}\sum_{\ell'=1}^j \sum_{\substack{|F|=i+\ell'\\F\subset \Ed(\cu_m)}}\sum_{\substack{\cu \in \pPN_{h}(\cu_m)}}  |\cu|^{\alpha+\ell'+1}\sum_{e\in \EE (\cCp^F(\rr\cu \cap \cu_m))} \Ll|\na  V_m(F,j-\ell \vert \cu)(e) \Rr|^2.
			\end{split}
		\end{equation*}
		The lower order terms are slightly different from $\tilde{I}_m$, because the integration is locally over $\rr \cu$ instead of $\frac{16}{15} \cu$. Therefore, we apply the similar argument in the proof of \eqref{eq.I_by_It} to coarsen the cube
		\begin{multline}\label{eq.coarse_grained_VFj_pre}
			\tilde{I}^{\, (N,k)}_{m, \alpha, h}(i,j) \leq C \Ll(1 +  I^{*  \, (N,k)}_{m,\alpha+1, h+1}(i,j)\Rr) \\
			+  C\sum_{\ell' = 1}^i \sum_{\ell=0}^j \tilde{I}^{\, (N,k)}_{m,\alpha+\ell+ \ell' + 2, h+1}(i-\ell'+\ell,j-\ell) \\
			+  C\sum_{\ell' = 1}^j \sum_{\ell=0}^{j-\ell'} \tilde{I}^{\, (N,k)}_{m,\alpha+\ell+ \ell' + 2, h+1}(i+\ell'+\ell,j-\ell-\ell').
		\end{multline}
		With a change of variable $\ell'' = \ell + \ell'$ for the expression in the second line, it becomes
		\begin{align*}
			\sum_{\ell' = 1}^i \sum_{\ell=0}^{j-\ell'} \tilde{I}^{\, (N,k)}_{m,\alpha+\ell+ \ell' + 2, h+1}(i+\ell'+\ell,j-\ell-\ell') 
			\leq C \sum_{\ell'' = 1}^j  \tilde{I}^{\, (N,k)}_{m,\alpha+\ell'' + 2, h+1}(i+\ell'',j-\ell'').
		\end{align*}
		Thus the {\rhs} of \eqref{eq.coarse_grained_VFj_pre} can be unified, and we obtain the desired result \eqref{eq.coarse_grained_VFj}.
	\end{proof}

	\subsection{Estimate for perturbed corrector equation}\label{subsec.Inductionij1}
	We now turn to the perturbed corrector equation for general $j \geq 1$.		
	
	\begin{lemma}\label{lem.Is_by_I}
		There exists a finite positive constant $C = C(i,j,h,N,k,d,\rr)$, such that for every $\cu_m \in  \gGN_{k}$, the following estimate holds under the condition \eqref{eq.condition_hk_NEW}, $N \geq i+j$, and $\rr \in (\frac{50}{27}, 2)$
		\begin{multline}\label{eq.Is_by_I}
			I^{*  \, (N,k)}_{m,\alpha, h}(i,j) 
			\leq  C\Big( I^{\,(N,k)}_{m, 3, h+1}(i,j-1)+ I^{\,(N,k)}_{m, 3, h+1}(i+1,j-2)  \\
			\qquad + I^{\,(N,k)}_{m, 3, h+1}(i+1,j-1)+I^{\,(N,k)}_{m, 3, h+1}(i-1,j)\Big) \\
			+ C \sum_{\substack{0 \leq \ell' \leq i, 0 \leq \ell \leq j\\ (\ell',\ell) \neq (0,0)}}\tilde{I}^{\, (N,k)}_{m, \ell+ \ell' + 5, h+2}(i-\ell'+\ell,j-\ell).
		\end{multline}
	\end{lemma}
	
	\begin{proof}
		We aim to apply Proposition~\ref{prop.weighted_L2_NEW}, and a key step is to find the coarse-controlled function.
		
		\textit{Step~1: coarse-controlled function.}
		Given $e \in \Ed(\cu_m)$, we  fix a lexicographic order and denote by $\cu_{m,0}(e)$ the unique element in $\pPN_{0}(\cu_m)$ that contains $e$. We then define $\g_{m}(F,j \vert h)$ as a coarse-controlled function satisfying \eqref{eq.vg} for $v = V_m(F,j)$, using the hole separation with respect to $\pPN_{h}(\cu_m)$. We just treat one term $e$ near $\cu \in \pPN_{h}(\cu_m)$. Following the hole-separation identity \eqref{eq.hole_separation_id}
		\begin{align*}
			V_m(F,j) = \sum_{\ell=0}^j \sum_{\substack{|G|= \ell \\ G\subset \Ed(\rr \cu) \setminus F}} V_m(F \cup G,j-\ell \vert \cu).
		\end{align*}
		Then we apply the coarse-graining estimate \eqref{eq.coarse_grained_VFj_point} to the leading term $\ell = 0$ and obtain 
		\begin{align*}
			\vert \nabla V_m(F,j)(e)\vert \leq \vert \g_m(F,j \vert h)(e)\vert + \sum_{e' \in \EE(\cCp^F(\rr \cu_{m,0}(e) \cap \cu_m))} \vert  \nabla V_m(F,j)(e')\vert \Ind{F_\circ = \emptyset},  
		\end{align*}
		with the coarse-controlled function $\g_m(F,j \vert h)$ defined as
		\begin{align}\label{eq.defgm}
			\begin{split}
				&\g_m(F,j \vert h)(e) \\
				&:=  \sum_{\ell=1}^j \sum_{\substack{|G|= \ell \\ G\subset \Ed(\rr \cu) \setminus F}}  \Ll\vert \nabla V_m(F \cup G,j-\ell \vert \cu)(e) \Rr\vert \\
				&\quad + \sum_{\ell=1}^j \sum_{\substack{|G|= \ell \\ G\subset \Ed(\rr \cu) \setminus F}} \sum_{e' \in \EE(\cCp^F(\rr \cu_{m,0}(e) \cap \cu_m))} \Ll\vert \nabla V_m(F \cup G,j-\ell \vert \cu)(e') \Rr\vert \Ind{F_\circ = \emptyset} \\
				&\quad +2^{\vert F_\circ \vert} \sum_{e' \in \EE(\cCp^F(\rr \cu_{m,0}(e) \cap \cu_m))} |\na V_m(F_*,j \vert \cu)(e')|  \Ind{F_\circ \subset \Ed(\rr  \cu_{m,0}(e))}\Ind{F_\circ \neq \emptyset}.
			\end{split}
		\end{align}
		Notice that, for different $0 \leq h \leq k-2$, the definition above yields different  coarse-controlled functions $\g_m(F,j \vert h)$, but they all satisfy the definition \eqref{eq.vg} for $v = V_m(F,j)$.
		
		We apply Proposition~\ref{prop.weighted_L2_NEW} and $\g_m(F,j \vert h+1)$ to the perturbed corrector equation \eqref{eq.recurrence_eq}
		\begin{multline}\label{eq.VWBound}
			\sum_{\substack{\cu \in \pPN_{h}(\cu_m)}} |\cu|^{\alpha}\sum_{e\in \EE\Ll(\cC_*^\#(\rr\cu)\Rr)} |\na V_m(F, j)(e)|^2\\
			\leq C\sum_{\substack{\cu \in \pPN_{h+1}(\cu_m)}}|\cu|^3 \sum_{e\in \Ed(\rr\cu)} | \W_m(F, j)(e)|^2 + \vert\g_m(F,j \vert h+1)(e)\vert^2.
		\end{multline} 	
		The part concerning the controlled function $\g_m(F,j \vert h+1)$ can be bounded by \eqref{eq.coarse_grained_VFj}, because the coarse-graining occurs in cubes of $\pPN_{0}(\cu_m)$ (see its definition \eqref{eq.defgm}), which is smaller than that in $\pPN_{h+1}(\cu_m)$. Thus,  we obtain that 
		\begin{multline}\label{eq.gBound}
			\sum_{\substack{|F|=i\\F\subset \Ed(\cu_m)}}\sum_{\substack{\cu \in \pPN_{h+1}(\cu_m)}}|\cu|^3 \sum_{e\in \Ed(\rr\cu)} \vert\g_m(F,j \vert h+1)(e)\vert^2 \\
			\leq C\sum_{\substack{0 \leq \ell' \leq i, 0 \leq \ell \leq j\\ (\ell',\ell) \neq (0,0)}}\tilde{I}^{\, (N,k)}_{m, \ell+ \ell' + 5, h+2}(i-\ell'+\ell,j-\ell).
		\end{multline}
		Following the notation \eqref{eq.defIalpha}, we obtain that 
		\begin{multline}\label{eq.Ibound}
			I^{\,(N,k)}_{m, \alpha, h}(i,j) \leq  \frac{C}{\vert \cu_m \vert}\sum_{\substack{|F|=i\\F\subset \Ed(\cu_m)}} \sum_{\substack{\cu \in \pPN_{h+1}(\cu_m)}}|\cu|^3 \sum_{e\in \Ed(\rr\cu)} |\W_m(F, j)(e)|^2 \\
			+ C\sum_{\substack{0 \leq \ell' \leq i, 0 \leq \ell \leq j\\ (\ell',\ell) \neq (0,0)}}\tilde{I}^{\, (N,k)}_{m, \ell+ \ell' + 5, h+2}(i-\ell'+\ell,j-\ell).
		\end{multline}
		
		\textit{Step~2: clarify the recurrence.} It remains to make $\W_m(F,j)$ more explicit. Recall the perturbed corrector equation \eqref{e:induction_eqW}, for any $e\in \Ed(\cu_m)$ we have
		\begin{align*}
			\W_m(F,j) &= \sum_{e\in \Ed(\cu_m) \setminus F}(\aio{e}-\a)\Big(\na V_m(F,j-1) - \na V_m(F\cup \{e\}, j-2) \\
			\nonumber & \qquad \qquad \qquad \qquad + \na V_m(F\cup \{e\}, j-1)\Big) \\
			\nonumber &\qquad +\sum_{e\in F}(\aio{e}-\a)\Big(\na V_m(F\setminus\{e\},j)-\na V_m(F,j-1)\Big).
		\end{align*}
		which implies that 
		\begin{equation}\label{eq.Wbound}
			\begin{split}
				\vert \W_m(F,j)(e) \vert &\leq \Ind{e \in \Ed(\cu_m) \setminus F}\Big(\vert\na V_m(F,j-1)\vert + \vert\na V_m(F\cup \{e\}, j-2)\vert  \\
				& \qquad \qquad \qquad \qquad + \vert \na V_m(F\cup \{e\}, j-1) \vert \Big)(e) \\
				& \qquad + \Ind{e \in F}\Big(\vert \na V_m(F\setminus\{e\},j) \vert + \vert\na V_m(F,j-1)\vert\Big)(e).
			\end{split}
		\end{equation}
		This is because in the percolation $\aio{e'}(e) = \a(e)$ whenever $e' \neq e$, and only one term in the sum really contributes. Then the last sum in \eqref{eq.Ibound} becomes
		\begin{equation}\label{eq.Wbound2}
			\begin{split}
				&\sum_{\substack{|F|=i\\F\subset \Ed(\cu_m)}} |\W_m(F,j)(e)|^2 \leq 5\sum_{\substack{|F|=i\\F\subset \Ed(\cu_m)\setminus\{e\}}} \Big[|\na V_m(F,j-1)(e)|^2 + |\na V_m(F\cup \{e\},j-2)(e)|^2 \\
				&\qquad \qquad \qquad \qquad \qquad + |\na V_m(F\cup \{e\},j-1)(e)|^2\Big] \\
				& \qquad \qquad \qquad \qquad \qquad + 5\sum_{\substack{|F|=i\\e\in F\subset \Ed(\cu_m)}}\Big[|\na V_m(F\setminus \{e\},j)(e)|^2 + |\na V_m(F,j-1)(e)|^2\Big].
			\end{split}
		\end{equation}
		We analyze these $5$ terms one by one.

		For the first term on the {\rhs} of \eqref{eq.Wbound2}, we relax the condition $F\subset \Ed(\cu_m)\setminus\{e\}$ and change the order of sum to give
		\begin{equation}\label{eq.Term1}
			\begin{split}
				& \vert \cu_m \vert^{-1} \sum_{\substack{\cu \in \pPN_{h+1}(\cu_m)}} |\cu|^3\sum_{e\in \Ed(\rr \cu)} \sum_{\substack{|F|=i\\F\subset \Ed(\cu_m)\setminus\{e\}}} |\na V_m(F,j-1)(e)|^2\\
				&\leq \vert \cu_m \vert^{-1}\sum_{\substack{\cu \in \pPN_{h+1}(\cu_m)}} |\cu|^3\sum_{e\in \Ed(\rr \cu)} \sum_{\substack{|F|=i\\F\subset \Ed(\cu_m)}} |\na V_m(F,j-1)(e)|^2\\
				& = I^{\,(N,k)}_{m, 3, h+1}(i,j-1).
			\end{split}
		\end{equation}
		
		The same reasoning works for the second term on the {\rhs} of \eqref{eq.Wbound2}, with a change of variable ${F' = F \cup \{e\}}$
		\begin{equation}\label{eq.Term2}
			\begin{split}
				&\vert \cu_m \vert^{-1} \sum_{\substack{\cu \in \pPN_{h+1}(\cu_m)}} |\cu|^3\sum_{e\in \Ed(\rr \cu)} \sum_{\substack{|F|=i\\F\subset \Ed(\cu_m)\setminus\{e\}}} |\na V_m(F\cup \{e\},j-2)(e)|^2\\
				&= \vert \cu_m \vert^{-1} \sum_{\substack{\cu \in \pPN_{h+1}(\cu_m)}} |\cu|^3\sum_{e\in \Ed(\rr \cu)} \sum_{\substack{|F'|=i+1\\e \in F'\subset \Ed(\cu_m)}} |\na V_m(F',j-2)(e)|^2\\
				&\leq \vert \cu_m \vert^{-1} \sum_{\substack{\cu \in \pPN_{h+1}(\cu_m)}} |\cu|^3\sum_{e\in \Ed(\rr \cu)} \sum_{\substack{|F'|=i+1\\  F'\subset \Ed(\cu_m)}} |\na V_m(F',j-2)(e)|^2 \\
				&= I^{\,(N,k)}_{m, 3, h+1}(i+1,j-2).
			\end{split}
		\end{equation}

		The third term on the {\rhs} of \eqref{eq.Wbound2} works like the second term with the only change from the index $(j-2)$ to $(j-1)$
		\begin{equation}\label{eq.Term3}
			\begin{split}
				&\vert \cu_m \vert^{-1} \sum_{\substack{\cu \in \pPN_{h+1}(\cu_m)}} |\cu|^3\sum_{e\in \Ed(\rr \cu)} \sum_{\substack{|F|=i\\F\subset \Ed(\cu_m)\setminus\{e\}}} |\na V_m(F\cup \{e\},j-1)(e)|^2\\
				&= \vert \cu_m \vert^{-1} \sum_{\substack{\cu \in \pPN_{h+1}(\cu_m)}} |\cu|^3\sum_{e\in \Ed(\rr \cu)} \sum_{\substack{|F'|=i+1\\e \in F'\subset \Ed(\cu_m)}} |\na V_m(F',j-1)(e)|^2\\
				&\leq I^{\,(N,k)}_{m, 3, h+1}(i+1,j-1).
			\end{split}
		\end{equation}
		
		For the fourth term on the {\rhs} of \eqref{eq.Wbound2}, we make the change of variable $F' = F \setminus \{e\}$
		\begin{equation}\label{eq.Term4}
			\begin{split}
				&\vert \cu_m \vert^{-1} \sum_{\substack{\cu \in \pPN_{h+1}(\cu_m)}} |\cu|^3\sum_{e\in \Ed(\rr \cu)} \sum_{\substack{|F|=i\\e \in F\subset \Ed(\cu_m)}} |\na V_m(F\setminus \{e\},j)(e)|^2\\
				&= \vert \cu_m \vert^{-1} \sum_{\substack{\cu \in \pPN_{h+1}(\cu_m)}} |\cu|^3\sum_{e\in \Ed(\rr \cu)} \sum_{\substack{|F'|=i-1\\ F'\subset \Ed(\cu_m) \setminus \{e\}}} |\na V_m(F',j)(e)|^2\\
				&\leq \vert \cu_m \vert^{-1} \sum_{\substack{\cu \in \pPN_{h+1}(\cu_m)}} |\cu|^3\sum_{e\in \Ed(\rr \cu)} \sum_{\substack{|F'|=i-1\\  F'\subset \Ed(\cu_m) }} |\na V_m(F',j)(e)|^2 \\
				&= I^{\,(N,k)}_{m, 3, h+1}(i-1,j).
			\end{split}
		\end{equation}
		
		For the fifth term on the {\rhs} of \eqref{eq.Wbound2}, we relax directly the condition $e \in F$
		\begin{equation}\label{eq.Term5}
			\begin{split}
				&\vert \cu_m \vert^{-1} \sum_{\substack{\cu \in \pPN_{h+1}(\cu_m)}} |\cu|^3\sum_{e\in \Ed(\rr \cu)} \sum_{\substack{|F|=i\\ e \in F\subset \Ed(\cu_m)}} |\na V_m(F,j-1)(e)|^2\\
				&\leq \vert \cu_m \vert^{-1} \sum_{\substack{\cu \in \pPN_{h+1}(\cu_m)}} |\cu|^3\sum_{e\in \Ed(\rr \cu)} \sum_{\substack{|F|=i\\F\subset \Ed(\cu_m)}} |\na V_m(F,j-1)(e)|^2\\
				& = I^{\,(N,k)}_{m, 3, h+1}(i,j-1).
			\end{split}
		\end{equation}
		
		Combing \eqref{eq.Term1}, \eqref{eq.Term2}, \eqref{eq.Term3}, \eqref{eq.Term4}, \eqref{eq.Term5} and \eqref{eq.Ibound}, we obtain the desired estimate \eqref{eq.Is_by_I}.
	\end{proof}

	\subsection{General recurrence of Proposition~\ref{pr:key_estimate}}\label{subsec.Inductionij2}
	Using the hole-separation trick, we will give a proof of alternative recurrence for quantities $I^{(N,k)}_{m,\alpha, h}(i,j), I^{* \,(N,k)}_{m,\alpha, h}(i,j), \tilde{I}^{ \,(N,k)}_{m,\alpha, h}(i,j)$ introduced in \eqref{eq.defIalpha}, \eqref{eq.defIalpha_cluster}, \eqref{eq.defIalpha_hole}. They are the weighted $\ell^1$-$L^2$ energy, its version on cluster,  and its version of hole separation.
	\begin{prop}\label{prop.3recurrence}
		Under the same condition as Proposition~\ref{pr:key_estimate}, there exist constants $C, C^*, \tilde{C} < +\infty$ depending on $i,j,\alpha,h,N,k,d,\rr$ satisfying
		\begin{align*}
			I^{* \;(N,k)}_{\; m,\alpha, h}(i,j) &\leq C^*(i,j,\alpha, h,N,k,d,\rr), \\
			\tilde{I}^{ \;(N,k)}_{\; m,\alpha, h}(i,j) &\leq \tilde{C}(i,j,\alpha,h,N,k,d,\rr), \\
			I^{\; (N,k)}_{\; m,\alpha, h}(i,j) &\leq C(i,j,\alpha,h,N,k,d,\rr). \\
		\end{align*}
	\end{prop}
	\begin{proof}
		The basis of proof for $(i,j) = (0,0)$ was already contained in Section~\ref{subsec.basis}, and we implement the following alternative recurrence to close the proof. We define the following partial order 
		\begin{align}\label{eq.partial_order}
			(i,j) \prec (i_0,j_0) \Longleftrightarrow i+j \leq i_0+j_0 \text{ and }  \Ll\{\begin{array}{ll}
				j < j_0, \\
				\text{ or } j=j', i < i'.
			\end{array}\Rr.  
		\end{align}
		We then use $I_m(i,j), I^*_m(i,j), \tilde{I}_m(i,j)$ respectively to indicate the three statements with admissible parameters. In Step~1-3 below, we give the argument of induction, and then discuss the technical conditions needed in Step~4.
		
		\textit{Step~1: $\{I_m(i,j), \tilde{I}_m(i,j)\}_{(i,j) \prec (i_0, j_0)}   \Longrightarrow 	I^{*}_{m}(i_0,j_0)$.} Suppose that we have established all these statements for $(i,j) \prec (i_0, j_0)$, then to obtain that for estimate on clusters $I^{*}_{m}(i_0,j_0)$,  we just apply the result from Proposition~\ref{prop.weighted_L2_NEW} (see also Lemma~\ref{lem.Is_by_I})
		\begin{multline}\label{eq.key_estimate_Step1}
			I^{*  \, (N,k)}_{m,\alpha, h}(i_0,j_0) 
			\leq  C\Big(I^{\,(N,k)}_{m, 3, h+1}(i_0,j_0-1)+ I^{\,(N,k)}_{m, 3, h+1}(i_0+1,j_0-2)  \\
			\qquad + I^{\,(N,k)}_{m, 3, h+1}(i_0+1,j_0-1)+I^{\,(N,k)}_{m, 3, h+1}(i_0-1,j_0)\Big) \\
			+ C \sum_{\substack{0 \leq \ell' \leq i, 0 \leq \ell \leq j\\ (\ell',\ell) \neq (0,0)}}\tilde{I}^{\, (N,k)}_{m, \ell+ \ell' + 5, h+2}(i_0-\ell'+\ell,j_0-\ell).
		\end{multline}
		where the {\rhs} only involves the terms of lower order.
		
		\smallskip
		
		\textit{Step~2: $\{\tilde{I}_{m}(i,j)\}_{(i,j) \prec (i_0, j_0)} + I^{*}_{m}(i_0,j_0)    \Longrightarrow \tilde{I}_{m}(i_0,j_0)$.} 
		Once we have $I^{*}_{m}(i_0,j_0)$, we recall the coarse-graining estimate \eqref{eq.coarse_grained_VFj} in Lemma~\ref{lem.coarse_grained_VFj}, and yield the estimate for the version of hole-separation
		\begin{multline}\label{eq.key_estimate_Step2}
			\tilde{I}^{\, (N,k)}_{m,\alpha, h}(i_0,j_0) \leq C I^{*  \, (N,k)}_{m,\alpha+1, h}(i_0,j_0)\\
			+ C\sum_{\substack{0 \leq \ell' \leq i_0, 0 \leq \ell \leq j_0\\ (\ell',\ell) \neq (0,0)}}\tilde{I}^{\, (N,k)}_{m,\alpha+\ell+ \ell' + 2, h+1}(i_0-\ell'+\ell,j_0-\ell).
		\end{multline}
		
		\smallskip
		
		\textit{Step~3: $\{\tilde{I}_{m}(i,j)\}_{(i,j) \prec (i_0, j_0)} + \tilde{I}_{m}(i_0,j_0)  \Longrightarrow I_{m}(i_0,j_0)$.} The statement of $\tilde{I}_{m}(i_0,j_0)$ and its version of lower orders yield that of $I_{m}(i_0,j_0)$ thanks to the hole-separation trick in  Lemma~\ref{lem.I_by_It}
		\begin{align}\label{eq.key_estimate_Step3}
			I^{\,(N,k)}_{m,\alpha, h}( i_0,j_0)  \leq (j+1)\sum_{\ell=0}^j \tilde{I}^{\, (N,k)}_{m, \alpha+\ell+1, h+1}(i_0+\ell,j_0-\ell).
		\end{align}
		
		\smallskip
		
		\textit{Step~4: conditions for the parameter.} We discuss a sufficient condition about the parameters to complete the induction. We fix $(\alpha, i, j)$ in the following paragraph, then the choice for $N$ and $\rr$ is quite natural.
		\begin{itemize}[label=---]
			\item $\rr \in (\frac{50}{27}, 2)$: the geometric condition  is required when using Proposition~\ref{prop.good_cube_property} and Lemma~\ref{lem.center_cube}.
			\item $N \geq i+j$: We need to keep the $N$-$\stable$ properties of the partition cubes. Notice that at most $(i+j)$-th Glauber derivative is applied in the induction.
		\end{itemize}
		
		The conditions for $h,k$ are more involved. We define the degree like \eqref{eq.defdegV}
		\begin{align*}
			\deg(I^{*}_{m}(i,j)) = \deg(\tilde{I}_{m}(i,j)) = \deg(I_m(i,j)) := i+2j.
		\end{align*}
		When we run Step~1-3 for one round, every term will decrease at least by $1$, so we need at most $(i+2j)$ rounds. Meanwhile, in  $(i+2j)$ rounds, the scale $h$ and the moment $\alpha$ will increase.
		\begin{itemize}[label=---]
			\item The partition cube coarsens at most from $h$ to $(h+2)$ in \eqref{eq.key_estimate_Step1}-\eqref{eq.key_estimate_Step3}, so the largest  scale can be
			\begin{align}\label{eq.key_estimate_coarseness}
				h + 3 \cdot	(i+2j) \cdot  2 \leq h + 12 (i+j).
			\end{align}
			\item The moment has an increment at most from $\alpha$ to $\alpha + \ell+ \ell' + 5$ in \eqref{eq.key_estimate_Step1}-\eqref{eq.key_estimate_Step3}, with the constrain $\ell \leq j$ and $\ell' \leq i+j$ there. Therefore, the largest moment can be 
			\begin{align}\label{eq.key_estimate_moment}
				\alpha + 3 \cdot (i+2j)  \cdot(i+2j+5) \leq \alpha + 12 (i+j+2)^2.
			\end{align}
		\end{itemize}
		Because we need to apply Proposition~\ref{prop.weighted_L2_NEW}, then \eqref{eq.key_estimate_moment} suggests $h \geq h_*(\alpha,i,j)$ as a minimal scale to ensure \eqref{eq.condition_hk_NEW} throughout the induction. Afterwards, by \eqref{eq.key_estimate_coarseness}, we need $k \geq h + 12(i+j) + 2$ to finish all the renormalization needed.  
		
		Finally, we notice that the weighted sum in Proposition~\ref{prop.3recurrence}  like $I^{\; (N,k)}_{\; m,\alpha, h}(i,j)$ with ${h < h_*(\alpha,i,j)}$ can be naturally dominated  by the case $h = h_*(\alpha,i,j)$. This concludes the choice of parameters in \eqref{eq.condition_Nk_main}.

	\end{proof}

	\begin{proof}[Proof of Proposition~\ref{pr:key_estimate}  with $j \geq 1$]
		This is a direct result of Proposition~\ref{prop.3recurrence}.
	\end{proof}

	\section{Conclusion}\label{sec.Pf}
	
	The last section is devoted to the proof of the main theorem.
	\subsection{Annealed improved $\ell^1$-$L^2$ estimate}
	The first result in this section is an annealed version of Proposition~\ref{pr:key_estimate}, thus we finally get the desired estimate of the improved energy  $\bar{I}_m(i,j)$ defined in \eqref{eq.improvedEnergy}.
	\begin{prop}[Annealed improved $\ell^1$-$L^2$ estimate]\label{pr:key_estimate_avg}
		Given $i,j \in \N$, there exist constants $\{C_{i,j}(\p,\rr, d)\}_{i,j \in \N}$ independent of $m \in \N_+$ such that the following statement holds for $\bar{I}_m(i,j)$ that
		\begin{equation}
			\label{e:key_estimate2}
			\E_\p\Ll[\frac{1}{\vert \cu_m \vert} \sum_{F \subset \Ed(\cu_m)} \sum_{e \in \Ed(\cu_m)} \vert \nabla V_m(F,j)(e)\vert^2  \Rr] \leq C_{i,j}(\p,\rr, d).
		\end{equation}           
		The constant $C_{i,j}(\p,\rr, d)$ is locally uniformly with respect to $\p$.
	\end{prop}

	\begin{proof}
		The main part of the proof is still the quenched version. We set $\alpha = 0$ in Proposition~\ref{pr:key_estimate} and then pick integers 
		\begin{align*}
			N = i+j, \qquad	h = h_*(0,i,j), \qquad k= h_*(0,i,j) + 12(i+j) + 2,
		\end{align*}
		such that the quenched version \eqref{e:key_estimate} holds. We decompose the estimate of \eqref{e:key_estimate2} into two parts, the good case $\{ \cu_m \in \gGN_k\}$ and the bad case $\{\cu_m \notin \gGN_k\}$, then treat them respectively. 
		
		\smallskip
		\textit{Step~1: bad case $\{\cu_m \notin \gGN_k\}$. } We observe the following trivial bound, which comes directly from \eqref{eq.V_F_j}, \eqref{eq.Difference} and maximum principle
		\begin{align}\label{eq.TrivialBoundV}
			\forall e \in \Ed(\cu_m), \qquad \vert V_m(F,j)\vert \leq 2^j \vert \cu_m \vert^j  3^{m}.
		\end{align}
		We apply directly this trivial bound of $V_m(F,j)$ together with the tail estimate of probability
		\begin{equation}\label{eq.rareBound}
			\begin{split}
				&\E_\p\Ll[\frac{1}{|\cu_m|}\sum_{\substack{|F|=i\\F\subset \Ed(\cu_m)}}  \sum_{e \in \Ed(\cu_m)}  |\na V_m(F, j)(e)|^2 \Ind{\cu_m \notin  \gGN_k}\Rr]\\
				&\leq 4^j \vert \cu_m \vert^{i+2j} 3^{2m} \P_\p[\cu \notin \gGN_{k}] \\
				&\leq 4^j 3^{m(2+d(i+2j))} K_k \exp(-K_k 3^{m s_k}) \\
				&\leq C'(i,j).
			\end{split}
		\end{equation}
		Here we recall the good cube $\gGN_{k}$ satisfies the tail probability bound \eqref{e:rarely_bad}, as indicated by (2) Proposition~\ref{prop.good_cube_property}. This gives a uniform bound with respect to $m \in \N_+$, because the tail probability overcomes the contribution of the sum. 
		
		\smallskip
		\textit{Step~2: good case $\{\cu_m \in \gGN_k\}$. } 
		For this part, the quenched version \eqref{e:key_estimate} applies 
		\begin{align}\label{eq.contriQuenched}
			\frac{1}{|\cu_m|}\sum_{\substack{|F|=i\\F\subset \Ed(\cu_m)}} \sum_{e\in \Ed(\cu_m)} |\na V_m(F, j)(e)|^2 \Ind{ \cu_m \in \gGN_k} \leq I^{\; (N,k)}_{\; m, 0, h}(i,j) \leq C.
		\end{align}
		so its expectation is also uniformly bounded with respect to $m \in \N_+$. 
		
		\smallskip
		We combine \eqref{eq.rareBound} and \eqref{eq.contriQuenched}, and conclude \eqref{e:key_estimate2}. The upper bound is locally uniform because the connectivity is monotone in function of $\p$.

	\end{proof}
	
	\subsection{Proof of main theorems}
	We recall the following classical result on the uniform convergence and derivatives.

	\begin{lemma}\label{lem.conceptual}
		Let $\nu$ be a quantity defined on an open interval $I$. If $(\nu_L)_{L \in \N_+}$ satisfy conditions (1)(2)(3), then $\nu$ is also smooth on $I$ and the derivatives also converge
		\begin{align}\label{eq.conceptual}
			\forall \rho \in I, \qquad \nu^{(k)}(\rho) = \lim_{L \to \infty} \nu_L^{(k)}(\rho).
		\end{align}
		\begin{enumerate}
			\item \textit{Approximation}: $\nu_L$ converges pointwisely to $\nu$ on $I$.
			\item \textit{Smoothess}: $\nu_L$ is smooth on $I$ for every $L \in \N_+$.
			\item \textit{Locally bounded derivatives}: for any $k \in \N$, the mapping $\rho \mapsto \sup_{L \in \N_+} \vert \nu_L^{(k)}(\rho) \vert$
			is locally uniform.
		\end{enumerate}

		Moreover, if we replace the condition (1) with a quantitative version:  
		\begin{enumerate}
			\item[(1+)] there exist two finite positive functions $C(\rho), \alpha(\rho)$ locally uniform on $\rho$ such that
			\begin{align*}
				\forall L \in \N_+, \qquad \vert \nu_L(\rho)  - \nu(\rho)\vert \leq C L^{-\alpha};
			\end{align*}
		\end{enumerate}
		then we can also improve \eqref{eq.conceptual}: for any $\alpha' \in (0,\alpha)$ on $I$, there exists a finite positive constant $C(\rho, \alpha')$ such that  
		\begin{align}\label{eq.conceptualPlus}
			\forall L \in \N_+, \qquad \vert \nu_L^{(k)}(\rho)  - \nu^{(k)}(\rho)\vert \leq C'(\rho, \alpha') L^{- \alpha'}.
		\end{align}
	\end{lemma}
	\begin{proof}
		We prove \eqref{eq.conceptual} by induction. The basis is ensured by condition (1). Suppose now the result is established until  $k$-th derivative, and aim to prove the $(k+1)$-th derivative.  Viewing (2), we have Taylor's expansion for $\nu_L$ that for all $\rho_0, \rho \in I$, thus there exists $\theta \in (0,1)$  such that
		\begin{align*}
			\nu^{(k)}_L(\rho) = \nu^{(k)}_L(\rho_0) + \nu^{(k+1)}_L(\rho_0)(\rho - \rho_0) + \frac{1}{2} \nu^{(k+2)}_L(\rho_0 + \theta(\rho-\rho_0))(\rho - \rho_0)^2.
		\end{align*}
		Now applying the condition (3) to $\nu^{(k+2)}_L$. There exists $C_{k+2}(\rho_0) < \infty$, such that the following estimate holds for $\rho$ in a small neighborhood of $\rho_0$
		\begin{align*}
			\forall L \in \N_+, \qquad \Ll\vert \nu^{(k)}_L(\rho) - \nu^{(k)}_L(\rho_0) - \nu^{(k+1)}_L(\rho_0)(\rho - \rho_0) \Rr\vert \leq C_{k+2}(\rho_0)(\rho - \rho_0)^{2}.
		\end{align*}
		We pass $L$ to infinity for the equation above, then by induction
		\begin{align*}
			\lim_{L \to \infty} \nu_L^{(k)}(\rho) = \nu^{(k)}(\rho), \qquad \lim_{L \to \infty} \nu_L^{(k)}(\rho_0) = \nu^{(k)}(\rho_0),
		\end{align*}
		and $\nu_L^{(k+1)}(\rho_0)$ also converges to a limit $\nu^{(k+1), *}(\rho_0)$ along a subsequence thanks to the uniform bound (3). This gives us
		\begin{align}\label{eq.kdiff}
			\Ll\vert \nu^{(k)}(\rho) - \nu^{(k)}(\rho_0) - \nu^{(k+1),*}(\rho_0)(\rho - \rho_0) \Rr\vert \leq C_{k+2}(\rho_0)(\rho - \rho_0)^{2},
		\end{align}
		which implies that $\nu$ is $(k+1)$-th differentiable at $\rho_0$. 
		\begin{align}\label{eq.induction_derivative}
			\lim_{\rho \to \rho_0}\Ll\vert \frac{\nu^{(k)}(\rho) - \nu^{(k)}(\rho_0)}{\rho - \rho_0} - \nu^{(k+1),*}(\rho_0) \Rr\vert \leq \lim_{\rho \to \rho_0} C_{k+2}(\rho_0)(\rho - \rho_0) = 0.
		\end{align}
		This also gives a characterization of $\nu^{(k+1), *}(\rho_0) = \nu^{(k+1)}(\rho_0)$, so it is also the limit of the whole sequence $\lim_{L \to \infty} \nu_L^{(k+1)}(\rho)$:  for any subsequence of $\nu_L$, when extracting a further convergent sequence of $(k+1)$-th derivative, \eqref{eq.induction_derivative} is always satisfied. Thus, the limit has no choice but $\nu^{(k+1)}(\rho_0)$, and this justifies \eqref{eq.conceptual}.    
		
		Concerning the convergence rate, a quick approach is Landau--Kolmogorov inequality \cite{landau1914, kolmogoroff1949}: for any $\rho \in I$, we pick a neighborhood $J$ such that $\rho \in J \subset I$, then 
		\begin{align*}
			\forall 1 \leq k <n, \qquad \norm{(\nu_L - \nu)^{(k)}}_{L^{\infty}(J)} \leq C(k,n,J) \norm{\nu_L - \nu}^{1-\frac{k}{n}}_{L^{\infty}(J)} \norm{(\nu_L - \nu)^{(n)}}^{\frac{k}{n}}_{L^{\infty}(J)}.
		\end{align*} 
		Then we can obtain the convergence rate using the condition (1+), with a very large $n$ to be chosen.
		
	\end{proof}

	\begin{proof}[Proof of Theorem~\ref{thm.main} and Theorem~\ref{thm.ab}]
		It suffices to verify the 3 conditions in Lemma~\ref{lem.conceptual} for percolation model. As mentioned in \eqref{eq.Einstein} (see also \cite[eq.(181)]{dario2021quantitative}), we only need to prove the result for $\ab$, and it has a finite-volume approximation $\ab_m$ defined in \eqref{eq.ab}.
		\begin{itemize}
			\item The condition (1) and (1+) is ensured by \cite[Proposition~5.2]{armstrong2018elliptic}.
			\item The condition (2) is verified in Proposition~\ref{pr:expansion_init}.
			\item The condition (3) is the heart of this paper, which relies on a chaos expansion \eqref{e:derivative_formula}, an upper bound by improved $\ell^1$-$L^2$ energy in Lemma~\ref{lem.abk_VFj}, and the perturbed corrector equations \eqref{eq.recurrence_eq}. Then the techniques including the renormalizaiton, the coarse-graining, the cluster-growth decomposition, and the hole separation are developed in Section~4-7. Finally, the condition (3) is proved in Proposition~\ref{pr:key_estimate_avg}.
		\end{itemize}
	\end{proof}

	\appendix
	\section{Index of notation}\label{sec.notation_list}
	We list the notations commonly used in the paper, with their meaning briefly recalled and the position they are defined. 
	\begin{multicols}{2}
		\begin{itemize}[label=--]
			\setlength\itemsep{0.9em}
			\item $\a$: the state of percolation; \eqref{eq.def_Percolation}.
			
			\item $\theta(\p)$: the connected probability; \eqref{def.theta}.
			
			\item $\ab(\p)$: the conductivity; \eqref{eq.ab}.
			
			\item $\ab_m^{(k)}(\p)$: the $k$-th derivative of approximated conductivity in $\cu_m$; \eqref{e:derivative_formula}.
			
			\item $\a^G$: the percolation when opening bonds in $G$; \eqref{eq.defaG}.
			
			\item $f^G$: the function taking $\a^G$ as variable; \eqref{eq.f_convention}.
			
			\item $D^G$: the Glauber derivative; \eqref{e.def.DE}.
			
			\item $\EE$: the intrinsic geometry; \eqref{eq.def_intrinsic}.
			
			\item $\a^\#, f^\#, \cC^\#, \OO^\#$: percolation of the environment $\#$, and the random variables (function, cluster, hole) taking $\a^\#$ as the sample; \eqref{eq.sharp}.
			
			\item $\cu_m$: the cube of size $3^m$; \eqref{eq.defCube}.
			
			\item $\cCf$: the infinite cluster; Section~\ref{subsec.perco}.
			
			\item $\OO(x)$: the hole containing $x$; Section~\ref{subsec.perco}.
			
			\item $\cC_*(\cu)$: the maximal cluster in $\cu$; Definition~\ref{def:well_connected}.
			
			\item $\cCb(\cu)$: the boundary-connecting clusters of $\cu$; \eqref{eq.def_Cbc}.
			
			\item $\cCp(c \cu \cap \cu_m)$: the clusters in ${c \cu \cap \cu_m}$ connecting to the boundary $\cu_m$; \eqref{eq.def_clt_boundary}.

			\item $\mathscr{E}_{m,\xi}$: the minimum Dirichlet energy in $\cu_m$ with boundary condition $\ell_{\xi}$; \eqref{eq.defmu}.
			
			\item $v_{m,\xi}$: the minimiser to attain $\mathscr{E}_{m,\xi}$ defined in \eqref{eq.harmonicBC}. It is assumed to follow the constant extension rule \eqref{eq.harmonicCube} in the paper, and then to follow \eqref{eq.decom_caonical_easy_1} starting from Section~\ref{sec.HarmonicExtension}.
			
			\item $V_{m,\xi}(F,j)$: $=\sum_{\substack{|G|=j\\G\subset \Ed(\cu_m) \setminus F}} D_{F\cup G} v_{m,\xi}$.
			
			\item $v_{m}, V_{m}(F,j)$: shorthand notation respectively for  $v_{m,\xi}, V_{m,\xi}(F,j)$.
			
			\item $\sSl_{\G}(\cu)$: local partition of good cubes; Proposition~\ref{prop.partition}.
			
			\item $\Lambda(\G,t,C)$: the good cubes with $t$-moment for its partition; Definition~\ref{def.LambdaG}.
			
			\item $\gGN_{k}$: the good cube of scale $k$ and $N$-stability; Proposition~\ref{prop.good_cube_property}.
			
			\item $\pPN_{h}$: the pyramid partition; \eqref{e.pyrmid}
			
			\item $N$-$\stable$: the stability under perturbation of $N$ bonds; Definition~\ref{def.N_stable}.
			
			\item $[x]^\#$: the canonical grain; \eqref{eq.def_center}.
			
			\item $F_*, F_\circ$: the disjoint parts of $F$ in cluster-growth decomposition; \eqref{eq.def_WIP_new}.
			
			\item $\nabla_G$: spatial finite difference; \eqref{eq.def_diff_shift}.

			\item $I^{\; (N,k)}_{\; m,\alpha, h}(i,j)$: weighted $\ell^1$-$L^2$ energy; \eqref{eq.defIalpha}.
			
			\item $I^{* \;(N,k)}_{\; m,\alpha, h}(i,j)$: weighted $\ell^1$-$L^2$ energy on clusters; \eqref{eq.defIalpha_cluster}.
			
			\item $\tilde{I}^{ \;(N,k)}_{\; m,\alpha, h}(i,j)$: weighted $\ell^1$-$L^2$ energy via hole separation; \eqref{eq.defIalpha_hole}.
			
			\item $V_m(F,j \vert \cu)$: hole separation; \eqref{eq.V_F_j_Generalized_cube}.
			
			\item $\rr$: a constant to slightly enlarge the cube, of value $(\frac{4}{3},2)$ by default; Definition~\ref{def.Meyers}. In Section~\ref{sec.KeyEstimate1} and \ref{sec.KeyEstimate2}, it is usually assumed to be of value $(\frac{50}{27},2)$.
		\end{itemize}
	\end{multicols}

	\begin{center}
		ACKNOWLEDGEMENTS 
	\end{center}
	\medskip
	Gu was supported by National Key R\&D Program of China (No.2023YFA1010400) and NSFC  (No.12301166). The two authors also thank TSIMF, IASM-BIRS, and NYU Shanghai for hospitality, where parts of this project were developed during the workshops there.

	\bibliographystyle{abbrv}
	\bibliography{Ref}
	
\end{document}